\documentclass[a4paper,leqno,11pt,twoside]{amsart}

\usepackage[left=1.0in,right=1.0in]{geometry}
\usepackage{times}
\usepackage[all]{xy}

\usepackage{amsmath, amssymb, amsfonts, latexsym, mdwlist, amsthm}
\usepackage{graphicx}
\usepackage{wrapfig}
\usepackage{longtable}
\usepackage[colorlinks=true,citecolor=blue, urlcolor=blue, linkcolor=blue, pagebackref]{hyperref}
\usepackage{caption, subcaption}
\usepackage{enumerate}
\usepackage{pgf,tikz}
\usetikzlibrary{arrows,chains,
 decorations.pathreplacing,
shapes,%
}

\definecolor{uququq}{rgb}{0.25,0.25,0.25}
\definecolor{qqwuqq}{rgb}{0,0.39,0}
\definecolor{zzttqq}{rgb}{0.6,0.2,0}
\definecolor{tttttt}{rgb}{0.2,0.2,0.2}
\definecolor{qqqqff}{rgb}{0,0,1}

\def\C{\ensuremath{\mathbb{C}}}

\def\P{\ensuremath{\mathbb{P}}}
\def\Q{\ensuremath{\mathbb{Q}}}
\def\R{\ensuremath{\mathbb{R}}}
\def\Z{\ensuremath{\mathbb{Z}}}

\def\cA{\ensuremath{\mathcal A}}
\def\AA{\ensuremath{\mathcal A}}
\def\cB{\ensuremath{\mathcal B}}
\def\BB{\ensuremath{\mathcal B}}

\def\cD{\ensuremath{\mathcal D}}
\def\DD{\ensuremath{\mathcal D}}
\def\cE{\ensuremath{\mathcal E}}
\def\EE{\ensuremath{\mathcal E}}
\def\cF{\ensuremath{\mathcal F}}
\def\FF{\ensuremath{\mathcal F}}
\def\cG{\ensuremath{\mathcal G}}
\def\GG{\ensuremath{\mathcal G}}

\def\LL{\ensuremath{\mathcal L}}

\def\OO{\ensuremath{\mathcal O}}
\def\cO{\ensuremath{\mathcal O}}
\def\PP{\ensuremath{\mathcal P}}

\def\cT{\ensuremath{\mathcal T}}
\def\TT{\ensuremath{\mathcal T}}

\def\VV{\ensuremath{\mathcal V}}
\def\cW{\ensuremath{\mathcal W}}

\def\vv{\ensuremath{\mathbf v}}

\def\llambda{\ensuremath{\boldsymbol{\lambda}}} 

\def\ch{\mathop{\mathrm{ch}}\nolimits}
\def\Char{\mathop{\mathrm{char}}\nolimits}
\def\Coh{\mathop{\mathrm{Coh}}\nolimits}

\def\Db{\mathop{\mathrm{D}^{\mathrm{b}}}\nolimits}
\def\deg{\mathop{\mathrm{deg}}\nolimits}

\def\dim{\mathop{\mathrm{dim}}\nolimits}

\def\ext{\mathop{\mathrm{ext}}\nolimits}
\def\Ext{\mathop{\mathrm{Ext}}\nolimits}
\def\lExt{\mathop{\mathcal Ext}\nolimits}

\def\Forg{\mathop{{\mathrm{Forg}}}\nolimits}
\def\GL{\mathop{\mathrm{GL}}\nolimits}

\def\Hom{\mathop{\mathrm{Hom}}\nolimits}
\def\lHom{\mathop{\mathcal Hom}\nolimits}
\def\RlHom{\mathop{\mathbf{R}\mathcal Hom}\nolimits}
\def\RHom{\mathop{\mathbf{R}\mathrm{Hom}}\nolimits}
\def\id{\mathop{\mathrm{id}}\nolimits}

\def\Ker{\mathop{\mathrm{Ker}}\nolimits}

\def\min{\mathop{\mathrm{min}}\nolimits}

\def\num{\mathop{\mathrm{num}}\nolimits}

\def\NS{\mathop{\mathrm{NS}}\nolimits}

\def\rk{\mathop{\mathrm{rk}}}

\def\Sym{\mathop{\mathrm{Sym}}}

\def\Stab{\mathop{\mathrm{Stab}}\nolimits}

\def\cH{\mathrm{H}}

\def\Forg{\mathrm{Forg}}
\def\Knum{K_{\num}}

\def\Ku{\mathcal{K}u}

\def\wPP{\ensuremath{\widetilde{\mathbb{P}}}}
\def\wX{\ensuremath{\widetilde{X}}}
\def\wY{\ensuremath{\widetilde{Y}}}
\def\wpsi{\ensuremath{\widetilde{\psi}}}

\newcommand\stv[2]{\left\{#1\,\colon\,#2\right\}}
\def\abs#1{\left\lvert#1\right\rvert}

\DeclareRobustCommand\longtwoheadrightarrow
     {\relbar\joinrel\twoheadrightarrow}

\newcommand{\cat}[1]{\begin{bf}#1\end{bf}}
\DeclareMathOperator{\Rmut}{\cat{R}}
\DeclareMathOperator{\Lmut}{\cat{L}}

\newcommand{\set}[1]{\left\{#1\right\}}
\newcommand{\gen}[1]{\left\langle#1\right\rangle}
\newcommand{\res}[2]{\left.#1\right|_{#2}} 

\def\blank{\underline{\hphantom{A}}}

\def\into{\ensuremath{\hookrightarrow}}
\def\onto{\ensuremath{\twoheadrightarrow}}
\def\iso{\ensuremath{\simeq}}

\newcommand{\mapto}[1]{\xrightarrow{#1}}

\usepackage{todonotes}

\makeatletter
\newtheorem*{rep@theorem}{\rep@title}
\newcommand{\newreptheorem}[2]{%
\newenvironment{rep#1}[1]{%
 \def\rep@title{#2~\ref{##1}}%
 \begin{rep@theorem}}%
 {\end{rep@theorem}}}
\makeatother

\newtheorem{Thm}{Theorem}[section]
\newreptheorem{Thm}{Theorem}
\newtheorem{Prop}[Thm]{Proposition}
\newtheorem{PropDef}[Thm]{Proposition and Definition}
\newtheorem{Lem}[Thm]{Lemma}
\newtheorem{Cor}[Thm]{Corollary}

\newtheorem*{Ques*}{Question}

\theoremstyle{definition}
\newtheorem{Def}[Thm]{Definition}
\newtheorem{Rem}[Thm]{Remark}

\newtheorem{Ex}[Thm]{Example}

\makeatletter
\newcommand\footnoteref[1]{\protected@xdef\@thefnmark{\ref{#1}}\@footnotemark}
\makeatother

\begin{document}

\title{Stability conditions on Kuznetsov components}

\author[A.~Bayer, M.~Lahoz, E.~Macr\`i, and P.~Stellari]{Arend Bayer, Mart\'i Lahoz, Emanuele Macr\`i, and Paolo Stellari}

\address{A.B.: School of Mathematics and Maxwell Institute,
University of Edinburgh,
James Clerk Maxwell Building,
Peter Guthrie Tait Road, Edinburgh, EH9 3FD,
United Kingdom}
\email{arend.bayer@ed.ac.uk}
\urladdr{\url{http://www.maths.ed.ac.uk/~abayer/}}

\address{M.L.: Universit\'{e} Paris Diderot -- Paris~7, B\^{a}timent Sophie Germain, Case 7012, 75205 Paris Cedex 13, France}
\email{marti.lahoz@ub.edu}
\urladdr{\url{http://www.ub.edu/geomap/lahoz/}}
\curraddr{Departament de Matem\`atiques i Inform\`atica,
Universitat de Barcelona, Gran Via de les Corts Catalanes, 585, 08007 Barcelona, Spain}

\address{E.M.: Department of Mathematics, Northeastern University, 360 Huntington Avenue, Boston, MA 02115, USA}
\email{emanuele.macri@universite-paris-saclay.fr}
\urladdr{\url{https://www.imo.universite-paris-saclay.fr/~macri/}}
\curraddr{Universit\'e Paris-Saclay, CNRS, Laboratoire de Math\'ematiques d'Orsay, Rue Michel Magat, B\^at. 307, 91405 Orsay, France}

\address{P.S.: Dipartimento di Matematica ``F.~Enriques'', Universit{\`a} degli Studi di Milano, Via Cesare Saldini 50, 20133 Milano, Italy}
\email{paolo.stellari@unimi.it}
\urladdr{\url{https://sites.unimi.it/stellari}}

\address{X.Z.: Department of Mathematics, Northeastern University, 360 Huntington Avenue, Boston, MA 02115, USA}
\email{xlzhao@ucsb.edu}
\urladdr{\url{https://sites.google.com/site/xiaoleizhaoswebsite/}}
\curraddr{Department of Mathematics, University of California, Santa Barbara, South Hall 6705, Santa Barbara, CA 93106, USA.}

\thanks{A.~B.~was supported by ERC-2013-StG-337039-WallXBirGeom, by  ERC-2018-CoG-819864-WallCrossAG, and by the NSF grant DMS-1440140 while the author was at the MSRI in Berkeley in the Spring 2019.
M.~L.~was supported by a Ram\'on y Cajal fellowship and 
partially by the research project PID2019-104047GB-I00 and the grant number 230986 of the Research Council of Norway.
E.~M.~was partially supported by the NSF grants DMS-1523496 and DMS-1700751, by a Poincar\'e Chair, by ERC-2020-SyG-854361-HyperK, and by GNSAGA-INdAM.
P.~S.~was partially supported by the grants ERC-2017-CoG-771507-StabCondEn, PRIN 2017 ``Moduli and Lie Theory'', and FARE 2018 HighCaSt (grant number R18YA3ESPJ)}



\begin{abstract}
We introduce a general method to induce Bridgeland stability conditions on semiorthogonal components of triangulated categories. 
In particular, we prove the existence of Bridgeland stability conditions on the Kuznetsov component of the derived category of Fano threefolds and of cubic fourfolds.
As an application, in the appendix, written jointly with Xiaolei Zhao, we give a variant of the proof of the Torelli theorem for cubic fourfolds by Huybrechts and Rennemo.

\end{abstract}

\maketitle
\setcounter{tocdepth}{1}
\tableofcontents


\section{Introduction}

\subsection*{Main results}
Let $X$ be a smooth Fano variety and let $\Db(X)$ denote its bounded derived category of coherent sheaves.
Let $E_1,\dots,E_m\in\Db(X)$ be an exceptional collection in $\Db(X)$. We call its right orthogonal complement
\begin{align*}
\Ku(X) &= \langle E_1, \dots, E_m \rangle^\perp\\
&=\left\{ C\in\Db(X)\,:\, \Hom(E_i,C[p])=0,\, \forall i=1,\dots,m,\, \forall p\in\Z \right\}
\end{align*}
a \emph{Kuznetsov component} of $X$.
In a series of papers, Kuznetsov has shown that much of the geometry of Fano varieties, and their moduli spaces, can be captured efficiently by $\Ku(X)$, for appropriate exceptional collections.

On the other hand, stability conditions on triangulated categories as introduced by Bridgeland in \cite{Bridgeland:Stab} and wall-crossing have turned out to be an extremely powerful tool for the study of moduli spaces of stable sheaves. We connect these two developments with the following two results:

\begin{Thm}\label{thm:main1}
Let $X$ be a Fano threefold of Picard rank $1$ over an algebraically closed field of characteristic either zero or sufficiently large.
Then the Kuznetsov semiorthogonal component $\Ku(X)$ has a Bridgeland stability condition.\footnote{In some case, we deduce Theorem~\ref{thm:main1} as an immediate consequence of an explicit description of the Kuznetsov component, which however is stated in the literature only for characteristic zero (and so it also holds if the characteristic of the base field is sufficiently large), see the tables in Section~\ref{sec:3foldsoverview}.}
\end{Thm}

The most interesting cases of Theorem~\ref{thm:main1} are Fano threefolds of index two, and those of index one and even genus; in these case, the theorem holds over any algebraically closed field, independently on the characteristic. We refer to Section~\ref{sec:3foldsoverview} for an overview of the classification of Fano threefolds of Picard rank one, and the exceptional collections appearing implicitly in Theorem~\ref{thm:main1}.

\begin{Thm}\label{thm:main2}
Let $X$ be a cubic fourfold over an algebraically closed field $k$ with $\Char k \neq 2$.
Then $\Ku(X)$ has a Bridgeland stability condition.
\end{Thm}

Here $\Ku(X)$ is defined by the semiorthogonal decomposition 
\[ \Db(X) = \langle \Ku(X), \cO_X, \cO_X(H), \cO_X(2H) \rangle, \]
where $H$ is a hyperplane section.
Here $\Ku(X)$ is a K3 category (i.e., the double shift $[2]$ is a Serre functor); conjecturally \cite[Conjecture 1.1]{Kuz:fourfold} it is the derived category of a K3 surface if and only if $X$ is rational. Our results also give the first stability conditions on $\Db(X)$ when $\Ku(X)$ is not equivalent to the derived category of a twisted K3 surface.

\subsection*{Background and motivation}
Wall-crossing for stability conditions on surfaces has had numerous powerful applications, e.g.,~to the geometry of moduli spaces of stable sheaves \cite{BM:walls, Izzet-Jack:ample, Chunyi-Xiaolei:birational}, or to questions of Brill-Noether type \cite{Arend:BN, Feyzbakhsh:Mukai, Izzet-Jack:BN-Hirzebruch}.
It is unrealistic to expect similarly systematic results for higher-dimensional varieties: as even Hilbert schemes of curves on $\P^3$ satisfy Murphy's law \cite{Murphy}, one should instead expect that wall-crossing lacks any generally effective control.
However, Kuznetsov components of Fano varieties are homologically much better behaved than their entire derived category (for example, they can be of Calabi-Yau or Enriques type of smaller dimension).
Thus, one can expect that moduli spaces and wall-crossing for objects can be controlled much more effectively, and are thus a natural starting point for extracting geometric results from categorical properties.

The study of Kuznetsov components of derived categories of Fano varieties started with \cite{BondalOrlov:Main}, and has seen a lot of recent interest, see e.g., \cite{Kuz:V14,Kuz:V12, IlievKatzarkovPrz} for threefolds, and \cite{Kuz:fourfold,AddingtonThomas:CubicFourfolds,addington:two_conjectures} for the cubic fourfold, as well as \cite{Kuz:Fano3folds, Kuz:ICM, Kuz:lectures} for surveys. The interest in them comes from various directions.
They are part of Kuznetsov's powerful framework of Homological Projective Duality \cite{Kuznetsov:HPD}. 
They often seem to encode the most interesting and geometric information about $\Db(X)$ and moduli spaces of sheaves on $X$; 
e.g., several recent constructions of hyperk\"ahler varieties associated to moduli spaces of sheaves on the cubic fourfold are induced by the projection to the K3 category $\Ku(X)$ \cite{KuznetsovMarkushevic} 
(where moduli spaces naturally come with a holomorphic symplectic structure, due to the fact that $\Ku(X)$ is a K3 category). 
In the case of Fano threefolds, there are a number of unexpected equivalences (some conjectural) between Kuznetsov components of pairs of Fano threefolds of index one and two, see \cite{Kuz:Fano3folds} for the theory, and \cite{KuznetsovProkhorovShramov} for an application to Hilbert schemes. In the case of cubic fourfolds, as mentioned above, they conjecturally determine rationality of $X$.
Finally, they are naturally related to Torelli type questions: on the one hand, they still encode much of the cohomological information of $X$;
on the other hand, one can hope to recover $X$ from $\Ku(X)$ (in some cases when equipped with some additional data); see \cite{BMMS:Cubics} for such a result for cubic threefolds, and \cite{HR16:Torelli_cubic_fourfolds} for many hypersurfaces, including cubic fourfolds.

Perhaps the most natural way to extract geometry from $\Ku(X)$ is to study moduli spaces of stable objects---hence the interest in the existence of stability conditions on $\Ku(X)$.
This question was first raised for cubic threefolds in \cite{Kuz:V14},  for cubic fourfolds by Addington and Thomas \cite{AddingtonThomas:CubicFourfolds} and Huybrechts \cite{Huy:cubics}, and in the generality of our results  by Kuznetsov in his lecture series \cite{Kuznetsov:Trieste-lectures}. 

\subsection*{Prior work}
When $X$ is a Fano threefold of Picard rank one, stability conditions on $\Db(X)$ have been constructed in \cite{Chunyi:Fano3folds}. 
However, in general these do not descend to stability conditions on the semiorthogonal component $\Ku(X)$, and due to their importance for moduli spaces, a direct construction of stability conditions on $\Ku(X)$ is of independent interest.

For Fano threefolds of index two, our Theorem~\ref{thm:main1} is referring to the decomposition $\Db(X) = \langle \Ku(X), \cO_X, \cO_X(H) \rangle$. Their deformation type is determined by $d = H^3 \in \{1, 2, 3, 4, 5\}$. The result is straightforward from prior descriptions of $\Ku(X)$ for $d \ge 4$ in \cite{orlov:Y5,BondalOrlov:Main}, due to \cite{BMMS:Cubics} for cubic threefolds ($d = 3$) and new for $d \in \{1, 2\}$. 
The most interesting cases of index one are those of even genus $g_X = \frac 12 H^3 + 1$, for which Mukai \cite{Mukai:Fano3folds} constructed an exceptional rank two vector bundle $\cE_2$ of slope $-\frac 12$; in these cases our Theorem refers to the semiorthogonal decomposition 
$\Db(X) = \langle \Ku(X), \cE_2, \cO_X \rangle$. The result is straightforward from previous descriptions of $\Ku(X)$ for $g_X \in \{10, 12\}$ in \cite{Kuznetsov:Hyperplane,Kuz:Fano3folds}, due to \cite{BMMS:Cubics} for $g_X = 8$, and new for $g_X = 6$.

For cubic fourfolds containing a plane, stability conditions on $\Ku(X)$ were constructed in \cite{MacriStellari:Cubics}, and \emph{Gepner point} stability conditions 
(invariant, up to rescaling, under the functor $(1)$  in Theorem~\ref{thm:CategoricalTorelli})
in \cite{Toda:Gepner-Orlov-Kuznetsov}. In this case, $\Ku(X)$ is equivalent to the derived category of a K3 surface with a Brauer twist.

\subsection*{Applications} 
Kuznetsov conjectured an equivalence between $\Ku(Y_d)$ and $\Ku(X_{4d+2})$ for appropriate pairs $Y_d$ and $X_{4d+2}$, where $Y_d$ is a Fano threefold of Picard rank one, index two and degree $d\geq 2$, and $X_{4d+2}$ is
of index one and genus $2d+2$ (degree $4d+2$). Our results may be helpful in reproving known cases, and proving new cases of these equivalences, by identifying moduli spaces of stable objects in both categories. We illustrate this for $d=4$, see
Example~\ref{ex:d4}.

For cubic fourfolds over the complex numbers, we show in the appendix, written jointly with Xiaolei Zhao, that the existence of stability conditions on $\Ku(X)$ is already enough to reprove a
categorical Torelli theorem for very general cubic fourfolds (which is a special case of \cite[Corollary~2.10]{HR16:Torelli_cubic_fourfolds}):

\begin{repThm}{thm:CategoricalTorelli}
Let $X$ and $Y$ be smooth cubic fourfolds over $\C$.
Assume that $H_{\mathrm{alg}}^*(\Ku(X),\Z)$ has no $(-2)$-classes.
Then $X\cong Y$ if and only if there is an equivalence
$\Phi \colon \Ku(X) \to \Ku(Y)$ whose induced map $H^*_{\mathrm{alg}}(\Ku(X), \Z) \to H^*_{\mathrm{alg}}(\Ku(Y), \Z)$ commutes with the action of $(1)$.
\end{repThm}
Here $(1)$ denotes the autoequivalence of $\Ku(X)$ induced by $\blank \otimes \cO_X(1)$ on $\Db(X)$; the numerical Grothendieck group of $\Ku(X)$ is denoted by $H_{\mathrm{alg}}^*(\Ku(X),\Z)$.

The idea is quite simple: we show that the projection of ideal sheaves of lines on $X$ to $\Ku(X)$ are stable for \emph{all} stability conditions on $\Ku(X)$; therefore, the Fano variety of lines can be recovered from $\Ku(X)$. 
An additional argument based on the compatibility with $(1)$ shows that the polarization coming from the Pl\"ucker embedding is preserved. By a classical argument, this is enough to recover $X$.

We also show that, with the same arguments as in \cite{HR16:Torelli_cubic_fourfolds}, Theorem~\ref{thm:CategoricalTorelli} is enough to reprove the classical Torelli theorem for cubic fourfolds.

\subsection*{Approach}
We establish  methods for inducing t-structures and stability conditions on $\Ku(X)$ from $\Db(X)$. The former, see Corollary~\ref{cor:induce}, generalizes a  construction that first appeared in \cite{vdB:blowingdown}. 
The latter, Proposition~\ref{prop:inducestability}, gives in addition  Harder-Narasimhan filtrations and the support property on  $\Ku(X)$ (and thus a stability condition), given an appropriate \emph{weak stability condition} on $\Db(X)$.

The crucial assumption for both methods is that the relevant \emph{heart} $\cA$ in $\Db(X)$ contains the exceptional objects $E_1, \dots, E_m$, while its shift $\cA[1]$ contains their Serre duals $S(E_1), \dots, S(E_m)$. This turns out to be a surprisingly subtle property.
Already in the case of Fano threefolds, we have to go in three steps. We start with ordinary slope-stability, tilt and deform to obtain a weak stability condition,  called \emph{tilt-stability} in \cite{BMT:3folds-BG}, on a heart $\Coh^\beta(X)$ of two-term complexes.
Then we have to move essentially as far from the large-volume limit in the space of tilt-stability conditions as the Bogomolov-Gieseker (BG) inequality allows us to;
finally, we tilt a second time to obtain a weak stability condition on $\Db(X)$ that induces one on $\Ku(X)$.

A similar approach for cubic fourfolds would be beyond current methods, as the required third tilt would need a strong BG type inequality for the third Chern character (as proposed in \cite{PT15:bridgeland_moduli_properties}) that is not currently known for any fourfold. Instead, we use the rational fibration in conics
$X \dashrightarrow \P^3$ and Kuznetsov's theory of derived categories of quadric fibrations \cite{Kuz:Quadric} to 
reinterpret $\Ku(X)$ as a semiorthogonal component in the derived category $\Db(\P^3, \cB_0)$ of modules over the associated sheaf of Clifford algebras on $\P^3$, see Section~\ref{sec:geometric}.
The key is now that on the one hand, Riemann-Roch along with a precise description of the $K$-group yields a strong BG inequality in $\Db(\P^2, \cB_0|_{\P^2})$ for any hyperplane $\P^2 \subset \P^3$ (stronger than the classical BG inequality); on the other hand, we adapt Langer's machinery of effective restriction theorems to extend this inequality to higher dimension, see Section~\ref{sec:Bogomolov}.
In other words, $\Ku(X) \subset \Db(\P^3, \cB_0)$ still behaves somewhat like the Kuznetsov category of a Fano threefold admitting a sufficiently strong BG inequality.

\subsection*{Updates} 
We conclude with a discussion of the significance of our recent work
\cite{BLMNSP19}, joint also with Howard Nuer and Alexander Perry, for the results in this article, as well as other recent work that builds on this article.
In [ibid.], we define and construct a notion of a stability condition for a family of varieties, and appropriate semiorthogonal components in their derived categories; 
in particular, based on the fiberwise construction in the present article we construct a stability condition for a family of Kuznetsov categories of Fano threefolds, and cubic fourfolds, in the same generality as Theorem~\ref{thm:main1} and Theorem~\ref{thm:main2}. 

This in particular includes the construction of moduli spaces of semistable objects in the setting of the present paper, which can then be studied by deformation to special threefolds or cubic fourfolds via a relative moduli space of semistable objects. In the case of cubic fourfolds, this leads for example to the non-emptiness of moduli spaces whenever the expected dimension is non-negative (thus extending the deformation arguments of \cite{Goettsche-Huybrechts,OGrady:weighttwo,Yoshioka:Abelian} from families of K3 surfaces to families of Kuznetsov components of cubic fourfolds, as well as the basic existence result for moduli spaces of stable complexes in \cite{Toda:K3}); to the extension of results by Addington-Thomas \cite{AddingtonThomas:CubicFourfolds} and Huybrechts \cite{Huy:cubics}, thereby giving a Hodge-theoretic characterisation of the locus where $\Ku(X)$ is derived equivalent to a (twisted) K3 surface (and, conjecturally, where $X$ is rational); and to the construction of locally complete unirational families of hyperk\"ahler varieties of arbitrarily high dimension and degree. It also provides the full strength of the results of \cite{BM:walls} on the birational geometry of moduli spaces in $\Ku(X)$. 

Some interesting examples of such moduli spaces have already been studied in detail since the first version of this article appeared.
For example, in \cite{LPZ:Cubics}, the Fano variety of lines on a cubic fourfold is described as a moduli space of stable objects in general (without the assumption on $(-2)$-classes appearing in the proof of Theorem~\ref{thm:CategoricalTorelli}), as well as the $8$-dimensional hyperk\"ahler variety associated to cubic fourfolds via the Hilbert scheme of twisted cubics \cite{LLSvS} (thus extending \cite{LLMS:LLSvS}, which only holds for a very general cubic fourfold; this construction also behaves well in family over the moduli space of cubic fourfolds, thus including the results in \cite{Ouchi:cubic4fold} when the cubic fourfold contains a plane). 
In \cite{Alex-Laura-Xiaolei:GM}, the authors produce stability conditions on Kuznetsov components of Gushel-Mukai. Their general approach is similar to ours, and in particular based on the inducing methods of Sections \ref{sec:inducing} and \ref{sec:inducestability}, but their geometric setup is more involved than ours; they obtain many results analogous to the case of cubic fourfolds.

In another direction, moduli spaces of stable objects on Fano threefolds of index two have been studied in \cite{APR:indextwo}, with a categorical Torelli theorem in the case of index two and degree two as a corollary.

\subsection*{Base field}
For most of the paper, we work over an algebraically closed field $k$ of arbitrary characteristic. In the case of Fano threefolds, we will assume the characteristic to be either zero or sufficiently large when we need explicit semiorthogonal decompositions of the derived category (such as those in \cite{Kuznetsov:Hyperplane,Kuz:Fano3folds}), for which we do not know a reference in full generality. In the case of cubic fourfolds, we need $\Char k \neq 2$ to apply the finite characteristic version given in \cite{ABB:intersectionsofquadrics} of Kuznetsov's description of the derived category of quadric fibrations \cite{Kuz:Quadric}. Finally, we assume $k = \C$ when we use the Mukai lattice and Hodge-theory of the Kuznetsov component of cubic fourfolds introduced in \cite{AddingtonThomas:CubicFourfolds}.

\subsection*{Acknowledgements} We are very grateful to Xiaolei Zhao for agreeing to make the results of the appendix part of this paper.
We are also indebted to Marcello Bernardara for very valuable help in a long discussion on derived categories of quadric fibrations and modules over the Clifford algebra, and to Howard Nuer and Alex Perry for many suggestions and conversations.
Moreover, the paper benefited from many useful discussions with Asher Auel, Agnieszka Bodzenta, Fran\c{c}ois Charles, Igor Dolgachev, Fr\'{e}d\'{e}ric Han, Daniel Huybrechts, Alexander Kuznetsov, Chunyi Li, Richard Thomas, Yukinobu Toda, and Song Yang.
We are grateful to Franco Rota and Marin Petkovic for pointing out an error in Lemma~\ref{lem:4term} in the first two arXiv versions of this paper, along with the appropriate correction.
We would also like to thank the referee for the careful reading of the manuscript and detailed and thoughtful comments.

The authors would like to acknowledge the following institutions, where parts of this paper have been written: Institut des Hautes \'Etudes Scientifiques, Institut Henri Poincar\'e, Mathematical Sciences Research Institute, Northeastern University, Universit\`a degli studi di Milano, Universit\'e Paris Diderot, Universit\'e Paul Sabatier, and University of Edinburgh.

\section{Review on tilt and Bridgeland stability}\label{sec:tilt}

We begin with a quick review about weak and Bridgeland stability conditions. 
Let $\cD$ be a triangulated category and let $K(\cD)$ denote the Grothendieck group of $\cD$. Fix a finite rank lattice $\Lambda$ and a surjective group homomorphism $v\colon K(\cD) \onto \Lambda$.

\subsection*{Weak stability conditions}

A weak stability condition has two ingredients: a heart of a bounded t-structure and a weak stability function.

\begin{Def}[{\cite[Lemma~3.2]{Bridgeland:Stab}}] \label{def:heart}
A \emph{heart of a bounded t-structure} is a full subcategory $\AA \subset \DD$ such that 
\begin{enumerate}[{\rm (a)}]
\item for $E, F \in \AA$ and $n < 0$ we have $\Hom(E, F[n]) = 0$, and
\item for every $E \in \DD$ there exists a sequence of morphisms
\[ 0 = E_0 \xrightarrow{\phi_1} E_1 \to \dots \xrightarrow{\phi_m} E_m = E \]
such that the cone of $\phi_i$ is of the form $A_i[k_i]$ for some sequence
$k_1 > k_2 > \dots > k_m$ of integers and objects $A_i \in \AA$.
\end{enumerate}
\end{Def}

We write $H^{-k_i}_\AA(E) = A_i$ for the cohomology objects of $E$ with respect to the bounded t-structure.

\begin{Def} \label{def:stabilityfunction}
Let $\cA$ be an abelian category. We say that a group homomorphism $Z \colon K(\cA) \to \C$ is a
\emph{weak stability function} on $\cA$ if, for $E \in \cA$, we have $\Im Z(E) \ge 0$, with $\Im Z(E)
= 0 \Rightarrow \Re Z(E) \le 0$. 
If, moreover, for all $0\neq E\in\cA$, $\Im Z(E) = 0 \Rightarrow \Re Z(E) < 0$, we say that $Z$ is a \emph{stability function} on $\cA$.
\end{Def}

\begin{Def} \label{def:weakstability}
A {\em weak stability condition} on $\cD$ is a pair $\sigma= (\AA, Z)$ consisting of the heart of a bounded t-structure $\AA\subset\cD$ and a group homomorphism $Z\colon \Lambda\to\C$ such that (a)--(c) below are satisfied:
\begin{enumerate}[{\rm (a)}]
\item[(a)] \label{item:upper}
The composition $K(\cA) = K(\cD) \xrightarrow{v} \Lambda \xrightarrow{Z} \C$ is a weak stability function on
$\cA$. By abuse of notation, we will write $Z(E)$ instead of $Z \circ v([E])$ for any $E \in \cD$.
\end{enumerate}
The function $Z$ allows one to define a \emph{slope} for any $E \in \cA$ by setting 
\[ \mu_{\sigma}(E) := \begin{cases} - \frac{\Re Z(E)}{\Im Z(E)} & \text{if $\Im Z(E) > 0$} \\
+\infty & \text{otherwise}
\end{cases}
\]
and a notion of stability: An object $0 \not= E\in\AA$ is $\sigma$-\emph{semistable} if for every proper subobject $F$, we have $\mu_{\sigma}(F) \leq \mu_{\sigma}(E)$.
We will often use the notation $\mu_Z$ as well.
\begin{enumerate}[{\rm (a)}]
\item[(b)] (HN filtrations) We require any object $E$ of $\AA$ to have a Harder-Narasimhan filtration in $\sigma$-semistable ones.

\item[(c)] (Support property) There exists a quadratic form $Q$ on $\Lambda\otimes\R$ such that $Q|_{\ker Z}$ is negative definite, and $Q(E)\geq0$, for all $\sigma$-semistable objects $E\in\AA$.
\end{enumerate}
\end{Def}

As usual, given a non-zero object $E\in\AA$, we will denote by $\mu_\sigma^+(E)$ (resp. $\mu_\sigma^-(E)$) the biggest (resp. smallest) slope of a Harder-Narasimhan factor.

\begin{Rem} \label{rem:defstability}
If in fact $Z$ is a stability function on $\cA$, then the pair $\sigma$ defines a \emph{Bridgeland stability condition}, see \cite[Proposition~5.3]{Bridgeland:Stab}.
\end{Rem}

\begin{Rem} \label{rem:HNexists}
If $Z$ has discrete image in $\C$, and if $\AA$ is noetherian, then the existence of Harder-Narasimhan filtrations 
is automatic by \cite[Lemma~2.4]{Bridgeland:Stab} (see \cite[Proposition~B.2]{localP2}).
\end{Rem}

\begin{Rem} \label{rem:supportproptrivial}
If $\Lambda$ has rank two, and if $Z \colon \Lambda \to \C$ is injective, then the support property is trivially satisfied for any non-negative quadratic form $Q$ on $\Lambda \otimes \R \cong \R^2$.
\end{Rem}

\begin{Rem}
For the purpose of Theorems~\ref{thm:main1} and~\ref{thm:main2}, we will choose $\Lambda$ to be the numerical K-group $\Knum(\DD)$ of $\DD$: it is defined as the quotient of $K(\DD)$ by the kernel of the Euler characteristic pairing $\chi(E, F) = \sum_i (-1)^i \dim \Ext^i(E, F)$.
\end{Rem}

\begin{Ex} \label{ex:slopestability}
To fix notation, we first recall slope-stability as a weak stability condition.
Let $X$ be an $n$-dimensional smooth projective variety and let $H$ be an hyperplane section. For $j=0,\dots,n$, consider the lattices $\Lambda_H^j\cong\Z^{j+1}$ generated by vectors of the form
\[
\left(H^n \ch_0(E), H^{n-1}\ch_1(E),\dots,H^{n-j}\ch_j(E)\right) \in \Q^{j+1}
\]
together with the natural map $v_H^j\colon K(X) \to \Lambda_H^j$.

Then the pair $(\AA = \Coh(X), Z_H)$ with 
\[ Z_H(E) = \mathfrak{i}\, H^n \ch_0(E) - H^{n-1} \ch_1(E) \]
defines slope-stability as a weak stability condition with respect to $\Lambda_H^1$; here, by Remark~\ref{rem:supportproptrivial}, we can choose $Q=0$.
We write $\mu_H$ for the associated slope function.

Slope-semistable sheaves satisfy a further inequality, which will allow us in Proposition~\ref{prop:tiltstability} to improve our positivity condition, by changing the bounded t-structure.
More precisely, by the Bogomolov-Gieseker inequality, for any slope-semistable sheaf $E$, we have $\Delta_H(E)\geq0$, where
\begin{equation} \label{eq:defDeltaH}
 \Delta_H(E) = \left(H^{n-1}\ch_1(E)\right)^2 - 2 H^n \ch_0(E) H^{n-2}\ch_2(E). 
\end{equation}
\end{Ex}

\begin{Rem} \label{rem:BGpositive}
In general, the Bogomolov-Gieseker inequality is known only in characteristic zero;
however, it does hold for Fano threefolds of Picard rank one and arbitrary characteristic.
Indeed, by the induction arguments in \cite{Langer:positive} (which we will use similarly in Section~
\ref{sec:Bogomolov} for sheaves over certain Clifford algebras) it is enough to prove \eqref{eq:defDeltaH} on a smooth hyperplane section $S \in \abs{K_X}$. Then $S$ is a K3 surface, in which case \eqref{eq:defDeltaH} follows from $\chi(E, E) \le 2$ and Hirzebruch-Riemann-Roch. 
\end{Rem}

\subsection*{Tilting} Assume that we are given a weak stability condition $\sigma = (\AA, Z)$, and let $\mu \in \R$. We can form the following subcategories of $\AA$ (where $\langle \dots \rangle$ denotes the extension closure):
\begin{align*}
\TT_\sigma^\mu & = \stv{E}{\text{All HN factors $F$ of $E$ have slope
$\mu_{\sigma}(F) > \mu $}} \\
& = \langle E \colon \text{$E$ is $\sigma$-semistable with $\mu_{\sigma}(E) > \mu$} \rangle, \\
\FF_\sigma^\mu & = \stv{E}{\text{All HN factors $F$ of $E$ have slope
$\mu_{\sigma}(F) \le \mu $}} \\
& = \langle E \colon \text{$E$ is $\sigma$-semistable with $\mu_{\sigma}(E) \leq \mu$} \rangle.
\end{align*}
It follows from the existence of Harder-Narasimhan filtrations that $(\TT_\sigma^\mu, \FF_\sigma^\mu)$ forms a torsion pair in $\AA$ in the sense of \cite{Happel-al:tilting}. In particular, we can obtain a new heart of a bounded t-structure by tilting:

\begin{PropDef}[\cite{Happel-al:tilting}] \label{prop:slopetilt}
Given a weak stability condition $\sigma = (Z, \AA)$ and a choice of slope $\mu \in \R$, there exists a heart of a bounded t-structure defined by
\[ \AA_\sigma^\mu = \langle \TT_\sigma^\mu, \FF_\sigma^\mu[1] \rangle. \]
\end{PropDef}
We will call $\AA^\mu_\sigma$ the heart obtained by tilting $\AA$ with respect to the stability condition $\sigma$ at the slope $\mu$.

Now return to the setting of slope stability as in Example~\ref{ex:slopestability}, and choose a parameter $\beta \in \R$. Then we can apply Proposition~\ref{prop:slopetilt} and obtain: 

\begin{Def}\label{def:Cohbeta}
We write $\Coh_H^\beta(X)\subseteq \Db(X)$ for the heart of a bounded t-structure obtained by tilting $\Coh(X)$ with respect to slope-stability 
at the slope $\mu = \beta$. 
\end{Def}
In particular, $\Coh_H^\beta(X)$ contains slope-semistable sheaves $F$ 
of slope $\mu(F) > \beta$, and shifts $F[1]$ of slope-semistable sheaves $F$ of slope $\mu(F) \le \beta$.
In our setting, the polarization will often be unique, in which case we drop the subscript $H$ from the notation. 

For a coherent sheaf $E$ on $X$, we define the vector
\[
\ch^{\beta}(E) = e^{-\beta H} \ch(E) = 
\left(\ch_0^\beta(E), \ch_1^\beta(E), \dots, \ch_n^\beta(E) \right) \in H^*(X,\mathbb{R}).
\]

\begin{Prop}[\cite{BMT:3folds-BG,BMS}] \label{prop:tiltstability}
Given $\alpha > 0, \beta \in \R$, the pair $\sigma_{\alpha, \beta} = (\Coh^\beta(X), Z_{\alpha, \beta})$
with $\Coh^\beta(X)$ as constructed above, and
\[
Z_{\alpha, \beta}(E) := \mathfrak{i}\, H^{n-1}\ch_1^\beta(E) + \frac 12 \alpha^2 H^n \ch_0^\beta(E) - H^{n-2}\ch_2^\beta(E)
\]
defines a weak stability condition on $\Db(X)$ with respect to $\Lambda_H^2$. The quadratic form $Q$ can be given by the discriminant $\Delta_H$ as defined in \eqref{eq:defDeltaH}.

These stability conditions vary continuously as $(\alpha, \beta) \in \R_{>0} \times \R$ varies.
\end{Prop}

In particular, this means that the family of stability conditions $\sigma_{\alpha, \beta}$ satisfies wall-crossing: for every fixed class $v \in \Lambda_H^2$, there is a locally finite wall-and-chamber structure on $\R_{>0} \times \R$ controlling stability of objects of class $v$.

\subsection*{Geometry of walls}
It is very helpful to visualize the structure of this family of stability conditions, and the associated walls, via the cone associated to the quadratic form $\Delta_H$. Consider $\R^3 = \Lambda_H^2 \otimes \R$ with coordinates
$\left(H^n \ch_0, H^{n-1}\ch_1, H^{n-2}\ch_2\right)$, and with the quadratic form of signature (2, 1) induced by $\Delta_H$. The map
\[ (\alpha, \beta) \mapsto \Ker Z_{\alpha, \beta} \subset \R^3
\]
assigns to each point in the upper half plane $\R_{>0} \times \R$ a line contained in the negative cone of $\Delta_H$; this induces a homeomorphism between the upper half plane and the projectivization of the negative cone of $\Delta_H$.
The kernels $\Ker Z_{\alpha,\beta}$ with a fixed $\beta=\mu$ lie all in the same plane passing through $(0,0,0)$ and $(0,0,1)$. 
The quadric $\Delta_H(\blank) = 0$ contains, of course, the vectors $v_H(L)$ for any line bundle $L$ proportional to $H$, as well as $(0, 0, 1)$. 

Now fix a Chern character $v$. Then the walls of tilt-stability correspond to hyperplanes $\cW$ in $\R^3$ containing $v_H(v)$: a stability conditions $\sigma_{\alpha, \beta}$ is contained in the wall if and only if $\Ker Z_{\alpha, \beta}$ is contained in $\cW$.
Moreover, Proposition~\ref{prop:largevolume} below will translate into the statement that for $\Ker Z_{\alpha, \beta}$ near $(0, 0, 1)$, slope-stable vector bundles of a fixed class are $\sigma_{\alpha, \beta}$-stable. In Figure~\ref{fig:tiltwalls}, we draw a cross-section of the negative cone. 

\begin{figure}[!tbp]
\definecolor{ttttff}{rgb}{0.2,0.2,1}
\definecolor{zzttqq}{rgb}{0.6,0.2,0}
\definecolor{qqqqff}{rgb}{0,0,1}
\begin{tikzpicture}[line cap=round,line join=round,>=triangle 45,x=0.65cm,y=0.65cm]
\clip(-2.1,-5.3) rectangle (7.12,2.5);
\draw [fill=black,fill opacity=0.05] (2.32,-1.42) circle (2.07cm);
\begin{scriptsize}
\fill [color=black] (0,-5) circle (1.5pt) coordinate (v);
\fill [color=qqqqff] (4.26,-1.17) circle (1.5pt);
\fill [color=zzttqq] (2.32,1.77) circle (1.5pt);
\end{scriptsize}
\draw (v)-- (-0.8,2.66);
\draw (v) -- (4.64, 8.54);
\draw (v)-- (6.46,2.36);
\draw (v) -- (8.52, 2.66);
\begin{footnotesize}
\draw [color=ttttff](4.26,-1.17) node[anchor=east] {$\mathrm{Ker} Z_{\alpha,\beta}$};
\draw [color=zzttqq](2.45,1.7) node[anchor=south east] {$(0,0,1)$};
\draw (-0.2,.82) node[anchor=north west] {$\Delta_H<0$};
\draw (-1.6,-4.5) node[anchor=north west] {$v_H(v)$};
\draw [color=qqqqff] (2.32,1.77)-- (5.25,-2.68);
\draw [color=qqqqff](3.12,0.95) node[anchor=north west] {$\beta=\mu$};
\end{footnotesize}
\begin{tiny}
\draw [color=qqqqff](2.45,1.7) node[anchor=north] {$\alpha\gg 0$};
\draw [color=qqqqff](5.16,-2.31) node[anchor=north east] {$\alpha\gtrsim 0$};
\end{tiny}

\end{tikzpicture}
\caption{Walls in the cross-section of the negative cone}
\label{fig:tiltwalls}
\end{figure}

\subsection*{Basic properties of tilt-stability}
We recall here further properties of tilt-stability that we will use later.
The first is a well-known variant of \cite[Lemma~14.2]{Bridgeland:K3}.

\begin{Prop} \label{prop:largevolume}
Let $\beta\in\R$ and let $E$ be a slope-stable vector bundle.
If $\mu_H(E)>\beta$, then $E\in\Coh^\beta(X)$ is $\sigma_{\alpha, \beta}$-stable for all $\alpha$ sufficiently large.
If $\mu_H(E)\leq\beta$, then $E[1]\in\Coh^\beta(X)$ is $\sigma_{\alpha, \beta}$-stable for all $\alpha$ sufficiently large.
\end{Prop}

The next property is a consequence of Bogomolov-Gieseker inequality for tilt-stability.

\begin{Prop}[{\cite[Proposition~7.4.1]{BMT:3folds-BG} or \cite[Corollary~3.11]{BMS}}] \label{prop:Delta0stable}
Let $E$ be a slope-stable vector bundle with $\Delta_H(E) = 0$. Then $E$, respectively $E[1]$, is $\sigma_{\alpha, \beta}$-stable for all $(\alpha, \beta) \in \R_{>0} \times \R_{<\mu_H(E)}$, respectively $(\alpha, \beta) \in \R_{>0} \times \R_{\geq\mu_H(E)}$.

Conversely, let $E$ be a $\sigma_{\alpha, \beta}$-stable object 
with $\Delta_H(E) = 0$ and $\beta > \mu_H(E)$. Then $E = F[1]$ is the shift of a slope-semistable vector bundle $F$.
\end{Prop}

We also need the following variant of Proposition~\ref{prop:tiltstability}, that appears 
implicitly, but not explicitly, in \cite{BMT:3folds-BG} for $\mu = 0$.
It is also a consequence of the general results in \cite{PT15:bridgeland_moduli_properties}, which are, however, depending on a conjectural Bogomolov-Gieseker type inequality involving $\ch_3$; we will give a proof without such an assumption.

Choose a weak stability condition $\sigma_{\alpha, \beta}$ as in Proposition
\ref{prop:tiltstability}, and $\mu \in \R$. 
Following Proposition/Definition~\ref{prop:slopetilt}, we obtain a tilted heart, which we will
denote by
\[ 
\Coh_{\alpha, \beta}^\mu(X):= \cA_{\sigma_{\alpha, \beta}}^\mu.
\]
Let $u \in \C$ be the unit vector in the upper half plane with $\mu = -\frac{\Re u}{\Im u}$. 
Then it is straightforward to see that 
\[ Z_{\alpha, \beta}^\mu := \frac 1u Z_{\alpha, \beta} \]
is a weak stability function on 
$\Coh_{\alpha, \beta}^\mu(X)$.

\begin{Prop} \label{prop:tiltedtiltstability}
The pair $(\Coh_{\alpha, \beta}^\mu(X), Z_{\alpha, \beta}^\mu)$ is a weak stability
condition on $\Db(X)$.
\end{Prop}

In the context of stability conditions, this statement would follow automatically from Proposition~\ref{prop:tiltstability} via the $\widetilde{\GL_2^+}(\R)$-action on the space of stability
conditions.
However, due to the special treatment of objects $E \in \cA$ with $Z(E) = 0$, there is no a priori reason for the existence of such
action on the set of weak stability conditions; in fact, Proposition~\ref{prop:tiltedtiltstability} typically cannot be applied iteratively.

Proposition~\ref{prop:tiltedtiltstability} will be a consequence of the following general result.

\begin{Lem}\label{lem:TiltGoodIsAlmostGood}
Let $\sigma=(\AA,Z)$ be a weak stability condition.
Let $\AA_0\subset \AA$ be the abelian subcategory whose objects have $Z=0$.
Assume the following:
\begin{enumerate}[{\rm (a)}]
\item\label{item:TiltGoodIsAlmostGood1} $\AA_0$ is noetherian.
\item\label{item:TiltGoodIsAlmostGood2} For $A\in\AA$, there exists a maximal subobject $\widetilde{A}\hookrightarrow A$, $\widetilde{A}\in\AA_0$, such that $\Hom(\AA_0,A/\widetilde{A})=0$.
\item\label{item:TiltGoodIsAlmostGood3} For $A\in\AA$ with $\mu_\sigma^+(A)<+\infty$, there exists $A\hookrightarrow \widehat{A}\in \AA$ such that $\Hom(\AA_0^{},\widehat{A}[1])=0$ and $\widehat{A}/A\in\AA_0^{}$.
\end{enumerate}
Then, given $\mu\in\R$, we have
\[
\AA_0^{}=\left\{E\in\AA_\sigma^\mu\,:\, Z(E)=0 \right\}
\]
and, for all $B\in\AA_\sigma^\mu$, there exists a maximal subobject $\widetilde{B}\hookrightarrow B$, $\widetilde{B}\in\AA_0^{}$, such that $\Hom(\AA_0^{},B/\widetilde{B})=0$. 
\end{Lem}

\begin{proof}
First of all, note that $\AA_0^{}$ is closed under subobjects, quotients, and extensions in $\AA$.

To prove the first statement, note that $\cA_0^{}\subset \cT_\sigma^\mu\subset\cA_\sigma^\mu$ and so $\cA_0^{}\subset \left\{E\in\AA_\sigma^\mu\,:\, Z(E)=0 \right\}$.
For the reverse inclusion note that $\frac{1}{u}Z$ is a weak stability function on $\cA_\sigma^\mu$, where $u \in \C$ is the unit vector in the upper half plane with $\mu = -\frac{\Re u}{\Im u}$.
Thus, for $E\in\cA_\sigma^\mu$ with $\frac{1}{u}Z(E)=0$, we must have $\frac{1}{u}Z(H^{-1}_\cA(E))=\frac{1}{u}Z(H^{0}_\cA(E))=0$.
By definition of $\cF_\sigma^\mu$, this implies that $H_{\cA}^{-1}(E)=0$.
Therefore, $E\in \cT_\sigma^\mu\subset \cA$, and so $E\in\AA_0$, proving the first claim.

We observe that $\cA_0^{}$ is therefore closed under subobjects, quotients, and extensions in $\AA^\mu_\sigma$ as well.

To prove the second statement, let $B\in \cA_\sigma^\mu$, let $K\in \AA_0^{}$, and assume $\Hom(K,B)\neq 0$.
By the previous observation, we can assume that $K$ is a subobject of $B$. 
Let $\widetilde{K}\subset H_\AA^0(B)$ be the image of the composition $K\into B\onto H^0_\cA(B)$ and 
$\widehat{K}\subset H^{-1}_\cA(B)[1]$ its kernel. Then $\widetilde{K},\widehat{K}\in\cA_0$. 
By property \eqref{item:TiltGoodIsAlmostGood2}, we have $\widetilde{K}\subset \widetilde{H^0_\cA(B)}$ in $\cA_0^{}$.
Similarly by property \eqref{item:TiltGoodIsAlmostGood3}, we have 
$\widehat{K}\subset \widehat{H^{-1}_\cA(B)}/H^{-1}_\cA(B)$ in $\cA_0^{}$. Given an increasing sequence of subobjects $K_n \subset B$ with $K_n \in \cA_0$, the corresponding sequences $\widetilde{K_n}$ and $\widehat{K_n}$ also form increasing sequences of subobjects; 
by noetherianity of $\AA_0$, both terminate, and thus we obtain the existence of a maximal subobject $\widetilde{B}$ as we wanted.
\end{proof}

\begin{Ex}\label{ex:CohXisgood}
Let $\sigma=(\Coh(X),Z_H)$ be the weak stability condition of Example~\ref{ex:slopestability}.
Then $\sigma$ satisfies the conditions of Lemma~\ref{lem:TiltGoodIsAlmostGood}. 
Here $\Coh(X)_0$ are the torsion sheaves supported in codimension at least $2$. For $A\in\Coh(X)$, $\widetilde{A}$ is the torsion part in codimension at least $2$ of the torsion filtration, while $\widehat{A}$ is the double-dual of a torsion-free sheaf. 
\end{Ex}

The key fact is that the same holds for tilt-stability:

\begin{Prop}\label{prop:CohBetaIsGood}
The weak stability condition $\sigma_{\alpha,\beta}=(\Coh^\beta(X),Z_{\alpha,\beta})$ satisfies the hypothesis of Lemma~\ref{lem:TiltGoodIsAlmostGood}.
\end{Prop}

In \cite[Definition~2.13]{PT15:bridgeland_moduli_properties}, condition \eqref{item:TiltGoodIsAlmostGood3} of Lemma~\ref{lem:TiltGoodIsAlmostGood} is part of the definition of \emph{good very weak stability condition}.
Proposition~\ref{prop:CohBetaIsGood} can be deduced from the general result \cite[Proposition~3.10]{PT15:bridgeland_moduli_properties}; the issue is that \cite[Conjecture 3.8]{PT15:bridgeland_moduli_properties} is assumed as hypothesis.
Our proof is unconditional and follows closely \cite[Section~5]{BMT:3folds-BG}.
More precisely, on the tilted category, there is a double-dual operation as well as for coherent sheaves.
This was defined in \cite[Proposition~5.1.3]{BMT:3folds-BG} for threefolds, and an analogous statement works in any dimension.

Let $\mathbb{D}(\blank):=\RlHom(-,\cO_X)[1]$ denote the duality functor.

\begin{Lem}\label{lem:DoubleDualCohBeta}
Let $E\in \Coh^\beta(X)$ be such that $\mu_{\sigma_{\alpha,\beta}}^+(E)<+\infty$.
\begin{enumerate}[{\rm (a)}]
\item\label{item:DoubleDualCohBeta1} We have $H^j_{\Coh^{-\beta}(X)}(\mathbb{D}(E))=0$, for all $j<0$, and $H^j_{\Coh^{-\beta}(X)}(\mathbb{D}(E))$ is a torsion sheaf supported in codimension at least $j+2$, for all $j\geq 1$. We define $E^\sharp$ as $H^0_{\Coh^{-\beta}(X)}(\mathbb{D}(E))$.
\item\label{item:DoubleDualCohBeta2} There exists an exact sequence in $\Coh^\beta(X)$
\[
0\to E \to E^{\sharp\sharp} \to E^{\sharp\sharp}/E \to 0
\]
where $E^{\sharp\sharp}/E$ is a torsion sheaf supported in codimension at least 3, and $E^{\sharp\sharp}$ is quasi-isomorphic to a two term complex $C^{-1} \to C^0$ with $C^{-1}$ locally-free and $C^0$ reflexive.
\end{enumerate}
\end{Lem}

Part \eqref{item:DoubleDualCohBeta1} of Lemma~\ref{lem:DoubleDualCohBeta} can be rephrased by saying that there exists an exact triangle
\begin{equation*}\label{eqn:dual}
E^\sharp \to \mathbb{D}(E)\to Q,
\end{equation*}
with $E^\sharp\in \Coh^{-\beta}(X)$, $H_{\Coh(X)}^j(Q)=0$ for all $j\leq 0$, and $H_{\Coh(X)}^j(Q)$ a torsion sheaf supported in codimension at least $j+2$, for all $j\geq 1$.

\def\bD{\mathbb{D}}
\begin{proof}
In this proof, we will write $(\cD^{\le 0}, \cD^{\ge 0})$ for the standard t-structure on $\Db(X)$, and $H^i$ and $\tau_{\leq n}, \tau_{\ge n}$ for the associated cohomology and truncation functors.

We first recall that for a coherent sheaf $F$, the complex $\bD(F)$ satisfies
\[ H^{j}(\bD(F)) = \begin{cases}
0 & \text{for $j < -1$} \\
\lHom(F, \cO_X) & \text{for $j = -1$} \\ 
\text{$\lExt^{j+1}(F, \cO_X)$, a sheaf supported in codimension $\ge j+1$} &  \text{for $j \ge 0$.}
\end{cases}
\]
Moreover, if $F$ is supported in codimension $k$, then $H^{k-1}(\bD(E))$ is the smallest degree with a non-vanishing cohomology sheaf.

\eqref{item:DoubleDualCohBeta1}
We dualize the triangle
$H^{-1}(E)[1] \to E \to H^0(E)$ and consider the long exact cohomology sequence with respect to $\Coh X$. We first get an isomorphism 
\[ \lHom(H^0(E), \cO_X) = H^{-1}\left(\bD(H^0(E))\right) \cong H^{-1}(\bD(E)). \]
As $H^0(E) \in \cT^\beta$, i.e., $\mu^- (H^0(E)) < \beta$, we get $\mu^+(H^{-1}(\bD(E))) > - \beta$, i.e., $H^{-1}(\bD(E))) \in \cF^{-\beta}$. 

We next get a long exact sequence
\[
0 \to \lExt^1(H^0(E),\cO_X)\\
  \to H^0(\mathbb{D}(E))\to \lHom(H^{-1}(E),\cO_X)\xrightarrow{\delta} \lExt^2(H^0(E),\cO_X).
\]
Since $H^{-1}(E) \in \cF^\beta$, we have $\mu^+(H^{-1}(E)) \le \beta$. In case of equality, let $F \subset H^{-1}(E)$ be the first step of its HN filtration; then the composition $F[1] \to H^{-1}(E)[1] \to E$ is a subobject of $E$ in 
$\Coh^\beta X$ with $\mu_{\sigma_{\alpha, \beta}}(F[1]) = +\infty$; this is a contradiction to the assumption $\mu_{\sigma_{\alpha, \beta}}^+ (E) < +\infty$. Therefore, we have strict inequality $\mu^+(H^{-1}(E)) < \beta$, and so $\lHom(H^{-1}(E),\cO_X) \in \cT^{-\beta}$. Since $\lExt^2(H^0(E),\cO_X)$ is supported in codimension at least two, the kernel of $\delta$ is also contained in $\cT^{-\beta}$. Since $\lExt^1(H^0(E),\cO_X)$ is a torsion sheaf, this shows $H^0(\mathbb{D}(E)) \in \cT^{-\beta}$. 

We now define $E^\sharp := \tau_{\le 0} \bD(E)$ and $Q:= \tau_{\ge 1} \bD(E)$. The previous arguments show $H^{-1}(E^\sharp) = H^{-1}(\bD(E)) \in \cF^{-\beta}$ and $H^0(E^\sharp) = H^0(\bD(E))  \in \cT^{-\beta}$, and thus $E^\sharp \in \Coh^{-\beta} X$. It remains to show for $j > 0$ that $H^j(\bD(E)) = H^j(Q)$ is a torsion sheaf supported in codimension at least $j+2$. 

The continuation of the long exact cohomology sequence above shows that $H^j(Q)$ is supported in codimension $\ge j+1$. Therefore, 
$\bD(H^j(Q)[-j]) \in \cD^{\ge 0}$, and $H^j(Q)$ is supported in codimension at least $j+2$ if and only if $\bD(H^j(Q)[-j]) \in \cD^{\ge 1}$. 
Assume for contradiction that there is a largest possible $j_0$  such that $H^0(\bD(H^{j_0}(Q)[-j_0])) \neq 0$. By induction on the number of non-zero cohomology objects, we see that $\bD(\tau_{\le k}Q), \bD(\tau_{\ge k}Q), \bD(Q) \in \cD^{\ge 0}$ for all $k$. Similarly, $\bD(\tau_{\ge j_0 +1} Q ) \in \cD^{\ge 1}$, and 
$H^1(\bD(\tau_{\ge j_0 +1} Q ))$ is supported in codimension at least $j_0 + 3$. 
Dualizing the exact triangle
\[ H^{j_0}(Q)[-j_0] \to \tau_{\ge j_0} Q \to \tau_{\ge j_0 +1} Q  \]
and taking its long exact cohomology sequence give
\[ 0 \to H^0(\bD(\tau_{\ge j_0} Q)) \to H^0(\bD(H^{j_0}(Q)[-j_0]))
\to H^1(\bD(\tau_{\ge j_0 +1} Q )).
\]
Since the middle object is supported in  codimension exactly $j_0+1$, and the right object in codimension $\ge j_0+3$, this shows $H^0(\bD(\tau_{\ge j_0} Q)) \neq 0$. Dualizing the exact triangle 
$\tau_{\le j_0 -1} Q \to  Q \to \tau_{\ge j_0} Q$ gives an injective map
\[ 0 \neq H^0(\bD(\tau_{\ge j_0} Q))  \into H^0(\bD(Q)), \]
thus $H^0(\bD(Q))$ is a non-zero torsion sheaf supported in codimension at least two. Now consider the exact triangle
\[
\bD(Q) \to \bD(\bD(E)) = E \to \bD(E^\sharp).
\]
The same arguments as before show $H^{-1}(\bD(E^\sharp)) \in \cF^{\beta}$, and so  $\Hom(H^0(\bD(Q)) , \bD(E^{\sharp})[-1]) = \Hom(H^0(\bD(Q)), H^{-1}(\bD(E^\sharp))) = 0$.
Therefore, the composition $H^0(\bD(Q)) \to \bD(Q) \to E$ (where the first morphism exists due to $\bD(Q) \in \cD^{\ge 0}$) is non-zero. This is a contradiction to $\mu_{\sigma_{\alpha, \beta}}^+(E) < +\infty$.

\eqref{item:DoubleDualCohBeta2} By part \eqref{item:DoubleDualCohBeta1} we have two exact triangles
\begin{equation*}
\mathbb{D}(Q)\to E\to \mathbb{D}(E^{\sharp})\qquad E^{\sharp\sharp}\to \mathbb{D}(E^{\sharp})\to Q'
\end{equation*}
with $\bD(Q), Q' \in \cD^{\ge 1}$,  and all their cohomology sheaves are  supported in codimension at least 3, whereas $E, E^{\sharp\sharp} \in \Coh^\beta X$. In particular, $E \in \DD^{\le 0}, Q' \in \DD^{\ge 1}$ and the definition of t-structures imply
 $\Hom(E,Q')=0$, so we have an induced morphism $E\to E^{\sharp\sharp}$. The cone $C$ of this morphism fits into an exact triangle $Q'[-1] \to C \to \bD(Q)[1]$
and therefore has only non-negative cohomologies with respect to the t-structure $\Coh^\beta X$. The long exact cohomology sequence of the exact triangle $E \to E^{\sharp\sharp} \to C$ with respect to $\Coh^\beta X$ then shows that $E \to E^{\sharp\sharp}$ is injective in $\Coh^\beta X$, with cokernel a torsion sheaf supported in codimension at least 3.

To finish the proof, we consider a locally-free resolution $G^\bullet$ of $E^\sharp$.
By taking the functor $\mathbb{D}$, we obtain a morphism
\[
E^{\sharp\sharp} \to \left( G_0^\vee \xrightarrow{\phi} G_1^\vee \xrightarrow{\psi} \dots \to G_k^\vee \right)[1].
\]
By part \eqref{item:DoubleDualCohBeta1}, $E^{\sharp\sharp}$ is quasi-isomorphic to the complex $G_0^\vee \xrightarrow{\phi} \Ker(\psi)$, and $\Ker(\psi)$ is reflexive since it is the kernel of a morphism of locally-free sheaves.
\end{proof}

\begin{proof}[Proof of Proposition~\ref{prop:CohBetaIsGood}]
First of all note that by construction of $\sigma_{\alpha, \beta}$, the objects of $\Coh^\beta(X)_0$ are the objects in $\Coh(X)_0$---i.e., torsion sheaves supported in codimension at least two---that additionally satisfy $H^{n-2}\ch_2 = 0$; in other words, torsion sheaves supported in codimension at least three. 
This shows property \eqref{item:TiltGoodIsAlmostGood1}.

Regarding \eqref{item:TiltGoodIsAlmostGood2}, by Lemma~\ref{lem:TiltGoodIsAlmostGood} applied to $(\Coh(X), Z_H)$, we know that for any $A\in \Coh^\beta(X)$ there exists a maximal subobject $\widetilde{A}'$ which is a torsion sheaf supported in codimension at least $2$. 
Since $\Coh^\beta(X)_0\subset\Coh(X)_0$, for any subobject $K\subset A$ such that $K\in \Coh^\beta(X)_0$ we have $K\subset \widetilde{A}'$.
Since $\Coh(X)_0$ is noetherian, we find a maximal subobject $\widetilde{A}\subset \widetilde{A}'\subset A$ satisfying property \eqref{item:TiltGoodIsAlmostGood2}.

To prove property \eqref{item:TiltGoodIsAlmostGood3}, by Lemma~\ref{lem:DoubleDualCohBeta}\eqref{item:DoubleDualCohBeta2}, we have an exact sequence in $\Coh^\beta(X)$
\[
0\to A \to A^{\sharp\sharp} \to A^{\sharp\sharp}/A \to 0
\]
with $A^{\sharp\sharp}/A\in\Coh^\beta(X)_0$ and $A^{\sharp\sharp}$ is quasi-isomorphic to a two-term complex $C^{-1} \to C^0$ with $\Ext^1(\Coh^\beta(X)_0,C^0)=\Ext^2(\Coh^\beta(X)_0,C^{-1})=0$.
Then $\widehat{A}:=A^{\sharp\sharp}$ satisfies property \eqref{item:TiltGoodIsAlmostGood3}.
Indeed, this follows immediately by using the exact triangle
\[
C^{0} \to A^{\sharp\sharp} \to C^{-1}[1].\qedhere
\]
\end{proof}

\begin{proof}[Proof of Proposition~\ref{prop:tiltedtiltstability}]
By Lemma~\ref{lem:TiltGoodIsAlmostGood} and Proposition~\ref{prop:CohBetaIsGood}, every $E \in \Coh_{\alpha, \beta}^\mu(X)$ has a subobject $\widetilde{E} \subset E$ with $\widetilde{E} \in \Coh^\beta(X)_0$ and $\Hom(\Coh^\beta(X)_0, E/\widetilde{E}) = 0$.

Let $(\TT^\mu, \FF^\mu)$ denote the torsion pair in $\Coh^\beta(X)$
from which we construct $\Coh_{\alpha, \beta}^\mu(X)$; by definition,  
$(\FF^\mu[1], \TT^\mu)$ is a torsion pair in $\Coh_{\alpha, \beta}^\mu(X)$. 
If an object $E \in \TT^\mu$ is $\sigma_{\alpha, \beta}$-semistable, then $E/\widetilde{E}$ is $\sigma_{\alpha,
\beta}^\mu$-semistable; similarly for objects $E \in \FF^\mu$ up to the shift $[1]$.
Conversely, any stable object in $\Coh_{\alpha, \beta}^\mu(X)$ is, up to shift,
a $\sigma_{\alpha, \beta}$-stable object in $\Coh^\beta(X)$.

Any object $E$ fits into a short exact sequence
\[ 0 \to F[1] \to E \to T \to 0 \]
with $F \in \FF_{\sigma_{\alpha, \beta}}^\mu$ and
 $T \in \TT_{\sigma_{\alpha, \beta}}^\mu$. The objects $F$ and $T$ have Harder-Narasimhan filtrations
with respect to $\sigma_{\alpha, \beta}$, such that all quotients and subobjects in the filtration lie in
$\cF_{\sigma_{\alpha, \beta}}^\mu$ and $\cT_{\sigma_{\alpha, \beta}}^\mu$, respectively.
Combined, they give a finite filtration of $E$. Let $E \onto Q$ be the quotient of $E$ corresponding
to the last filtration step of $E$. Then the composition $E \onto Q \onto Q/\widetilde{Q}$ gives the maximal
destabilizing quotient of $E$ with respect to $\sigma_{\alpha, \beta}^\mu$; continuing this process
produces the Harder-Narasimhan filtration of $E$. 
\end{proof}

\section{Review on semiorthogonal decompositions}\label{sec:sod}

The second main ingredient in this paper consists in semiorthogonal decompositions. We begin with a very general and quick review, by following \cite{BondalOrlov:Main}.
To this extent, let $\cD$ be a triangulated category.

\begin{Def}\label{def:semiorth}
	A \emph{semiorthogonal} decomposition of $\cD$ is a sequence of full triangulated subcategories $\cD_1,\ldots,\cD_m\subseteq\cD$ such that $\Hom_{\cD}(\cD_i,\cD_j)=0$, for $i>j$ and, for all $G\in\cD$, there exists a chain of morphisms in $\cD$
	\[
	0=G_m\to G_{m-1}\to\ldots\to G_1\to G_0=G
	\]
	with $\mathrm{cone}(G_i\to G_{i-1})\in\cD_i$, for all $i=1,\ldots,m$.
\end{Def}

We will denote such a decomposition by $\cD=\langle\cD_1,\ldots,\cD_m\rangle$.
The semiorthogonality condition implies that $G\mapsto \mathrm{cone}(G_i\to G_{i-1})\in\cD_i$ defines a functor $\mathop{\mathrm{pr}}_i:\DD\to \DD_i$, called the \emph{$i$-th projection functor}.

\begin{Def}\label{def:exceptional}
\begin{enumerate}[{\rm (a)}]
\item An object $E\in\cD$ is \emph{exceptional} if $\Hom_{\cD}(E,E[p])=0$, for all $p\neq0$, and $\Hom_{\cD}(E,E)\cong\C$.

\item A collection $\{E_1,\ldots,E_m\}$ of objects in $\cD$ is called an \emph{exceptional collection} if $E_i$ is an exceptional object, for all $i$, and $\Hom_{\cD}(E_i,E_j[p])=0$, for all $p$ and all $i>j$.
\end{enumerate}
\end{Def}

An exceptional collection $\{E_1,\ldots,E_m\}$ in $\cD$ provides a semiorthogonal decomposition
	\[
	\cD=\langle\cD',E_1,\ldots,E_m\rangle,
	\]
where we have denoted by $E_i$ the full triangulated subcategory of $\cD$ generated by $E_i$ and
	\[
	\cD'=\langle E_1,\ldots,E_m\rangle^\perp=\left\{G\in\cD\colon \Hom(E_i,G[p])=0,\text{ for all }p\text{ and }i\right\}.
	\]
	Similarly, one can define ${}^\perp\langle F_1,\ldots,F_m\rangle=\left\{G\in\DD\,:\,\Hom(G,F_i[p])=0,\text{ for all }p\text{ and }i\right\}$.

Let $E\in\cD$ be an exceptional object.
We can define the \emph{left and right mutation functors}, $\cat{L}_E,\cat{R}_E:\cD\to\cD$ in the following way
\begin{equation*}\label{eqn:LRmutation}
	\begin{split}
	\cat{L}_E(G)&:=\mathrm{cone}\left(\mathrm{ev}\colon \bigoplus_p\Hom_\cD(E,G[p])\otimes E[-p]\to G\right) \in E^\perp \subset \cD \\
	\cat{R}_E(G)&:=\mathrm{cone}\left(\mathrm{ev}^\vee\colon G\to\bigoplus_p\Hom_\cD(G,E[p])^\vee\otimes E[p]\right)[-1] \in {}^\perp E \subset \cD.
	\end{split}
\end{equation*}


We will use the following properties of mutations and semiorthogonal decompositions, where $E, F$ are exceptional objects, and $S$ is a Serre functor of $\cD$.
\begin{enumerate}[(a)]
\item Given a semiorthogonal decomposition
	\[
	\cD = \langle\cD_1,\ldots,\cD_k,E,\cD_{k+1},\ldots,\cD_n\rangle,
	\]
with $E$ exceptional, we can apply left and right mutations and get
	\[
\cD=\langle\cD_1,\ldots,\cD_k,\cat{L}_E(\cD_{k+1}),E,\cD_{k+2},\ldots,\cD_n\rangle
=\langle\cD_1,\ldots,\cD_{k-1},E,\cat{R}_E(\cD_k),\cD_{k+1},\ldots,\cD_n\rangle.
	\]
    \item If $(E, F)$ is an exceptional pair, then
    $\cat{R}_E \cat{L}_E F = F$.
   \item $\cat{R}_{S(E)}$ is right adjoint to $\cat{L}_{E}$ while $\cat{R}_E$ is left adjoint to $\cat{L}_E$.
   \item If $\cD = \langle \cD_1, \cD_2 \rangle$ is a semiorthogonal decomposition, then so are
   \[ \cD = \langle S(\cD_2), \cD_1 \rangle = \langle \cD_2, S^{-1}(\cD_1) \rangle.
   \]
\end{enumerate}

\section{Inducing t-structures} \label{sec:inducing}

Let $\cD$ be a triangulated category admitting a Serre functor $S$, and
with a semiorthogonal decomposition $\cD = \langle \cD_1, \cD_2 \rangle$.
In this section, we give a general criterion for inducing a bounded t-structure on $\cD_1$ from a
bounded t-structure on $\cD$.
While in this paper, we are only interested in the case where $\cD_2$ is generated by an exceptional collection
in $\cD$, we state our criterion in a more general setting in terms of a spanning class of $\cD_2$.

\begin{Def} A spanning class of a triangulated category $\cD$ is a set of objects $\cG$ such that
if $F \in \cD$ satisfies $\Hom(G, F[p]) = 0$ for all $G \in \cG$ and all $p \in \Z$, then $F = 0$.
\end{Def}

The following observation is immediate:
\begin{Lem} \label{lem:obvious}
Let $\cG$ be a spanning class of $\cD_2$. Then for an object $F \in \cD$ we have
$F \in \cD_1$ if and only if $\Hom(G, F[p]) = 0$ for all $G \in \cG$ and all $p \in \Z$.
\end{Lem}

The key ingredient of our entire construction is the following observation, slightly generalizing 
\cite[Theorem~4.1]{vdB:blowingdown} and \cite[Lemma~3.4]{BMMS:Cubics}.
\begin{Lem} \label{lem:induce}
Let $\cA \subset \cD$ be the heart of a bounded t-structure. Assume that the spanning
class $\cG$ of $\cD_2$ satisfies $\cG \subset \cA$, and 
$\Hom(G, F[p]) = 0$ for all $G \in \cG, F \in \cA$, and all $p > 1$. Then $\cA_1 := \cD_1 \cap \cA$
is the heart of a bounded t-structure on $\cD_1$. 
\end{Lem}
\begin{proof}
Clearly $\cA_1$ satisfies the first condition of Definition~\ref{def:heart}, and we only need to verify the second.

Consider $F \in \cD_1$. For every $G \in \cG$ there is a spectral sequence (see e.g., \cite[Proposition~2.4]{Okada:CY2})
\[
E_2^{p, q} = \Hom(G, H^q_\cA(F)[p]) \Rightarrow \Hom(G, F[p+q]).
\]
By the assumptions, these terms vanish except for $p = 0, 1$, and thus the spectral sequence
degenerates at $E_2$. On the other hand, since $G \in \cD_2$ and $F \in \cD_1$ we have $\Hom(G,
F[p+q]) = 0$. Therefore, $\Hom(G, H^q_\cA(F)[p]) = 0$ for all $G \in \cG$ and all $p \in \Z$;
by Lemma~\ref{lem:obvious} we conclude $H^q_\cA(F) \in \cA \cap \cD_1 = \cA_1$. This proves the
claim.
\end{proof}

We will always apply this lemma via the following consequence:

\begin{Cor} \label{cor:induce}
Let $\cA \subset \cD$ be the heart of a bounded t-structure such that $\cG \subset \cA$ and
$S(\cG) \subset \cA[1]$. Then $\cA_1 := \cA \cap \cD_1 \subset \cD_1$ is the heart of a bounded
t-structure.
\end{Cor}
\begin{proof}
Given $G \in \cG$ and $F \in \cA$, as well as $p > 1$ we have
\[ \Hom(G, F[p]) = \Hom(F, S(G)[-p])^\vee = 0 \]
as $S(G)[-p] \in \cA[1-p]$. Therefore, the assumptions of Lemma~\ref{lem:induce} are satisfied.
\end{proof}

\begin{Ex}\label{ex:surfaces}
Let $X$ be a smooth projective surface with canonical divisor $K_X$, and let $H$ be an ample divisor.
Assume that there exists a semiorthogonal decomposition
\[
\Db(X) = \langle \DD_1,\DD_2 \rangle
\]
and a set of generators $\GG$ of $\DD_2$ such that $\GG$ consists of slope-semistable torsion-free sheaves with
\[
\mu_H(G\otimes K_X) \leq \beta < \mu_H(G),
\]
for $\beta\in\Q$ and for all $G\in\GG$; these assumptions often hold for exceptional collections of vector bundles. Then one can use Corollary~\ref{cor:induce} to prove the existence of a heart on $\DD_1$ as follows. 

Consider the tilted heart $\Coh_H^{\beta}(X)$ of Definition~\ref{def:Cohbeta}.
For $G \in \GG$, we have $G \in \Coh_H^{\beta}(X)$ and
$S(G) = G \otimes K_X[2] \in \Coh_H^{\beta}(X)[1]$. So the assumptions of Corollary~\ref{cor:induce} are satisfied, and  
we obtain an induced heart of a bounded t-structure $\AA_1=\Coh_H^{\beta}(X)\cap\DD_1$ on $\DD_1$.

Moreover, the results of the following section will give a Bridgeland stability condition on $\DD_1$ with heart $\AA_1$, constructed as follows. 
Let ${Z}_H(E) := \mathfrak{i}\, H^2 \rk(E) - H\,\ch_1(E)$ be the weak central charge on $\Coh(X)$ inducing
slope-stability for coherent sheaves, and let $u_{\beta}$ be the unit vector in the upper half plane with slope $-\frac{\Re u_{\beta}}{\Im u_\beta} = \beta$.
The stability condition is defined by $(\AA_1, Z_1)$ with
\[ 
Z_1(E) := \frac 1{u_\beta} {Z}_H(E).
\]

We note that $Z_1$ is a weak stability function on $\Coh_H^\beta(X)$, with $Z_1(E) = 0$ for $E \in \Coh_H^\beta(X)$ if and only if $E = H^0(E)$ is a 0-dimensional
torsion sheaf. 
However, if $E= H^0(E)$ is a 0-dimensional torsion sheaf, then $\Hom(G, E) \neq 0$ for any $G \in
\cG$, and hence $E \notin \cD_1$. This shows that $Z_1$ is a stability function for $\AA_1$.

The assumption $\beta \in \Q$ implies that the category $\Coh_H^{\beta}(X)$ is noetherian (see e.g.~\cite[Proof of Lemma~3.2.4]{BMT:3folds-BG}), and thus the same holds for its subcategory $\AA_1$; 
$\beta \in \Q$ also implies that the stability function $Z_1$ is discrete. By \cite[Proposition~B.2]{localP2} this shows that $(\AA_1, Z_1)$
satisfies the Harder-Narasimhan property. Proposition~\ref{prop:inducestability} will both show the existence of HN filtrations more generally, and establish the support property for $(\AA_1, Z_1)$. 
\end{Ex}

\section{Inducing stability conditions}
\label{sec:inducestability}

The goal of this section is to enhance the method of the previous section, and show that when $\DD_2$ is generated by an exceptional collection, then we can use the same procedure to induce a stability condition on $\DD_1$ from a stability condition on $\DD$, such that the underlying hearts are related by the construction of Lemma~\ref{lem:induce} and Corollary~\ref{cor:induce}.

\subsection*{Result}
Let $E_1, \dots, E_m$ be an exceptional collection in a triangulated category
$\DD$.
We let $\DD_2 = \langle E_1, \dots, E_m \rangle$ be the category generated by the exceptional objects, and we write
\[ \DD = \langle \DD_1, \DD_2 \rangle \]
for the resulting semiorthogonal decomposition of $\DD$. We continue to write $S$ for the Serre functor on $\DD$. The main result of this section is the following:

\begin{Prop} \label{prop:inducestability}
Let $\sigma = (\cA, Z)$ be a weak stability condition on $\DD$ with the following properties for all $i = 1, \dots, m$:
\begin{enumerate}
\item $E_i \in \cA$,
\item $S(E_i) \in \cA[1]$, and
\item $Z(E_i) \neq 0$.
\end{enumerate}
Assume moreover that there are no objects $0 \neq F \in \cA_1:=\cA \cap \cD_1$ with
$Z(F) = 0$ (i.e., $Z_1:=Z|_{K(\cA_1)}$ is a stability function on $\cA_1$).
Then the pair $\sigma_1 = (\cA_1, Z_1)$ is a stability condition on $\DD_1$.
\end{Prop}

\subsection*{Inducing Harder-Narasimhan filtrations}
We start with some easy observations about the category $\cA_1$. If $F, G \in \cA_1$ are objects
with a morphism $f \colon F \to G$, then $f$ is injective (resp.~surjective) as a morphism in
$\cA_1$ if and only if it is injective (resp.~surjective) as a morphism in $\cA$. In other words,
the inclusion $\cA_1 \into \cA$ is an exact functor.

We will prove the existence of Harder-Narasimhan filtrations on $\cA_1$ in this setting:

\begin{Lem} \label{lem:induceHN}
Let $(\cA, Z)$ be a weak stability condition. Let $\cA_1 \subset \cA$ be an abelian
subcategory such that the inclusion functor is exact. Assume moreover that $Z$ restricted to $K(\cA_1)$ is a stability function.
Then Harder-Narasimhan filtrations exist in $\cA_1$ for the stability function $Z$.
\end{Lem}

We break up the proof of this result into several steps. We will deduce the existence of Harder-Narasimhan filtrations for objects $F \in \cA_1$ from the
existence of Harder-Narasimhan polygons and the concept of mass. We recall all the necessary
definitions and basic facts here; see \cite[Section~3]{Arend:short} for some context.

Let $\cB$ be an abelian category, and $Z \colon K(\cB) \to \C$ a weak stability function on
$\cB$ (see Definition~\ref{def:stabilityfunction}).

\def\HN{\mathop{\mathrm{HN}}}

\begin{Def}
The Harder-Narasimhan polygon $\HN^Z_\cB(F)$ of an object $F \in \cB$ is the convex hull in the complex plane of the set
\[
\stv{Z(A)}{A \subset F}.
\]
\end{Def}

We say that the Harder-Narasimhan polygon $\HN^Z_\cB(F)$ is \emph{finite polyhedral on the left} if
the intersection of $\HN^Z_\cB(F)$ with the closed half-plane to the left of the line through 0 and
$Z(F)$ is a polygon. The following Proposition is a variant of a well-known statement. 

\begin{Prop} \label{prop:HNpolygon}
If $F \in \cB$ admits a Harder-Narasimhan filtration, then $\HN^Z_\cB(F)$ is finite polyhedral
on the left. 
If moreover $Z$ is a stability function, then the converse holds.
\end{Prop}
\begin{proof}
When $Z$ is a stability function, then both directions are well-known, see e.g., \cite[Proposition
3.3]{Arend:short}. When $Z$ is only a weak stability function, the first statement is proved
easily with the same arguments as in the case of a stability function.
\end{proof}

The support property as defined here is equivalent to the one originally
appearing in \cite{Bridgeland:Stab}:

\begin{Prop}[{\cite[Lemma~A.4]{BMS}}] \label{prop:supportproperty2}
A pair $(\cB, Z)$ of an abelian category with a weak stability function satisfies the support
property if and only if there is a metric $\| \cdot \|$ on $\Lambda \otimes \R$, and a constant $C >
0$ such that for all
$Z$-semistable objects $F \in \cA$, we have 
\[ 
\abs{Z(F)} \ge C \| v(F) \|.
\]
\end{Prop}

We now return to the setting of Lemma~\ref{lem:induceHN}, and assume 
that $(\cA, Z)$ is a weak stability condition.

\begin{Def}
The \emph{mass} $m^Z_\cA(F)$ of an object $F \in \cA$ is the length of the boundary of the Harder-Narasimhan
polygon $\HN^Z_\cA(F)$ on the left, between $0$ and $Z(F)$; equivalently, if $F_i$ are the HN filtration steps of $F$, then $m^Z_\cA(F) = \sum \abs{Z(F_i/F_{i-1})}$. 
\end{Def}

The triangle inequality gives:

\begin{Prop} \label{prop:massboundsmetric}
With a metric $\| \cdot \|$ and $C > 0$ as in Proposition~\ref{prop:supportproperty2}, any object $F$ satisfies
\[
m^Z_\cB(F) \ge C \| v(F) \|.
\]
\end{Prop}

\begin{Lem} Let $F \in \cA_1$. Then the Harder-Narasimhan polygon $\HN^Z_{\cA_1}(F)$ (with respect to
subobjects in $\cA_1$) is finite polyhedral on the left.
\end{Lem}

\begin{proof}
For any subobject $A \subset F, A \in \cA_1$ we have $\HN^Z_\cA(A) \subset \HN^Z_\cA(F)$ by defintion. When we additionally assume $\mu_Z(A) > \mu_Z(A)$, a simple picture shows that the mass of $A$ is bounded (see also \cite[Lemma~5.5]{Arend:short}).

By Proposition~\ref{prop:massboundsmetric}, this implies a uniform bound for $\| v(A)\|$. Therefore, there are only finitely many possibilities for the class
$v(A)$ in $\Lambda$; therefore, the left-hand side of $\HN^Z_{\cA_1}(F)$ is the convex hull of finitely many points. 
\end{proof}

%

By Proposition~\ref{prop:HNpolygon}, this concludes the proof of Lemma~\ref{lem:induceHN}.

\subsection*{Inducing support property}
It remains to show that $(\cA_1, Z_1)$ has the support property.
By induction, we may assume for the rest of the section that $\DD_2$ is generated by a single exceptional object $E$.

\begin{Lem} \label{lem:4term}
Any $G \in \cA$ fits into a five term short exact sequence in $\cA$
\[ 0 \to \cH^{-1}_\cA(G_1) \to E \otimes \Hom(E, G) \to G \to \cH^0_{\cA} (G_1) \to E \otimes \Ext^1(E,G) \to 0 \]
with $G_1 \in \cD_1$.
\end{Lem}

\begin{proof}
This is the long exact cohomology sequence with respect to $\cA$ of the exact triangle
\[ E \otimes \RHom(E, G) \xrightarrow{\mathrm{ev}} G \to G_1 \]
where $G_1 = \cat{L}_E G \in \cD_1 = E^\perp$ is the mutation of $G$ at $E$.
\end{proof}

\begin{Lem}\label{lem:RestrSes} Consider a short exact sequence 
\begin{equation} \label{eq:AFBses} 0 \to A \to F \to B \to 0\end{equation}
in $\cA$ with $F \in \cA_1$. Then there is an exact sequence
\[ 0 \to \cH^{-1}_\cA(B_1) \to A_1 \to F \to \cH^0_\cA(B_1) \to 0 \] in $\cA_1$, together with  exact
sequences
\[
0 \to A \to A_1 \to E \otimes V \to 0, \quad
0 \to \cH^{-1}_\cA(B_1) \to E \otimes V \to B \to \cH^0_\cA(B_1) \to 0
\]
in $\cA$ where $V = \Hom(E, B)$.
\end{Lem}

\begin{proof}
As in Corollary \ref{cor:induce} we know that $\Ext^i(E, A) = \Ext^i(E, B) = 0 $ for $i \ge 2$.
The long exact Hom-sequence shows $\Hom(E, A) = 0 = \Ext^1(E, B)$ and
$V = \Hom(E, B) = \Ext^1(E, A)$. Therefore, the five term exact sequences 
of Lemma~\ref{lem:4term} for $A$ and $B$ behave as claimed. The long exact cohomology sequence of the projection of \eqref{eq:AFBses} to $\cD_1$ completes the proof.
\end{proof}

\begin{proof}[Proof of Proposition~\ref{prop:inducestability}]
We now fix a norm $\| \cdot \|$ on $\Lambda \otimes \R$ and a constant $C$ such that the weak stability conditions $(\cA, Z)$ satisfies the support property in the formulation given in Proposition~\ref{prop:supportproperty2}.
Assume that $F \in \cA_1$ is semistable within the category $\cA_1$, and consider a quotient $F \onto B$ as in Lemma~\ref{lem:RestrSes}. Since $\mu_Z(\cH^0_\cA(B_1)) \ge \mu_Z(F)$ (by semistability of $F$ in $\cA_1$), it follows from the last exact sequence of Lemma~\ref{lem:RestrSes} that $\mu_Z(B) \ge \min \{\mu_Z(F), \mu_Z^-(E)\}$.
In particular, if $\mu_Z(F) \le \mu_Z^-(E)$, then $F$ is also semistable as an object of $\cA$, and thus satisfies the support property with
respect to the same constant $C$. Otherwise, the left-hand side of the Harder-Narasimhan polygon $\HN^Z_{\cA}(F)$ is contained in the triangle with vertexes $0, z, Z(F)$ where $z$ is the point on the 
negative real line such that $Z(F) - z$ has slope corresponding to $\mu_Z^-(E)$ (see Figure~\ref{fig:phi} for the case where $E$ is semistable); therefore
$m^Z_\cA(F) \le \abs{z} + \abs{Z(F)-z}$. 

\begin{figure}[!tbp]
\begin{tikzpicture}[line cap=round,line join=round,>=triangle 45,x=1.0cm,y=1.0cm]
\clip(-6.26,-0.7) rectangle (5.42,3.3);
\fill[color=zzttqq,fill=zzttqq,fill opacity=0.1] (-4,1.5) -- (-3,0) -- (0,0) -- cycle;
\draw [shift={(-3,0)},color=qqwuqq,fill=qqwuqq,fill opacity=0.1] (0,0) -- (0:0.4) arc (0:123.69:0.4) -- cycle;
\draw [shift={(0,0)},color=qqwuqq,fill=qqwuqq,fill opacity=0.1] (0,0) -- (0:0.4) arc (0:123.69:0.4) -- cycle;
\draw [->] (0,0) -- (-1.5,2.25);
\draw [->] (0,0) -- (-4,1.5);
\draw (-4,1.5)-- (-3,0);
\draw (-4,1.5)-- (-1,1.5);
\draw [color=zzttqq] (-4,1.5)-- (-3,0);
\draw [color=zzttqq] (-3,0)-- (0,0);
\draw [color=zzttqq] (0,0)-- (-4,1.5);
\draw [->] (-5,0) -- (3,0);
\draw (2.9,0.14) node[anchor=north west] {$\Re$};
\draw [->] (0,0) -- (0,3);
\draw (0.02,3) node[anchor=north west] {$\Im$};
\begin{scriptsize}
\draw[color=tttttt] (-2,2.25) node {$Z(E)$};
\draw[color=tttttt] (-4.54,1.46) node {$Z(F)$};
\fill [color=tttttt] (-3,0) circle (1.5pt);
\draw[color=tttttt] (-3.0,-0.36) node {$z$};
\draw[color=tttttt] (-3.7,0.65) node {$a_1$};
\draw[color=tttttt] (-1.5,-0.26) node {$a_2$};
\draw[color=tttttt] (-1.79,0.89) node {$b$};
\draw[color=qqwuqq] (-2.54,0.32) node {$\phi$};
\draw[color=qqwuqq] (0.46,0.32) node {$\phi$};
\fill [color=uququq] (0,0) circle (1.5pt);
\draw[color=uququq] (0.00,-0.36) node {$0$};
\end{scriptsize}
\end{tikzpicture}
\caption{The left-hand side of $\HN^Z_{\cA}(F)$ is contained in the triangle $0, z, Z(F)$.}
\label{fig:phi}
\end{figure}

By Lemma~\ref{lem:8thgrade},
there is a constant $D > 0$, depending only on $\mu_Z^-(E)$, such that
\[
m^Z_\cA(F) \le \abs{z} + \abs{Z(F) - z} \le D \abs{Z(F)}.
\]
Combined with Proposition~\ref{prop:massboundsmetric}, it follows that $\| v(F) \| \le \frac DC
\abs{Z(F)}$ for all semistable objects $F \in \cA_1$, i.e., the pair $(\cA_1, Z)$ satisfies the
support property.
\end{proof}

\begin{Lem} \label{lem:8thgrade}
Let $0 < \phi < \pi$ be a fixed angle. Then there is a constant $D > 0$ such that for all triangles with
one angle given by $\phi$, and with adjacent side lengths $a_1, a_2$ and $b$ the side length opposite
of $\phi$, we have $a_1 + a_2 \le D b$.
\end{Lem}

\begin{proof}
With $D = \sqrt{\frac{2}{1+\cos \phi}}$, this follows from
\[ b^2 = a_1^2 + a_2^2 - 2 a_1 a_2 \cos \phi = \frac{1 + \cos \phi}{2}(a_1 + a_2)^2 + \frac{1 - \cos \phi}{2}(a_1 - a_2)^2.\qedhere \]
\end{proof}

We note the following observation made in the proof:
\begin{Rem}\label{rem:SstableA1IffSstableA}
An object $F \in \cA_1$ with $\mu_Z(F) \le \mu_Z^-(E)$ is
semistable as an object in $\cA_1$ if and only if it is semistable as an object in $\cA$. More generally, any object $F \in \cA_1$ satisfies
$\mu_{Z}^-(F) \ge \min \{ \mu_{Z_1}^-(F), \mu_Z^-(E) \}$, where 
$\mu_Z^-(F)$ is computed by the HN filtration of $F$ as an object of $\cA$, and $\mu_{Z_1}^-(F)$ by HN filtration in $\cA_1$.
\end{Rem}

When combined with the results in \cite{CollinsPolishchuk:Gluing}, Proposition~\ref{prop:inducestability} has the following consequence:
\begin{Prop} \label{prop:glue}
Given the same assumptions as in Proposition~\ref{prop:inducestability}, there exists a stability conditions
$\sigma' = (\cA', Z')$ on $\cD$ with 
\[ \cA' = \langle \cA_1, E_1[1], E_2[2], \dots, E_m[m] \rangle \]
and $Z'$ determined by 
\[ Z'|_{\cD_1} = Z|_{\cD_1}, \quad Z'(E_i) = (-1)^{i+1} \]
for $i = 1, \dots, m$.
\end{Prop}

\begin{proof}
Let $\cA_2 := \langle E_1[1], \dots, E_m[m] \rangle$; by \cite[Section~2]{CollinsPolishchuk:Gluing}, $\cA_2$ is the heart of a bounded t-structure in $\cD_2$, that together with $\cA_1$ produce a heart $\cA' \subset \cD$ with description as in the claim; moreover, the pair $(\cA_2, \cA_1)$ is a torsion pair in $\cA$. 

By construction, the objects $\cA_2 \subset \cA'$ have maximal slope; thus, the existence of the torsion pair, combined with the existence of HN filtrations in $\cA_1$ with respect to $Z_1$, gives HN filtrations in $\cA'$ (similar to the proof of \cite[Proposition~3.3]{CollinsPolishchuk:Gluing}). Finally, $\sigma'$-stable objects are either $\sigma_1$-stable objects of $\cA_1$, or of the form $E_i[i]$. By Proposition~\ref{prop:supportproperty2}, this shows that the support property for $\sigma_1$ implies the support property for $\sigma'$.
\end{proof}

In other words, a \emph{weak} stability condition on $\cD$ with the assumptions of Proposition~\ref{prop:inducestability}  produces an actual stability condition on $\cD$---at the cost of making the associated heart more implicit.

\section{Fano threefolds} \label{sec:3folds}

In this section we apply the results of the previous sections in order to construct stability conditions on the Kuznetsov component of all but one deformation type of Fano threefolds of Picard rank one.
We assume the base field is algebraically closed and has characteristic either zero or sufficiently large.\footnote{Our own constructions, in Theorems~\ref{thm:index2}, \ref{thm:1object} and \ref{thm:index1}, work for arbitrary characteristic; this is due to Remark~\ref{rem:BGpositive}. Where we refer to prior explicit descriptions of Kuznetsov components, this typically require the characteristic to be sufficiently large.}

\subsection*{Review} \label{sec:3foldsoverview}
We begin by reviewing Kuznetsov's semiorthogonal decompositions for the Fano threefolds appearing in Theorem~\ref{thm:main1}. We follow the overview in \cite{Kuz:Fano3folds} of semiorthogonal decompositions of Fano threefolds of Picard rank one.

Deformation types of Fano threefolds of Picard rank one are classified by their index $i_X$ and their degree $d$; this classification holds similarly over algebraically closed fields of characteristic $p$ by \cite{Shepherd-Barron:Fano3foldscharp}. The index
$i_X \in \{1, 2, 3, 4\}$ is given by 
 $K_X = - i_X \, H$, where $K_X$ is the canonical divisor, and $H$ is a primitive ample divisor; the degree is given by $d = H^3$.
 
We ignore the cases of index four (projective space) and three (quadric), as they admit full exceptional collections.
In the case of index two, there are five deformation types classified by $1 \le d \le 5$. Our definition of the Kuznetsov component is straightforward in this case:
\begin{Def}[\cite{Kuz:Fano3folds}] \label{def:KuXindex2}
Let $X$ be a Fano threefold of index two over an algebraically closed field, and let $H=-\frac12 K_X$.
The Kuznetsov component $\Ku(X)$ is defined by the semiorthogonal decomposition
\[ \Db(X) = \langle \Ku(X), \OO_X, \OO_X(H) \rangle. \]
\end{Def}

For index one and Picard rank one, the degree is even and can thus by written as $H^3 = 2g-2$ in terms of its genus
$g$. There are 10 deformation types, corresponding to $2 \le g \le 12$, $g \neq 11$. The definition of Kuznetsov components depends on the genus, and relies on the existence of exceptional vector bundles due to Mukai. 

\begin{Thm}[\cite{Mukai:Fano3folds,Kuz:Fano3folds}]
\label{thm:Mukai}
Let $X$ be a Fano threefold of Picard rank one, index one, and even genus $g=2s>2$ over any algebraically closed field. Then there exists a stable vector bundle $\cE_2$ on $X$ of rank $2$, with
$c_1(\cE_2) = -H$ and $\ch_2(\cE_2) = (s-2) L$, where $L$ is the class of a line on $X$.
\end{Thm}

\begin{PropDef}[{\cite{Kuz:Fano3folds}}] \label{propdef:excpair}
Let $X$ be a Fano threefold of Picard rank one, index one, and even genus $g>4$ over any algebraically closed field.
Then the pair $(\cE_2, \OO_X)$ is exceptional, and the Kuznetsov component of $X$ is defined by the semiorthogonal decomposition
\[
\Db(X) = \langle \Ku(X), \cE_2, \OO_X \rangle.
\]
\end{PropDef}

The existence of these vector bundles follow from the classification of Fano threefolds. However, as the argument sketched in \cite{Mukai:Fano3folds} seems incomplete\footnote{The assumptions for Fujita's extension theorem in \cite{Fujita:extendampledivisor} can never be satisfied on a surface.}, and as they are stated only for $\Char(k) = 0$ by both Mukai and Kuznetsov, we provide here a proof that is also independent of the characteristic.
\begin{proof}[Proof of Theorem~\ref{thm:Mukai} and Proposition and Definition~\ref{propdef:excpair}]
Consider a generic pencil of hyperplane sections in $\abs{H}$, and let $\widetilde{X} \to \P^1$ be the induced K3 fibration of the blowup of $X$ at its base locus $C_g$, a smooth curve of genus $g$.
Each K3 fiber $S$ admits a unique stable spherical vector bundle $\cE_S$ with Mukai vector $(2, -H, s)$; in other words, the corresponding relative coarse moduli space is isomorphic to $\P^1$. 
By Tsen's theorem, the Brauer group of $\P^1$ vanishes, see \cite[Theorem~6.2.8 and Proposition~6.2.3]{Gille-Szamuely:centralsimple}; therefore, the obstruction to the existence of a universal family vanishes, i.e., there is a vector bundle $\cE_{\widetilde{X}}$ whose restrictions are the given spherical vector bundles on the fibers. We claim that the restriction
of $\cE_{\widetilde{X}}$ to the exceptional divisor $C_g \times \P^1$ is of the form
$\cE_{C_g} \boxtimes \cO_{\P^1}(k)$: indeed, this can be checked infinitesimally from the fact that the unique deformation class $\Ext^1_{\wX}(i_{S*}\cE_S, i_{S*}\cE_S)$ is induced by the unique deformation class
of $i_{S*}\cO_S$, which induces the trivial deformation after restriction to $C_g$. 

After tensoring $\cE_{\wX}$ with the pull-back of $\cO_{\P^1}(-k)$ it is thus the pull-back of a vector bundle $\cE_X$ on $X$, whose restriction to the hyperplane section $S$ is still given by $\cE_S$. This proves Theorem~\ref{thm:Mukai}.

For the Proposition, since $s > 1$, Langer's version of Bogomolov's restriction theorem (which applies with no correction terms for K3 surfaces in characteristic $p$, \cite[Theorem~5.2]{Langer:positive}) shows that $\cE_{C_g}$ is \emph{stable}.
Therefore $\Hom(\cE_S, \cE_{C_g}) = k$, and the long exact Hom-sequence implies $\Ext^1(\cE_S, \cE_S(-1)) = \Ext^1(\cE_S, \cE_S)$, which vanishes by construction.
From this one deduces $\Ext^1(\cE_S, \cE_S(n)) = 0$ for all $n \in \Z$, using again the long exact Hom-sequence and Serre duality.
The analogous long exact Hom-sequence on $X$ then implies $\Ext^1(\cE_X, \cE_X) = \Ext^1(\cE_X, \cE_X(-1)) = \dots = \Ext^1(\cE_X, \cE_X(-n))$ for all $n > 0$, which of course vanishes for $n \gg 0$. Therefore, $0 = \Ext^1(\cE_X, \cE_X) = \Ext^1(\cE_X, \cE_X(-1)) = \Ext^2(\cE_X, \cE_X)$.
Since the required $\Ext^3$-vanishing is immediate from Serre duality, this finally shows that $\cE_X$ is exceptional.

That $(\cO_X, \cE_X)$ is an exceptional pair can similarly be deduced from the corresponding cohomology vanishing of $\cE_S$:
the key initial observation is  $H^1(\cE_S) = 0$, which holds as otherwise the corresponding extension $\cE_S \into F \onto \cO_S$ would be a stable vector bundle with $v(F)^2 < -2$, a contradiction. 
\end{proof}

A previous version of this paper incorrectly claimed Proposition~\ref{propdef:excpair} also for genus 4, the case of a complete intersection of a quadric and a cubic in $\P^5$. However, in this case there exist either exactly two (when the quadric is smooth) or no (when the quadric is singular) exceptional bundles of rank two.

The remaining cases of index one are covered by the following two definitions. 
\begin{PropDef}[{\cite{Mukai:Fano3folds,Kuznetsov:Hyperplane}}]
Let $X$ be a Fano threefold of Picard rank one, index one, and genus $7$ (resp. $9$).
Then there exists a rank $5$ (resp. $3$) vector bundle $\cE_5$ (resp. $\cE_3$) such that the pair $(\cE_5, \OO_X)$ (resp. $(\cE_3, \OO_X)$) is exceptional.
The Kuznetsov component of $X$ is defined to be its right orthogonal component.
\end{PropDef}

\begin{Def}
Let $X$ be a Fano threefold of Picard rank one, index one, and genus $2$, $3$, $4$ or $5$ over any algebraically closed field.
Then the Kuznetsov component of $X$ is the right orthogonal to $\gen{\OO_X}$.
\end{Def}

Let $X$ be a Fano threefold.
We consider the lattice $\Lambda_H^2\cong\Z^3$ as in Example~\ref{ex:slopestability} and the natural map $v_H^2\colon K(X)\to\Lambda_H^2$.
For an admissible subcategory $\DD\subset \Db(X)$, we denote by 
\begin{equation}\label{eqn:LambdaD}
\Lambda_{H,\DD}^2:=\mathrm{Im}\left(K(\DD) \to K(X) \to \Lambda_H^2\right).
\end{equation}
the image of the composite map and, by abuse of notation, the induced morphism $v_H^2\colon K(\DD)\to \Lambda_{H,\DD}^2$.

\subsection*{Result and context}

The goal of this section is to prove Theorem~\ref{thm:main1}. 
We will split the statement into various cases and we summarize the results in the following two tables of Fano threefolds of Picard rank one (if the base field has characteristic either zero or sufficiently large):

\begin{center}
\begin{tabular}{r|l|l|c}
\hline
\multicolumn{4}{ c }{$\rho_X=1$ \& $i_X=2$}\\
\hline
$\deg$
&\multicolumn{2}{ |c| }{Semiorthogonal decomposition} &\begin{minipage}{3.5cm} $\exists$ stability conditions
\end{minipage}\\
\hline
5 &$\Db(Y_5)=\gen{\FF_2(-H),\OO(-H),\FF_2,\OO}$ &{\cite{orlov:Y5}}&expl.~descr.\protect\footnotemark~ or Thm.~\ref{thm:index2}\\
4 &$\Db(Y_4)=\gen{\Db(C_2),\OO(-H),\OO}$~\protect\footnotemark
\;\; &{\cite[Thm.~2.9]{BondalOrlov:Main}}
&expl.~descr.\footnoteref{expldescr} or Thm.~\ref{thm:index2}\\
3 &$\Db(Y_3)=\gen{\Ku(Y_3),\OO(-H),\OO}$&&{\cite{BMMS:Cubics}} or Thm.~\ref{thm:index2}\\
2 &$\Db(Y_2)=\gen{\Ku(Y_2),\OO(-H),\OO}$
&&Thm.~\ref{thm:index2}\\
1 &$\Db(Y_1)=\gen{\Ku(Y_1),\OO(-H),\OO}$&&Thm.~\ref{thm:index2}\\
\hline
\end{tabular}
\end{center}
\addtocounter{footnote}{-1}
{\footnotetext{\label{expldescr}By \emph{explicit description (expl.~descr.)}, we mean that the explicit description of the Kuznetsov component (according to the reference in the middle column), combined with the construction of stability conditions for curves and categories of quiver representations, implies the existence of stability conditions, for characteristic zero or sufficiently large. }}
\addtocounter{footnote}{1}
{\footnotetext{\label{Cg}We denote by $C_g$ a smooth genus $g$ curve.}}

\medskip

\begin{center}
\begin{tabular}{r|l|l|c}
\hline
\multicolumn{4}{ c }{$\rho_X=1$ \& $i_X=1$}\\
\hline
$g_X$ &\multicolumn{2}{ |c| }{Semiorthogonal decomposition} &\begin{minipage}{3.5cm} $\exists$ stability conditions
\end{minipage}\\
\hline
12&$\Db(X_{22})=\gen{\EE_4,\EE_3,\EE_2,\OO}$ &{\cite[Thm.~4.1]{Kuz:Fano3folds}}&expl.~descr.\footnoteref{expldescr}
or Thm.~\ref{thm:index1}\\
10&$\Db(X_{18})=\gen{\Db(C_2),\EE_2,\OO}$~\footnoteref{Cg} &
{\cite[\S6.4]{Kuznetsov:Hyperplane}}&expl.~descr.\footnoteref{expldescr} or Thm.~\ref{thm:index1}\\
9&$\Db(X_{16})=\gen{\Db(C_3),\EE_3,\OO}$~\footnoteref{Cg} &{\cite[\S6.3]{Kuznetsov:Hyperplane}}&expl.~descr.\footnoteref{expldescr}\\
8&$\Db(X_{14})=\gen{\Ku(X_{14}),\EE_2,\OO}$&
{\cite{Kuz:V14}}&\cite{BMMS:Cubics} or Thm.~\ref{thm:index1}\\
7&$\Db(X_{12})=\gen{\Db(C_7),\cE_5,\OO}$~\footnoteref{Cg} &{\cite[\S6.2]{Kuznetsov:Hyperplane}}&expl.~descr.\footnoteref{expldescr}\\
6&$\Db(X_{10})=\gen{\Ku(X_{10}),\EE_2,\OO}$ &{\cite[Lem.~3.6]{Kuz:Fano3folds}}&Thm.~\ref{thm:index1}\\
5&$\Db(X_{8})=\gen{ \Ku(X_8),\OO}$ &&Thm.~\ref{thm:1object}\\
4&$\Db(X_{6})=\gen{\Ku(X_{6}),\OO}$&&Thm.~\ref{thm:1object} \\
3&$\Db(X_{4})=\gen{ \Ku(X_4),\OO}$ &&Thm.~\ref{thm:1object}\\
2&$\Db(X_{2})=\gen{ \Ku(X_2),\OO}$ &&Thm.~\ref{thm:1object}\\
\hline
\end{tabular}
\end{center}

\bigskip

For Fano threefolds of Picard rank one, there is a conjectural relation between the index two case of degree $d\geq 2$ and the index one and genus $2d+2$, due to Kuznetsov (see \cite[Conjecture 3.7]{Kuz:Fano3folds}).
This is proved in the case $d=3,4,5$ and asserts there the equivalence between the respective Kuznetsov components.
In fact, our result may turn useful in understanding this conjecture, as explained in Example~\ref{ex:d4} below.
For $d=2$ the conjecture in \emph{loc.~cit.~} needs to be modified as remarked in \cite[Theorem~7.2]{marcellotabuada:NCmotives}.
A modified version does not ask the maps to be dominant, so the two Kuznetsov components would only be deformation equivalent \cite{Kuz:lectures}.


\begin{Ex}[$d=4$] \label{ex:d4}
The space of Bridgeland stability conditions on $\Ku(X_{18})\cong\Db(C_2)$ consists of a unique orbit, with respect to the $\widetilde{\GL_2^+}(\R)$-action, containing $(\Coh(C_2),\mathfrak{i}\rk - \deg)$ (see \cite{Macri:curves}).
In particular, the stability condition $\sigma$ constructed in Theorem~\ref{thm:main1} lies in the same orbit.

The curve $C_2$ can be reconstructed as moduli space of skyscraper sheaves, which are stable with respect to any stability condition on $\Db(C_2)$.
Hence, $C_2$ can be identified with the moduli space of $\sigma$-stable objects in $\Ku(X_{18})$ with Chern character $3-2H+9L-\frac12 \mathrm{pt}$, which is the image of the Chern character of a skyscraper sheaf via the inclusion $\Knum(C_2)\cong \Knum(\Ku(X_{18}))\subset\Knum(X_{18})$ (see \cite[Proposition~3.9]{Kuz:Fano3folds}).
The equivalence $\Ku(X_{18})\cong\Db(C_2)$ can then be reinterpreted as the one coming from the universal family on such a moduli space.

The Fano threefold $Y_4$ can then be reconstructed as the moduli space of rank $2$ vector bundles on $C_2$ with fixed determinant of odd degree (see \cite[Theorem~1]{newstead:stableRk2} or \cite[Theorem~4]{NarasimhanRamanan:moduli}). The equivalence $\Db(C_2)\cong\Ku(Y_4)$ can again be reinterpreted as the one coming from the universal family (see \cite[Theorem~2.7]{BondalOrlov:Main} and \cite[Remark~5]{narasimhan:derived}).
\end{Ex}

\subsection*{Proof of Theorem~\ref{thm:main1}, case index one and low genus.}
We divide the proof of the Theorem~\ref{thm:main1} in three cases, according to the index and the genus.
We begin with the easiest case. We will prove the following more general statement, which holds for all Fano threefolds, of any Picard number and index.

\begin{Thm}\label{thm:1object}
Let $X$ be a Fano threefold over any algebraically closed field.
Consider the semiorthogonal decomposition $\Db(X) = \langle \OO_X^\perp, \OO_X \rangle$.
Then $\OO_X^\perp$ has a Bridgeland stability condition with respect to the lattice $\Lambda_H^2\cong\Z^3$.\footnote{Note that in this case the lattice $\Lambda_{H,\cO_X^{\perp}}^2$ defined in \eqref{eqn:LambdaD} coincides with $\Lambda_H^2$.}
\end{Thm}

In particular, if $X$ has index $1$ and genus $g\in\set{2,3,4,5}$, then $\Ku(X)$ has a stability condition. 

\begin{proof}
We want to apply Proposition~\ref{prop:inducestability} to the exceptional collection of length one given by $\cO_X$.

For $\Coh(X)$, we have $\cO_X \in \Coh(X)$, but $S(\cO_X) = \cO_X(K_X)[3] \in \Coh(X)[3]$. We need a weak stability condition whose heart still contains $\cO_X$, but also $\cO_X(K_X)[2] = \cO_X(-i_X H)[2]$ instead of $\cO_X(K_X)$, i.e., we will need to tilt $\Coh(X)$ twice.

Let $\beta = - \frac 12 i_X$. Then clearly $\cO_X, \cO_X(K_X)[1] \in \Coh_H^\beta(X)$. Now consider the weak stability condition $\sigma_{\alpha, \beta} = \left(\Coh_H^\beta(X), Z_{\alpha, \beta}\right)$ of Proposition~\ref{prop:tiltstability}, for
$\beta$ as above and for $\alpha$ sufficiently small. By Proposition~\ref{prop:Delta0stable}, both $\cO_X$ and $\cO_X(K_X)[1]$
are $\sigma_{\alpha, \beta}$-stable. Since $\alpha$ is sufficiently small, we have
\begin{eqnarray}
\Re Z_{\alpha, \beta}(\cO_X) = \frac {\alpha^2}2 H^3 - H\cdot \frac 12 \left(\frac{K_X}{2}\right)^2
& < 0 < &
\Re Z_{\alpha, \beta}(\cO_X(K_X)[1]) = -\frac{\alpha^2}2 H^3 + H \cdot \frac 12 \left(\frac{K_X}{2}\right)^2 \nonumber \\
\text{and so} \quad \mu_{\alpha, \beta}(\cO_X(K_X)[1])
& < 0 < &
\mu_{\alpha, \beta}(\cO_X). \label{ineq:mualphasmall}
\end{eqnarray}
Therefore, if we tilt a second time to obtain the weak stability condition $\sigma_{\alpha, \beta}^0$ of Proposition~\ref{prop:tiltedtiltstability}, then its heart $\Coh_{\alpha, \beta}^0(X)$ contains both $\cO_X$ and $\cO_X(K_X)[2]$. 

By Lemma~\ref{lem:TiltGoodIsAlmostGood}, we have $\Coh_{\alpha, \beta}^0(X)_0 = \Coh^\beta(X)_0$, which is the category of zero-dimensional torsion sheaves; as any such sheaf has global sections, the intersection $\OO_X^\perp \cap \Coh_{\alpha, \beta}^0(X)_0$ is empty.
Hence all assumptions of Proposition~\ref{prop:inducestability} are satisfied, and we obtain a stability condition on $\cO_X^\perp$.
\end{proof}

\subsection*{Proof of Theorem~\ref{thm:main1}, case index two.}
In this section we prove the following case of Theorem~\ref{thm:main1} without the Picard rank one assumption.

\begin{Thm}\label{thm:index2}
Let $X$ be a Fano threefold of index two over any algebraically closed field.
Then the Kuznetsov component $\Ku(X)$ has a Bridgeland stability condition with respect to $\Lambda_{H,\Ku(X)}^2\cong\Z^2$.
\end{Thm}

\begin{proof}
The proof goes along the exact same lines as the proof of Theorem~\ref{thm:1object}, by applying Proposition~\ref{prop:inducestability} to the exceptional collection $\langle \cO_X, \cO_X(H) \rangle$. We set $\beta = -\frac 12$;
then again $G, G(K_X)[1] \in \Coh_H^\beta(X)$ for $G \in \{\cO_X, \cO_X(H)\}$.
The exact same computation leading to \eqref{ineq:mualphasmall} shows that, for $\alpha$ small, we have
\[
\mu_{\alpha, -\frac 12}(G(K_X)[1]) < 0 < \mu_{\alpha, -\frac 12}(G)
\]
for the same $G$. With the same arguments as before, we apply Proposition~\ref{prop:inducestability} to show that the weak stability condition
$\sigma_{\alpha, -\frac 12}^0$ of Proposition~\ref{prop:tiltedtiltstability} induces a stability condition on $\Ku(X)$.
\end{proof}

\subsection*{Proof of Theorem~\ref{thm:main1}, case index one and high genus.}
The remaining non-trivial cases of Theorem~\ref{thm:main1} are covered by the following result:

\begin{Thm}\label{thm:index1}
Let $X$ be a Fano threefold of Picard rank one, index one, and even genus $g \ge 6$ over any algebraically closed field.
Then the Kuznetsov component $\Ku(X)$ has a Bridgeland stability condition with respect to the lattice $\Lambda_{H,\Ku(X)}^2\cong\Z^2$.
\end{Thm}

We want to apply the same proof as in the previous cases. Note that $\cE_2$ is slope-stable with $\mu_H(\cE_2) = -\frac 12$, whereas
$S(\cO_X) = \cO_X(-H)[3]$, $S(\cE_2) = \cE_2(-H)[3]$ are shifts of slope-stable sheaves of slope $-1$ and $-\frac 32$, respectively. Therefore, the first step works exactly as before: for any $\beta$ with $-1 < \beta < -\frac 12$, the 
abelian category $\Coh^\beta(X)$ contains $\cO_X, \cE_2$ as well as $\cO_X(-H)[1]$ and $\cE_2(-H)[1]$. 

To continue as before, we need to show tilt-stability of $\cE_2$; the corresponding statement was automatic in the previous cases by Proposition~\ref{prop:Delta0stable}. We start with an auxiliary observation.
Recall from Section~\ref{sec:tilt} that we have a quadratic form $\Delta_H$ on $\Lambda_H^2\otimes\R\cong\R^3$.

\begin{Lem} \label{lem:E2location}
Consider the tangent planes to the quadric $\Delta_H = 0$ in $\R^3$ at $v_H^2(\cO_X)$ and $v_H^2(\cO_X(-H))$, and the corresponding open half-spaces containing $(0, 0, 1)$.  
If $g \ge 6$ is an even genus, then $v_H^2(\cE_2)$ lies in the intersection of these two half-spaces.
\end{Lem}

In terms of visualisations of the negative cone via cross-sections of $\R^3$ as in Figure~\ref{fig:tiltwalls}, this means that $\EE_2$ lies above the tangent lines at $\OO_X(-H)$ and $\OO_X$ in Figure~\ref{fig:E2location}.

\begin{proof}
By the symmetry of $\R^3$ induced by $\lHom(\blank, \OO(-H))$, which leaves the quadratic form $\Delta_H$ invariant, the intersection of these two tangent planes is also contained in the plane
$\mu_H(\blank) = -\frac 12$ containing $v_H^2(\cE_2)$; therefore, it is enough to prove the claim for one of the two planes.

The tangent plane to $\Delta_H=0$ at $v_H^2(\cO_X)$ is given by $H\ch_2 (\blank) = 0$, and thus the claim follows from $H\ch_2(\cE_2) = H\left(\frac g2 - 2\right) L > 0$, see Theorem~\ref{thm:Mukai}.
\end{proof}

\begin{figure}
\centering
\hspace{0cm}
\begin{minipage}{.5\textwidth}
\definecolor{uququq}{rgb}{0.25,0.25,0.25}
\definecolor{ffqqqq}{rgb}{1,0,0}
\begin{tikzpicture}[line cap=round,line join=round,>=triangle 45,x=0.55cm,y=0.55cm]
\clip(-7.4,-7.) rectangle (6.7,6.);
\draw[fill=black,fill opacity=0.05](0,0) circle (2.75cm);
\begin{footnotesize}
\draw (-5.1,5.) node[anchor=north west] {$\Delta_H=0$};
\draw (-1.,1.61) node[anchor=north west] {$\mu_H=-\frac12$};
\draw (2.17,-4.5) node[anchor=north west] {$\mathcal{O}_X$};
\draw (-4.76,-1.53) node[anchor=north east] {$\mathcal{O}_X(-H)$};
\draw [color=ffqqqq](-2.6,-5.39) node[anchor=north west] {$\mathcal{E}_2$};
\draw (0,5) node[anchor=south west] {$(0,0,1)$};
\end{footnotesize}
\draw [dotted] (-3,-7)-- (0,5);
\draw [dash pattern=on 4pt off 4pt] (-6.87,5.03)-- (-3,-7);
\draw [dash pattern=on 4pt off 4pt] (-3,-7)-- (8.24,-1.57);
\begin{scriptsize}
\fill [color=black] (0,5) circle (1.5pt);
\fill [color=uququq] (-4.76,-1.53) circle (1.5pt);
\fill [color=uququq] (2.17,-4.5) circle (1.5pt);
\fill [color=ffqqqq] (-2.6,-5.39) circle (1.5pt);
\end{scriptsize}
\end{tikzpicture}
 \captionof{figure}{The location of $v_H(\cE_2)$}
 \label{fig:E2location}
\end{minipage}%
\begin{minipage}{.5\textwidth}
 \centering
\definecolor{xdxdff}{rgb}{0.49,0.49,1}
\definecolor{uququq}{rgb}{0.25,0.25,0.25}
\definecolor{ffqqqq}{rgb}{1,0,0}
\definecolor{zzttqq}{rgb}{0.6,0.2,0}
\begin{tikzpicture}[line cap=round,line join=round,>=triangle 45,x=0.6cm,y=0.6cm]
\clip(-7.2,-6.51) rectangle (1.5,6.);
\fill[color=zzttqq,fill=zzttqq,fill opacity=0.1] (-2.6,-5.39) -- (-4.76,-1.53) -- (0,5) -- cycle;
\draw[fill=black,fill opacity=0.05](0,0) circle (3cm);
\begin{footnotesize}
\draw (-4.81,4.78) node[anchor=north west] {$\Delta_H=0$};
\draw (-3.49,-0.6) node[anchor=north west] {$\mu_H=-1+\epsilon$};
\draw (2.15,-4.28) node[anchor=north west] {$\mathcal{O}_X$};
\draw [color=zzttqq](-4.76,-1.53) node[anchor=north east] {$\mathcal{O}_X(-H)$};
\draw [color=ffqqqq](-2.6,-5.39) node[anchor=north west] {$\mathcal{E}_2$};
\draw [color=zzttqq](0,5) node[anchor=south] {$(0,0,1)$};
\draw [color=zzttqq](-4.27,2.89) node[anchor=north west] {$\mu_H=-1$};
\end{footnotesize}
\draw [dotted] (-3,-7)-- (0,5);
\draw [dash pattern=on 3pt off 3pt] (-6.87,5.03)-- (-3,-7);
\draw [dash pattern=on 3pt off 3pt] (-3,-7)-- (8.24,-1.57);
\draw [color=zzttqq] (-2.6,-5.39)-- (-4.76,-1.53);
\draw [color=zzttqq] (-4.76,-1.53)-- (0,5);
\draw [color=zzttqq] (0,5)-- (-2.6,-5.39);
\draw (-4.29,-2.57)-- (0,5);
\begin{scriptsize}
\fill [color=black] (0,5) circle (1.5pt);
\fill [color=uququq] (-4.76,-1.53) circle (1.5pt);
\fill [color=uququq] (2.17,-4.5) circle (1.5pt);
\fill [color=ffqqqq] (-2.6,-5.39) circle (1.5pt);
\end{scriptsize}
\end{tikzpicture}
\captionof{figure}{(Lack of) Walls for $\cE_2$}
 \label{fig:E2WALLS}
\end{minipage}
\end{figure}

\begin{Lem}
For $\epsilon>0$ small, the objects $\cE_2$ and $\cE_2(-H)[1]$ are $\sigma_{\alpha, -1 + \epsilon}$-stable for all $\alpha > 0$.
\end{Lem}

\begin{proof}
We first observe that both objects are $\sigma_{\alpha, -1}$-stable for all $\alpha > 0$. Indeed, for $\alpha \gg 0$, this follows from slope-stability of $\cE_2$ and Proposition~\ref{prop:largevolume}. Moreover, since
\[
\Im {Z_{\alpha, -1}}(\cE_2) = H^2 \ch_1^{-1}(\cE_2) = H^3 = -H^2 \ch_1^{-1}(\cE_2(-H)) = \Im {Z_{\alpha, -1}}(\cE_2(-H)[1]),
\]
and since $\Im Z_{\alpha, -1}(F) \in \Z_{\ge 0}\cdot H^3$ for all objects $F \in \Coh^{-1}(X)$, neither object can be strictly $\sigma_{\alpha, -1}$-semistable for any $\alpha > 0$; by the existence of the wall-and-chamber structure for tilt-stability, this means they must be $\sigma_{\alpha, -1}$-stable for all $\alpha > 0$. 

In the case of $\EE_2(-H)[1]$, we apply the previous Lemma~\ref{lem:E2location} analogously to $\EE_2(-H)$ instead of $\EE_2$: the projection of $v_H^2(\EE_2(-H)[1])$ in Figure~\ref{fig:E2WALLS} lies above the dotted tangent line to $\Delta_H = 0$ at $v_H^2(\OO(-H))$, but further to the left on the line $\mu = -\frac 32$. Therefore, any wall intersecting the line segment of slope $\mu_H = -1+\epsilon$ would also intersect the line segment $\mu_H = -1$.

Now consider the location of possible walls for $\sigma_{\alpha, \beta}$-semistability of $\cE_2$, as in Figure~\ref{fig:E2WALLS}; in this picture, they are given as the intersection of lines through $v_H^2(\cE_2)$ with the interior of the negative cone
$\Delta_H(\blank) < 0$. By the argument in the previous paragraph, no such wall can be in the interior of the triangle with vertices $v_H^2(\cE_2)$, $v_H^2(\cO_X(-H))$ and $(0, 0, 1)$. By the local finiteness of walls, it suffices to prove that the line segment connecting $v_H^2(\cE_2)$ and $v_H^2(\cO_X(-H))$ is not a wall: any wall strictly below that one would not intersect the line segment corresponding to $\beta = -1 + \epsilon$ for $\epsilon \ll 1$. 

Assume otherwise. Then there is a short exact sequence 
$A \into \cE_2 \onto B$ such that when $(\alpha, \beta)$ lies on the wall, then
$Z_{\alpha, \beta}(A)$ and $Z_{\alpha, \beta}(B)$ lie on the open line segment
connecting $0$ and $Z_{\alpha, \beta}(E)$ in the complex plane. By continuity, this still holds at the end point $(\alpha, \beta) = (0, -1)$; with the same integrality argument as before, we conclude either $Z_{0, -1}(A) = 0$ or
$Z_{0, -1}(B) = 0$; in particular, $v_H^2(A)$ or $v_H^2(B)$ are proportional to
$v_H^2(\cO_X(-H))$, respectively. By Proposition~\ref{prop:Delta0stable}, we must have
$A \cong \cO_X(-H)[1]$ or $B \cong \cO_X(-H)[1]$. But both of these are impossible: we clearly have $\Hom(\cO_X(-H)[1], \cE_2) = 0$, and, by Serre duality and the fact that $(\cO_X, \cE_2)$ is an exceptional pair, also
$\Hom(\cE_2, \cO_X(-H)[1])= \Hom(\cO_X, \cE_2[2])^\vee = 0$.
\end{proof}

\begin{proof}[Proof of Theorem~\ref{thm:index1}]
We have all the ingredients in place to apply Proposition~\ref{prop:inducestability}. Indeed, for $\beta = -1 + \epsilon$ and $\alpha > 0$ sufficiently small, the objects
$\cO_X, \cE_2, \cO_X(-H)[1], \cE_2(-H)[1]$ of $\Coh^\beta(X)$ are $\sigma_{\alpha, \beta}$-stable; one also easily checks with a computation, or a picture using Lemma~\ref{lem:E2location}, that
\[
\mu_{\sigma_{\alpha, \beta}} \left(\cE_2(-H)[1]\right) <
\mu_{\sigma_{\alpha, \beta}} \left(\cO_X(-H)[1]\right) <
\mu_{\sigma_{\alpha, \beta}} \left(\cE_2\right) <
\mu_{\sigma_{\alpha, \beta}} \left(\cO_X\right).
\]
Indeed, for the ``proof by picture'' note that for $\alpha \gg 0$, the objects ordered by slope are $\EE_2$, $\OO_X$, $\EE_2(-H)[1]$, $\OO(-H)[1]$; as $\alpha$ decreases, the order of two objects $A, B$ changes whenever the point corresponding to $\Ker Z_{\alpha, \beta}$ crosses the line between $v_H(A)$ and $v_H(B)$.
Regarding $\EE_2(-H)[1]$, recall from above that $v_H(\EE_2(-H))[1]$ lies above the tangent line at $\OO_X(-H)$, and on the line $\mu_H = -\frac 32$ further to the left.

Therefore, for $\mu$ in between the second and the third slope in these inequalities,
the tilted category $\Coh_{\alpha, \beta}^\mu(X)$ contains all of 
$\cO_X, \cE_2, \cO_X(-H)[2]$ and $\cE_2(-H)[2]$. Thus the weak stability condition
$(\Coh^\mu_{\alpha, \beta}(X), Z^\mu_{\alpha, \beta})$ of Proposition~\ref{prop:tiltedtiltstability} satisfies all the assumptions of Proposition~\ref{prop:inducestability}.
\end{proof}

\subsection*{Relation to Bridgeland stability on $\Db(X)$}
In \cite{Chunyi:Fano3folds}, stability conditions have been constructed on the whole category $\Db(X)$, when $X$ is a Fano threefold of Picard rank one (and in general in \cite{Dulip:Fano,BMSZ:Fano}).
In particular, the category $\Coh_{\alpha,\beta}^\mu(X)$ in Theorems~\ref{thm:1object},~\ref{thm:index2}, and~\ref{thm:index1} is the heart of a Bridgeland stability condition on $\Db(X)$.
While this much stronger result is not needed for our construction, it may be useful to compare stable objects in $\Db(X)$ versus stable objects in $\Ku(X)$, in a similar fashion as what has been done in \cite[Section~3]{LMS:ACM3folds}.
More precisely, in \cite{BMMS:Cubics}, the Kuznetsov component $\Ku(Y_3)$ is realized as an admissible subcategory in $\Db(\P^2,\BB_0)$ orthogonal to the right of an exceptional object (see also Section~\ref{sec:geometric} below).
In \cite{LMS:ACM3folds} the comparison is between stable objects in $\Ku(Y_3)$ and Bridgeland stable objects in $\Db(\P^2,\cB_0)$.

\section{Conic fibrations associated to cubic fourfolds}\label{sec:geometric}

In this section, we start the study of the Kuznetsov component $\Ku(X)$ of a cubic fourfold $X$.
In principle, we would like to apply a similar argument as in the Fano threefold case above. To this end, we would need to tilt three times starting from $\Coh(X)$.
The issue is the lack of a positivity result, generalizing Bogomolov inequality for stable sheaves to tilt-stable objects, which prevents us to tilt a third time.
The key idea then is to realize $\Ku(X)$ as an admissible subcategory of a derived category of modules over $\P^3$ with respect to an algebra vector bundle $\cB_0$. 
By choosing a line in $X$ not contained in a plane in $X$, the induced conic fibration provides $\cB_0$ as the even part of the associated Clifford algebra vector bundle. 

After a brief recall on Kuznetsov's result on semiorthogonal decompositions for quadric fibrations, the goal of this section is to describe such an embedding (see Proposition~\ref{prop:decomp}).

\subsection*{Modules over algebra vector bundles}\label{subsect:modalg}
Let $Y$ be a smooth projective variety, and let $\BB$ be a sheaf of $\cO_Y$-algebras  over $Y$; we will always assume that $\BB$ is a locally free sheaf of finite rank over $Y$, and call such $\BB$ an \emph{algebra vector bundle}.
We denote by $\Coh (Y, \BB)$ the category of coherent sheaves on $Y$ with a right $\BB$-module structure, and denote its derived category by $\Db(Y, \BB)$. The forgetful functor is denoted by
$\Forg \colon \Db(Y, \BB) \to \Db(Y)$. We now review some basic properties of $\Coh(Y, \BB)$ and $\Db(Y, \BB)$, see  \cite[arXiv:v1 Appendix~D/published version Section~10]{Kuznetsov:Hyperplane} and \cite[Section~2.1]{Kuz:Quadric}.

Consider a morphism $f \colon Y' \to Y$, and let $\BB' := f^* \BB$. Then the usual pull-back and push-forward for coherent sheaves directly induce functors
$f^* \colon \Db(Y, \BB) \to \Db(Y', \BB')$ and, when $f$ is projective, $f_* \colon \Db(Y', \BB') \to \Db(Y, \BB)$; in other words, these functors commute with the forgetful functors on $Y$ and $Y'$, and the ordinary pull-back and push-forward for coherent sheaves, respectively.

Let $E \in \Db(Y, \BB)$. 
By abuse of notation and language, we will write $\ch(E) = \ch(\Forg(E))$ for the Chern character of the underlying complex of coherent sheaves, and call it the Chern character of $E$. 
By the observation in the previous paragraph, the behavior of this Chern character behaves exactly as the Chern character of coherent sheaves under pull-backs and push-forwards whenever we are in the situation above (in particular, with $\cB' = f^* \cB$).

Since we assume $Y$ to be smooth and $\BB$ to be a vector bundle, we can also write the Serre functor on $\Db(Y, \BB)$ explicitly as
$S(\blank) = \omega_Y \otimes_{\OO_Y}(\blank)\otimes_{\BB} \BB^\vee[\dim(Y)]$, where
$\BB^\vee$ denotes the dual of $\BB$ as a coherent sheaf, together with its canonical structure as a $\BB$-bimodule. This follows from Serre duality on $Y$ together with the standard adjunctions for the forgetful functor (see also \cite[arXiv:v1 eq.~(70) /published version eq.~(10.1)]{Kuznetsov:Hyperplane}).

\subsection*{Conic fibrations}\label{subsec:conic}
We now recall Kuznetsov's description of the derived category of quadric fibrations, specialized to the case of relative dimension one, and extended to positive characteristic in \cite{ABB:intersectionsofquadrics}.
So let $\pi \colon X \to Y$ be a fibration in conics over a smooth projective variety $Y$ over $k$, where $k$ is an algebraically closed field with $\Char k \neq 2$.
There is a rank three vector bundle $\cF$ on $Y$ and a line bundle $\LL$ such that $X$ embeds into the $\P^2$-bundle $\P_Y(\cF)$ as the zero locus of a section 
\[s_X \in H^0(Y, \textstyle{\Sym^2} \cF^\vee \otimes \LL^\vee) = H^0(\P_Y(\cF), \OO_{\P_Y(\cF)}(2) \otimes \LL^\vee),\]
where, to simplify the notation, we write $\LL^\vee$ both for the line bundle on $Y$ and its pull-back to $\P_Y(\cF)$.

The even part of the Clifford algebra of $\pi$ as a sheaf is
\begin{equation} \label{eq:defB0}
\Forg(\BB_0) = \OO_Y \oplus \big(\textstyle{\bigwedge^2} \cF \otimes \LL\big). \end{equation}
The algebra structure on its fiber $\FF|_y$ at a point $y \in Y$ is determined by
\[
v_i \wedge v_k \cdot v_k \wedge v_j = s_{X}(v_k\otimes v_k)\, v_i \wedge v_j, \quad
v_i \wedge v_k \cdot v_i \wedge v_k = s_{X}(v_i \otimes v_i)\, s_{X}(v_k \otimes v_k),
\]
for an orthogonal basis $(v_1, v_2, v_3)$ of $\cF|_y$ with respect to the quadratic form $s_{X}$ and $i \neq j \neq k \neq i$.

The odd part of the Clifford algebra of $\pi$ is denoted by $\cB_1$.
Furthermore, we define the following $\cB_0$-bimodules, for $j\in\Z$:
\[
\cB_{2j}:=\cB_0\otimes \LL^{-j}\ \text{ and } \ \cB_{2j+1}:=\cB_1\otimes \LL^{-j}.
\]

In positive characteristic we rely on \cite[Proposition~A.1]{ABB:intersectionsofquadrics}, which identifies, for $\Char(k) \neq 2$, the Clifford algebra constructed in \cite[Section~3.3]{Kuz:Quadric} and described above for the case of conics, with the construction in \cite[Section~1.5]{ABB:intersectionsofquadrics} that works in arbitrary characteristic.

The fundamental result on the derived categories of quadric fibrations is the following.
\begin{Thm}[{\cite{Kuz:Quadric}, \cite[Theorem~2.2.1]{ABB:intersectionsofquadrics}}] \label{thm:Kuzquadric} 
There is a semiorthogonal decomposition 
\[ \Db(X) = \langle \Phi(\Db(Y, \BB_0)), \pi^*\Db(Y) \rangle. \]
\end{Thm}
We describe the fully faithful functor $\Phi\colon\Db(Y, \BB_0) \to \Db(X)$ and its left adjoint $\Psi$ explicitly below.

Consider a base change $f \colon Y' \to Y$. Then there are two algebras on $Y'$: the Clifford algebra $\BB_0'$ of the conic fibration
$\pi' \colon Y' \times_Y X \to Y'$, and the pull-back $f^* \BB_0$.

\begin{Lem}\label{lem:basechange}
In the above situation, we have a natural isomorphism $\BB_0'\cong f^*\BB_0$.
\end{Lem}

\begin{proof}
This is observed in the proof of \cite[Lemma~3.2]{Kuz:Quadric}.
\end{proof}

\subsection*{Cubic hypersurfaces and conic fibrations}\label{subsec:cubic4folds}
Let $N\geq1$ and let $X\subset\P^{N+2}$ be a smooth cubic hypersurface of dimension $N+1$ over an algebraically closed field of $\Char(k) \neq 2$. 
We can associate to $X$ a conic fibration as follows.
Let $L_0\subset X$ be a line not contained in a plane in $X$.
Consider the blow-up $\sigma\colon\wX\to X$ along $L_0$, and denote by $i\colon D\hookrightarrow \wX$ its exceptional divisor.
Then the projection from $L_0$ onto the projective space $\P^N$ induces a conic fibration $\pi \colon \wX \to \P^N$ whose discriminant locus is a hypersurface of degree $5$.
We denote by $\alpha\colon\wX\hookrightarrow\wPP$ the embedding into the $\P^2$-bundle $q \colon \wPP \to \P^N$, where $\wPP$ is the blow-up of $\P^{N+2}$ along $L_0$.
Summarizing, we have the following diagram:
\begin{equation}\label{eqn:notacio}
\xymatrix{
& D \ar[r]^i \ar[dl]_p & \wX \ar[dl]_\sigma \ar[drr]_\pi\ar[r]^\alpha &\wPP\ar[dl]|!{[l];[rd]}\hole \ar[rd]^q \\ 
L_0 \ar[r] & X \ar[r]&\P^{N+2} && \P^N. 
}
\end{equation}

\begin{Rem}\label{rem:cubicconic}
Take a generic hyperplane $\P^{N-1}\hookrightarrow\P^{N}$.
The restriction of the conic fibration $\pi$ to $\P^{N-1}$ is the conic fibration obtained by blowing up along $L_0$ the smooth cubic hypersurface of dimension $N$ obtained by intersecting $X$ with the $\P^{N+1}$ spanned by $L_0$ and $\P^{N-1}$.
\end{Rem}

We will abuse notation and denote by $H$ (resp.~by $h$) both the class of the hyperplane in $\P^{N+2}$ (resp.~$\P^{N}$) and the pull-back to $\wX$ and to $\wPP$.

In the notation of the previous section, we then have $\wPP=\P(\cF)$, where 
$\cF = \OO_{\P^N}^{\oplus 2}\oplus \OO_{\P^N}^{}(-h)$; the line bundle $\LL$ is $\OO_{\P^N}(-h)$. The forgetful sheaf $\Forg(\cB_0)$ in \eqref{eq:defB0} of the even part $\cB_0$ of the Clifford algebra of $\pi$ is
\[
\Forg(\BB_0) \cong \OO_{\P^N} \oplus\OO_{\P^N} (-h) \oplus \OO_{\P^N} (-2h)^{\oplus 2},
\]
while the odd part $\cB_1$ is
\[
\Forg(\BB_1) \cong \cF \oplus \textstyle{\bigwedge^3} \cF \otimes \LL \cong \OO_{\P^N}^{\oplus 2} \oplus\OO_{\P^N}^{} (-h) \oplus \OO_{\P^N}^{} (-2h).
\]

We can now define the functors $\Phi$ and $\Psi$ of Theorem~\ref{thm:Kuzquadric}.
There is a canonical map of left $q^*\BB_0$-modules $q^*\BB_0 \to q^*\BB_1(H)$,
which is injective and its cokernel is supported on~$\wX$. Twisting by $\OO_{\wPP}(-2H)$, 
we obtain an exact sequence
\[
0 \to q^*\BB_0(-2H) \to q^*\BB_1(-H) \to \alpha_*\EE' \to 0,
\]
where $\cE'$ is a sheaf of left $\pi^*\BB_0$-modules on $\wX$ and $\Forg(\cE')$ is a vector bundle of rank $2$.
The functor $\Phi\colon\Db(\P^N,\BB_0) \to \Db(\wX)$ is defined as:
\[
\Phi(\blank) = \pi^*(\blank)\otimes_{\pi^*\BB_0} \EE'.
\]
The left adjoint functor of $\Phi$ is
\begin{align*}
\Psi(\blank)&:=\pi_*(\blank\otimes\OO_{\wX}(h)\otimes\EE[1]),
\end{align*}
where $\cE$ is a sheaf of right $\pi^*\BB_0$-modules on $\wX$ and $\Forg(\cE)$ is a vector bundle of rank $2$, defined by the following short exact sequence of $q^*\BB_0$-modules:
\begin{equation}\label{eqn:defdeE}
0\to q^*\BB_{-1}(-2H)\to q^*\BB_0(-H)\to\alpha_*\EE\to0.
\end{equation}

\subsection*{The Kuznetsov component of a cubic fourfold}
We can now describe the Kuznetsov component of a cubic fourfold as an admissible subcategory in $\Db(\P^3,\cB_0)$.

\begin{Def} Let $X$ be a cubic fourfold.
The Kuznetsov component $\Ku(X)$ of $X$ is defined by the semiorthogonal decomposition
\begin{equation*}
 \Db(X)=\gen{ \Ku(X),\OO_{X},\OO_{X}(H),\OO_{X}(2H)}.
\end{equation*}
\end{Def}

Assume now that $X$ is defined over an algebraically closed field with $\Char(k) \neq 2$. We fix a line $L_0\subset X$ not contained in a plane in $X$ and keep the notation as in the previous section.
We start by describing a fully faithful functor $\Xi\colon\Ku(X)\to\Db(\P^3,\cB_0)$. The semiorthogonal complement is then described in Proposition~\ref{prop:decomp}. 

In the proof of Lemma~\ref{lem:so} we will use several times the following elementary lemma, whose statement and proof are analogous to \cite[Lemma~4.1]{Kuz:fourfold}. 

\begin{Lem}\label{lem:conicFibr}
We have linear equivalences in $\wX$:
\[
D = H - h,\qquad
K_{\wX} = -3H + 2D = -H -2h.
\]
\end{Lem}

Theorem~\ref{thm:Kuzquadric} gives the following semiorthogonal decomposition:
\begin{equation}\label{eqn:quadrica}
\Db(\wX) =\langle \Phi(\Db(\P^3,\BB_0)),\underbrace{\OO_{\wX}(-h),\OO_{\wX},\OO_{\wX}(h),\OO_{\wX}(2h)}_{\pi^*\Db(\P^3)}\rangle.
\end{equation}
Set $\Phi':=\Rmut_{\OO_{\wX}(-h)}\circ \Phi$.

\begin{Lem}\label{lem:so}
The admissible subcategory $\Phi'(\Db(\P^3,\BB_0))\subset\Db(\wX)$ has a semiorthogonal decomposition
\begin{gather*}
\Phi'(\Db(\P^3,\BB_0))= \left\langle \sigma^* \Ku(X), \OO_{\wX}(h-H),\Lmut_{\OO_{}}\Lmut_{\OO_{}(h)}\OO_{\wX}(H),\Lmut_{\OO_{}}\Lmut_{\OO_{}(h)}i_*\OO_{D}(h)\right\rangle.
\end{gather*}
\end{Lem}

\begin{proof}
In view of \eqref{eqn:quadrica}, the derived category $\Db(\wX)$ has the following semiorthogonal decompositions
\begin{equation}\label{eqn:semi}
\begin{split}
\Db(\wX)&=	\langle \OO_{\wX}(-h),\Phi'(\Db(\P^3,\BB_0)), \OO_{\wX},\OO_{\wX}(h),\OO_{\wX}(2h)\rangle\\
	&=\langle \Phi'(\Db(\P^3,\BB_0)), \OO_{\wX},\OO_{\wX}(h),\OO_{\wX}(2h),\OO_{\wX}(H+h)\rangle,
\end{split}
\end{equation}
where the second one is obtained via Serre duality.

Since $\wX$ is the blow-up of $X$ along $L_0$, we can apply \cite{Orlovblowup} and get the following semiorthogonal decomposition of $\Db(\wX)$ (here we use the notation in \eqref{eqn:notacio})
\begin{equation}\label{eqn:semi1}
\langle \overbrace{\OO_{\wX}(-H),\sigma^*\Ku(X),\OO_{\wX},\OO_{\wX}(H)}^{\sigma^*\Db(X)},\overbrace{i_*\OO_{D},i_*\OO_{D}(H)}^{\Db(\P^1)=\langle\OO,\OO(1)\rangle},\overbrace{i_*\OO_{D}(H-D),i_*\OO_{D}(2H-D)}^{\Db(\P^1)=\langle\OO(1),\OO(2)\rangle}\rangle
\end{equation}

We claim that the pair $(\OO_{\wX}(H),i_*\OO_{D})$ is completely orthogonal.
To show this, by using the short exact sequence
\[
0\to \OO_{\wX}(-H+h)\to \OO_{\wX}\to i_*\OO_D\to 0
\]
and applying $\Ext^\bullet(\OO_{\wX}(H),\blank)$, we only have to show the vanishings
\begin{align*}
\Ext^\bullet(\OO_{\wX}, \OO_{\wX}(-2H+h))=0\qquad \text{and}\qquad \Ext^\bullet(\OO_{\wX}, \OO_{\wX}(-H)) &=0.
\end{align*}
The second is clear, while for the first one we observe that, since $\sigma_*\OO_{\wX}(-D)=I_{L_0}$ and $D=H-h$, we have
\[
\Ext^\bullet(\OO_{\wX}, \OO_{\wX}(-2H+h))=
\Ext^\bullet(\OO_{\wX}, \OO_{\wX}(-H-D))=
\Ext^\bullet(\OO_{X}, I_{L_0}(-H))=0.
\]

From the orthogonality of the pair $(\OO_{\wX}(H), i_*\OO_{D})$ and \eqref{eqn:semi1}, we get
\begin{equation}\label{eqn:semi2}
\Db(\wX)=\langle \OO_{\wX}(-H),\sigma^*\Ku(X),\OO_{\wX},i_*\OO_{D},\OO_{\wX}(H),i_*\OO_{D}(H),i_*\OO_{D}(h),i_*\OO_{D}(H+h)\rangle.
\end{equation}

Observe now that we have the following equalities 
\begin{align}\label{eqn:iso1}
\Lmut_{\OO_{\wX}}(i_*\OO_D) \cong \OO_{\wX}(h-H)[1]\qquad
\text{and}\qquad \Rmut_{\OO_{\wX}(h-H)}(i_*\OO_{D}) &\cong \OO_{\wX}.
\end{align}
Indeed, note that $\Ext^k(\cO_{\wX},\OO_{\wX}(H-h)) = 0$ if $k\neq 0$ and $\Hom(\cO_{\wX},\OO_{\wX}(H-h)) \cong \C$.
By definition of right mutation, we have the following distinguished triangle
\[
\Rmut_{\OO_{\wX}}(\OO_{\wX}(h-H)) \to \OO_{\wX}(h-H) \to \OO_{\wX}.
\]
Since $H - h = D$, the last map in this triangle is given by the equation of $D$. Thus $\Rmut_{\OO_{\wX}}(\OO_{\wX}(h-H))\cong i_*\OO_D[-1]$.
Equivalently, $\Lmut_{\OO_{\wX}}(i_*\OO_D)\cong \OO_{\wX}(h-H)[1]$.
For the second isomorphism in \eqref{eqn:iso1}, note that $\Ext^k(i_*\OO_D,\OO_{\wX}(h-H)) = 0$ if $k\neq 1$ and $\Ext^1(i_*\OO_D,\OO_{\wX}(h-H)) \cong \C$.
Again, we consider the distinguished triangle
\[
\Rmut_{\OO_{\wX}(h-H)}(i_*\OO_{D}) \to i_*\OO_{D}\to \OO_{\wX}(h-H)[1].
\]
Since $H - h = D$, we can argue as above and conclude that $\Rmut_{\OO_{\wX}(h-H)}(i_*\OO_{D})\cong \OO_{\wX}$.

Thus, by using $\Lmut_{\OO_{\wX}}(i_*\OO_D) \cong \OO_{\wX}(h-H)[1]$ in~\eqref{eqn:iso1} and its tensorization by $\OO_{\wX}(H)$, we obtain from \eqref{eqn:semi2}
\begin{equation}\label{eqn:semi3}
\Db(\wX)=\langle \OO_{\wX}(-H),\sigma^*\Ku(X),\OO_{\wX}(h-H),\underbrace{\OO_{\wX},\OO_{\wX}(h),\OO_{\wX}(H),i_*\OO_{D}(h)}_{\cD},i_*\OO_{D}(H+h)\rangle.
\end{equation}

By applying mutations in $\cD$, we get the semiorthogonal decomposition
\begin{equation}\label{eqn:semi4}
\cD=\langle \Lmut_{\OO_{\wX}}\Lmut_{\OO_{\wX}(h)}\OO_{\wX}(H),\Lmut_{\OO_{\wX}}\Lmut_{\OO_{\wX}(h)}i_*\OO_{D}(h),\OO_{\wX},\OO_{\wX}(h)\rangle.
\end{equation}
By plugging \eqref{eqn:semi4} into \eqref{eqn:semi3}, we get
\begin{equation*}
\Db(\wX)=\langle \OO_{\wX}(-H),\sigma^*\Ku(X),\OO_{\wX}(h-H),\cD,i_*\OO_{D}(H+h)\rangle.
\end{equation*}
We apply Serre duality, and we can rewrite it as
\begin{equation}\label{eqn:semi5}
\Db(\wX)=\langle \sigma^*\Ku(X),\OO_{\wX}(h-H),\cD,i_*\OO_{D}(H+h),\OO_{\wX}(2h)\rangle.
\end{equation}
Finally, we apply to \eqref{eqn:semi5} the isomorphism $\Rmut_{\OO_{\wX}(2h)}(i_*\OO_{D}(H+h))\cong \OO_{\wX}(H+h)$ which is obtained from~\eqref{eqn:iso1} by tensoring by $\OO_{\wX}(H+h)$. Thus we get the semiorthogonal decomposition
\begin{equation}\label{eqn:semi6}
\Db(\wX)=\langle \sigma^*\Ku(X),\OO_{\wX}(h-H),\cD,\OO_{\wX}(2h),\OO_{\wX}(H+h)\rangle.
\end{equation}

Comparing the two semiorthogonal decompositions \eqref{eqn:semi} and \eqref{eqn:semi6}, i.e., comparing
\[
\langle \OO_{\wX},\OO_{\wX}(h),\OO_{\wX}(2h),\OO_{\wX}(H+h)\rangle^\perp
\]
inside them, we get the desired equivalence.
\end{proof}

Consider the functor
\begin{equation}\label{eqn:Xi}
 \Xi=\Psi\circ\, \sigma^*\colon \Ku(X)\longrightarrow \Db(\P^3,\BB_0).
\end{equation}
We are now ready to prove the main result of this section. 

\begin{Prop}\label{prop:decomp}
The functor $\Xi$ is fully faithful. Moreover, 
\[
 \Db(\P^3,\BB_0)=\left\langle \Xi(\Ku(X)), \BB_{1},\BB_{2}, \BB_{3}\right\rangle.
\]
\end{Prop}

\begin{proof}
The projection formula, the fact that $\pi=q\circ\alpha$, and \eqref{eqn:defdeE} show that
\begin{equation}\label{eqn:PsiO}
\Psi(\OO_{\wX}(mh))=0,
\end{equation}
for all $m$.
This implies immediately that 
\begin{equation}\label{eqn:cancellation}
\Psi \circ \Lmut_{\OO_{\wX}(mh)} = \Psi \circ \Rmut_{\OO_{\wX}(mh)} = \Psi.
\end{equation}
In particular, we get
\begin{equation}\label{eqn:cancellation2}
\Psi\circ\Lmut_{\OO_{\wX}(-h)}\circ\,\sigma^*|_{\Ku(X)}=\Psi\circ \sigma^*|_{\Ku(X)} = \Xi.
\end{equation}

Since $\Phi$ is fully faithful, $\Psi\circ\Phi\cong\id$.
By using \eqref{eqn:cancellation2} and by applying $\Psi$ to the decomposition of Lemma~\ref{lem:so}, we get that $\Xi$ is fully faithful and we obtain the semiorthogonal decomposition
\[
\Db(\P^3,\BB_0)=\left\langle \Xi (\Ku(X)), \Psi(\OO_{\wX}(h-H)), \Psi(\OO_{\wX}(H)), \Psi(\OO_{\wX}(2h - H))\right\rangle,
\]
where we have used the short exact sequence $\OO_{\wX}(2h - H) \into \OO_{\wX}(h) \onto i_* \OO_D(h)$ and \eqref{eqn:cancellation}. 

We deduce the claim with a direct computation based on relative Grothendieck-Serre duality for the $\P^2$-fibration $q$ with relative dualizing complex $\omega_q = \OO_{\widetilde\P}(h-3H)[2]$ and \eqref{eqn:PsiO}:
\begin{align*}
\Psi(\OO_{\wX}(mh-H))& = q_* \alpha_*\EE((m+1)h-H)[1]=q_*(q^*\BB_{-1}((m+1)h-3H)[2]) \\
&= q_*(q^* \BB_{-1+2m} \otimes \omega_q)
 =\BB_{-1+2m} 
 \\
 \Psi(\OO_{\wX}(mh+H))& =q_*(q^*\BB_0((m+1)h))[1]
 =\BB_0((m+1)h)[1]=\BB_{2+2m}[1].\qedhere
\end{align*}
\end{proof}

\section{A Bogomolov inequality} \label{sec:Bogomolov}

To construct stability conditions on the Kuznetsov component of a cubic fourfold, we need to define tilt-stability on the category $\Db(\P^3,\cB_0)$ introduced in the previous section.
This requires a Bogomolov inequality for slope-stable torsion-free sheaves in $\Coh(\P^3,\cB_0)$.
This is the content of this section.

\subsection*{Slope-stability}
Let $Y$ be a smooth projective variety of dimension $n$, and let $\cB$ be an algebra vector bundle on $Y$.
Let $\underline{D}=\{D_1,\dots,D_{n-1}\}$ be nef divisor classes on $Y$ such that $D_1^2 D_2\cdots D_{n-1}>0$, and consider the lattice 
\[\Lambda_{\underline{D}}^1 = \langle D_1^2 D_2\cdots D_{n-1}\ch_0,\, D_1 D_2\cdots D_{n-1}\ch_1 \rangle\cong\Z^2.\] 
By Remark~\ref{rem:HNexists}, since $\Coh(Y,\cB)$ is noetherian, the following slight generalization of Example~\ref{ex:slopestability} holds:

\begin{PropDef} The pair $\sigma_{\underline{D}}=(\Coh(Y,\BB), Z_{\underline{D}})$, where
\[
Z_{\underline{D}}=\mathfrak{i}\,D_1^2 D_2\cdots D_{n-1}\ch_0 - D_1 D_2\cdots D_{n-1}\ch_1
\]
defines a weak stability condition on $\Db(Y,\cB)$, which we still call \emph{slope stability}.
We write $\mu_{\underline{D}}$ for the associated slope function.
\end{PropDef}

When $D_1=\dots=D_{n-2}=D$, we use the notation $Z_{D,D_{n-1}}$ and $\mu_{D,D_{n-1}}$.
When moreover $D_{n-1}=D$, we also use the notation $Z_D$ and $\mu_D$, compatibly with Example~\ref{ex:slopestability}.

\subsection*{The main theorem}

Let $N\geq 2$.
Let $X$ be a smooth cubic hypersurface in $\P^{N+2}$ and let $L\subseteq X$ be a line not contained in a plane in $X$; as above we assume that $X$ is defined over an algebraically closed field with $\Char(k) \neq 2$. Consider the blow-up $\wX$ of $X$ along $L$ and the natural conic bundle $\pi\colon\wX\to\P^N$. According to the discussion in Section~\ref{subsec:conic}, this yields an algebra vector bundle $\BB_0$ on $\P^N$.

\begin{Def}
Let $\VV=(V_1,\dots,V_m)$ be an ordered configuration of linear subspaces of codimension $2$ in $\P^N$.
We say that $\psi\colon Y\to\P^N$ is a blow-up along $\VV$ if $\psi$ is the iterated blow-up along the strict transforms of the $V_j$'s.
\end{Def}

\begin{Thm}\label{thm:Bogomolov}
Let $\psi\colon Y\to\P^N$ be a blow-up along an ordered configuration of codimension $2$ linear subspaces.
Let $h=\psi^*\OO_{\P^N}(1)$.
Assume that $E\in\Coh(Y,\psi^*\BB_0)$ is a $\mu_h$-slope semistable torsion-free sheaf.
Then
\[
\Delta_{\psi^*\cB_0}(E) := h^{N-2}\,\left( \ch_1(E)^2 - 2 \rk(E) \left( \ch_2(E) - \frac{11}{32}\rk(E)\right) \right) \geq0.
\]
\end{Thm}

The result will be proved in the rest of this section. It is mainly based on the induction on the rank of $E$, which is a variant of an argument by Langer (see \cite[Section~3]{Langer:positive}). This basically allows a reduction to the case $N=2$ where we provide the estimate using Grothendieck-Riemann-Roch.

While a general Bogomolov inequality can be proved along the lines of \cite[Section~3.2.3]{Lieblich:Twisted} (after reinterpreting $\cB_0$-modules as modules over an Azumaya algebra on a root stack over $Y$, see \cite[Section~3.6]{Kuz:Quadric}), this will not be strong enough for our argument.
In particular, we will need that our inequality is sharp for $\cB_j$, for all $j \in \Z$:

\begin{Rem}\label{rem:B0stableAndDelta0}
The rank of an object in $\Coh(Y,\psi^*\cB_0)$ is always a multiple of $4$: this is part of \cite[Proposition~2.12]{BMMS:Cubics} for $Y = \P^2$; 
using Remark~\ref{rem:cubicconic}, the general case follows by pushing forward along $\psi$, followed by restricting to a generic plane $\P^2 \subset \P^N$.  In particular $\psi^*\cB_j$ is $\mu_h$-stable, for all $j\in\Z$. Moreover, observe that $\Delta_{\psi^*\cB_0}(\psi^*\cB_j)=0$.
\end{Rem}

\subsection*{Blow-ups and the surface case}\label{subsect:Bgeneral}

Let $Y$ be a smooth projective variety of dimension $n$, and let $\BB$ be an algebra vector bundle on $Y$. Let $h$ be a big and nef divisor class on $Y$.

\begin{Lem}\label{lem:BlowUpSurface}
Let $q\colon\wY\to Y$ be the blow up of $Y$ at a smooth codimension $2$ subvariety $S$.
Let $E\in\Coh(\wY,q^*\BB)$ be a torsion-free object.
\begin{enumerate}[{\rm (a)}]
\item \label{item:pushforwardcoh}
The complex $q_*E\in\Db(Y,\BB)$ has two cohomology objects; the sheaf $R^0q_*E$ is torsion-free while $R^1q_*E$ is topologically supported on $S$.
\item \label{item:ch0} $\ch_0(q_*E)=\ch_0(R^0q_*E)=\ch_0(E)$.
\item \label{item:ch1} $h^{n-1}\ch_1(q_*E)=h^{n-1}\ch_1(R^0q_*E)=(q^*h)^{n-1}\ch_1(E)$.
\item \label{item:pushforwardstable}
If $E$ is $\mu_{q^*h}$-semistable, then $R^0q_*E$ is $\mu_h$-semistable.
\end{enumerate}
\end{Lem}

\begin{proof}
This follows immediately from the case $\BB\cong\OO_Y$, since the forgetful functor commutes with push-forward and pull-back.
For example, observe that a subsheaf $A \into R^0 q_*E$ induces a non-zero morphism
$L^0 q^* A \to E$;
the image of this map will give a contradiction to $E$ being torsion-free in \eqref{item:pushforwardcoh} when $A$ is torsion, and to $E$ being $\mu_{q^*h}$-semistable when $A$ destabilizes $R^0q_*E$ with respect to $\mu_h$.
\end{proof}

Assume now that $Y$ has dimension $2$.

\begin{Lem}\label{lem:Bblow}
Assume that, for any $F\in\Coh(Y,\cB)$ which is $\mu_h$-semistable, we have $\Delta_{\BB}(F)\geq 0$ and let $q \colon \wY \to Y$ be the blow-up at a point.
Then, for any $E\in\Coh(\wY,q^*\cB)$ which is $\mu_{q^*h}$-semistable, we have $\Delta_{q^*\cB}(E)\geq 0$.
\end{Lem}

\begin{proof}
Consider a $\mu_{q^*h}$-semistable $q^*\BB$-module $E$ on $\wY$. We want to apply Lemma~\ref{lem:BlowUpSurface}\eqref{item:pushforwardstable}.
We write
$\ch_1(E) = q^*l + a e $, where $e$ is the class of the exceptional divisor,
and $l$ a divisor class on $Y$. After tensoring with an appropriate power of
$\OO_{\wY}(e)$, we may assume that $0 \le a < \rk (E)$. The relative Todd class of $q$ is given by $1 - \frac e2$. We obtain
\begin{align*}
0 & \le \Delta_{\BB}(R^0 q_* E) \le \Delta_{\BB}(q_* E) \\
& = l^2 - 2 \rk(E) \left(\ch_2(E) + \frac a2 - \frac{11}{32}\rk(E)\right)
\\
& \le l^2 - a^2 - 2\rk(E) \left(\ch_2(E) - \frac{11}{32}\rk(E) \right)
 = \Delta_{q^*\BB}(E),
\end{align*}
where the first inequality used the assumption; the second inequality follows from Lemma~\ref{lem:BlowUpSurface}\eqref{item:pushforwardcoh}, the next equality follows from Grothendieck-Riemann-Roch, and the last inequality is due to $0 \le a < \rk(E)$.
\end{proof}

We can now prove Theorem~\ref{thm:Bogomolov} for $N=2$.
Let $\cB_0$ be the Clifford algebra bundle associated to a cubic threefold, as in the statement.

\begin{Prop} \label{prop:BG-surface}
Let $Y$ be a smooth projective surface with a birational morphism $\psi \colon Y \to \P^2$.
Let $E$ be a $\mu_{\psi^*h}$-semistable $\psi^*\BB_0$-module on $Y$.
Then $\Delta_{\psi^*\BB_0}(E) \ge 0$.
\end{Prop}

\begin{proof}
In view of Lemma~\ref{lem:Bblow}, it is enough to prove the result for $\psi=\id$. In this situation, we have (see \cite[Equation (2.2.2)]{LMS:ACM3folds})
\[
	\chi(E,E)=-\frac{7}{64}\rk(E)^2-\frac{1}{4}\ch_1(E)^2+\frac{1}{2}\rk(E)\ch_2(E).
\]
Since $E$ is $\mu_h$-semistable, $\Ext^2(E,E)=\Hom(E,E\otimes_{\BB_0}\BB_{-1})^\vee=0$, where for the first equality we use Serre duality (see \cite[Proposition~2.9(ii)]{BMMS:Cubics}) while the second one follows from \cite[Lemma~2.16]{BMMS:Cubics}.
By considering the Jordan-H\"older factors, we can assume that $E$ is $\mu_h$-stable, by~\cite[Lemma~A.6]{BMS} (see also Lemma~\ref{lem:InequalityJH} below for a more general statement).
Thus we have $\chi(E,E)\leq 1 \leq \frac{r^2}{16}$, where we have used that the rank of $E$ is always divisible by $4$, as already observed in Remark~\ref{rem:B0stableAndDelta0} (see \cite[Proposition~2.12]{BMMS:Cubics}).

Thus, we have
\[
	-\frac{7}{64}\rk(E)^2-\frac{1}{4}\ch_1(E)^2+\frac{1}{2}\rk(E)\ch_2(E)\leq \frac{r^2}{16},
\]
which is what we claimed.
\end{proof}

\subsection*{Deformation of stability}
Let $\psi\colon Y\to\P^N$ be a blow-up along an ordered configuration of codimension $2$ linear subspaces, and let $h=\psi^*\OO_{\P^N}(1)$.
Set $\Pi\subset |h|$ a general pencil and let
\begin{equation}\label{eqn:incidence}
\wY = \set{(D, y) \in \Pi \times Y \mid y \in D}
\end{equation}
be the incidence variety with projections $p\colon \wY \to \Pi$ and $q\colon \wY \to Y$.
Note that $q$ is the blow-up of $Y$ along the base locus of $\Pi$ which is a smooth codimension $2$ subvariety. Moreover, since $\Pi$ is a general pencil, the composition $\wpsi:=\psi\circ q\colon \wY\to \P^N$ is again a blow-up along an ordered configuration of codimension-$2$ linear subspaces.

Let $f$ be the class of a fiber of $p$ and, by abuse of notation, we also denote by $h$ the class of $\wpsi^*\OO_{\P^N}(1)$ in $\wY$.
We consider slope-stability on $\Coh(\wY,\wpsi^*\cB_0)$ with respect to the divisor classes $h,\stackrel{N-2}\dots,h,h_t$, where $\cB_0$ is the Clifford algebra in Theorem~\ref{thm:Bogomolov}, and $h_t:=t h+f$, for $t\in\R_{\ge 0}$.

To apply Langer's argument \cite[Section~3]{Langer:positive}, we want deduce the positivity of the discriminant of $\mu_{h,f}$-stable sheaves from the analogous positivity for $\mu_h$-stable objects with smaller rank.
We prove this by deforming the slope function $\mu_{h,h_t}$.

\begin{Lem}\label{lem:InequalityJH}
Let $t_0\in\R_{\geq0}$.
Let $E\in\Coh(\wY,\wpsi^*\BB_0)$ be a $\mu_{h,h_{t_0}}$-semistable torsion-free sheaf and let $E_1,\dots,E_m$ be its Jordan-H\"older $\mu_{h,h_{t_0}}$-stable factors.
If $\Delta_{\wpsi^*\BB_0}(E_j)\geq 0$, for all $j=1,\dots,m$, then $\Delta_{\wpsi^*\BB_0}(E)\geq0$.
\end{Lem}

\begin{proof}
This follows immediately from \cite[Lemma~A.6]{BMS}.
Indeed, first of all we observe that the lemma in \emph{loc.~cit.}\ still holds for weak stability conditions, with the same proof.
Then, the only thing to check is that $\Ker Z_{h,h_{t_0}}$ is negative semi-definite with respect to the quadratic form $\Delta_{\wpsi^*\BB_0}$.
Explicitly, this means the following.
Let $D$ be a divisors class on $\wY$ such that $h^{N-2}h_tD=0$. We need to show that $h^{N-2}D\leq 0$, which follows immediately from the Hodge Index Theorem.
\end{proof}

\begin{Prop}\label{prop:hf->hh}
Let $E\in\Coh(\wY,\wpsi^*\BB_0)$ be a $\mu_{h,f}$-semistable torsion-free sheaf.
If $\Delta_{\wpsi^*\BB_0}(A)\geq 0$ for all $\mu_h$-semistable torsion-free sheaves $A\in\Coh(\wY,\wpsi^*\BB_0)$ with $\ch_0(A)\leq\ch_0(E)$, then $\Delta_{\wpsi^*\BB_0}(E)\geq0$.
\end{Prop}

\begin{proof}
To start with, by Lemma~\ref{lem:InequalityJH}, we can assume that $E$ is $\mu_{h,f}$-stable.
By arguing as in \cite[Section~3.6]{Langer:positive}, if a sheaf $F\in\Coh(\wY,\wpsi^*\BB_0)$ is $\mu_{h,h_{t_0}}$-stable, then it is $\mu_{h,h_{t}}$-stable, for $t\in\R_{\geq0}$ sufficiently close to $t_0$.

We now have two possible situations. If $E$ is $\mu_{h,h_t}$-stable for all $t\geq0$, then $E$ is $\mu_h$-stable. Hence we can take $A=E$ and apply the assumption, getting the desired inequality.

Otherwise, $E$ is strictly $\mu_{h, h_t}$-semistable for some $t > 0$. 
Let $E_1,\ldots,E_m$ be its Jordan-H\"older factors. For each of them we can apply the same argument. Since whenever we replace $E$ with its Jordan-H\"older factors the rank drops, in a finite number of steps we get to a situation where all Jordan-H\"older factors are $\mu_{h,h_t}$-stable, for all $t$ large. Hence they are $\mu_h$-stable and, by assumption, they all satisfy the inequality $\Delta_{\wpsi^*\BB_0}\ge 0$.
By applying Lemma~\ref{lem:InequalityJH} again, we conclude the proof.
\end{proof}

\subsection*{Induction on the rank}\label{thm:induction}
Let $\psi\colon Y\to\P^N$ be a blow-up along an ordered configuration of codimension $2$ linear subspaces, let $h=\psi^*\OO_{\P^N}(1)$.
We want to prove Theorem~\ref{thm:Bogomolov} by induction on the rank.
To this end, as in \cite{Langer:positive}, we consider the following version of Theorem~\ref{thm:Bogomolov} with fixed rank:
\begin{repThm}{thm:Bogomolov}[$\rk(E)=r$]
Let $E\in\Coh(Y,\psi^*\BB_0)$ be a $\mu_h$-slope semistable torsion-free sheaf of rank $r$.
Then $\Delta_{\psi^*\cB_0}(E) \geq0$.
\end{repThm}

We also consider the following statement, again with fixed rank:

\begin{Thm}[$\rk(E)=r$]\label{thm:langerT1}
Let $E\in\Coh(Y,\psi^*\BB_0)$ be a rank $r$ $\mu_h$-slope semistable torsion-free sheaf.
	Assume that the restriction $\res{E}{D}\in\Coh(D,\res{\psi^*\BB_0}{D})$ of $E$ to a general divisor $D \in |h|$ is not slope semistable with respect to $\res{h}{D}$.
	Then
	\begin{equation*}\label{eqn:langerT1}
		\sum_{i<j} r_i r_j (\mu_i - \mu_j )^2 \leq \Delta_{\psi^*\BB_0}(E), 
	\end{equation*}
where $\mu_i$ (resp., $r_i$) denote the slopes (resp., the ranks) of the Harder-Narasimhan factors of $\res{E}{D}$.
\end{Thm}

Note that, since $h$ is a linear hyperplane section, $\res{\psi}{D}\colon D\to\P^{N-1}$ is again a blow-up along an ordered configuration of codimension $2$ linear subspaces in $\P^{N-1}$, and $\res{\psi^*\BB_0}{D}=\res{\psi}{D}^*(\res{\cB_0}{h})$.
By Remark~\ref{rem:cubicconic}, $\res{\cB_0}{h}$ on $\P^{N-1}$ is still the sheaf of even parts of the Clifford algebra associated to smooth cubic hypersurface of dimension $N$.

We will use induction on $N$, and, in each induction step, induction on the rank (which is divisible by 4 as observed in Remark~\ref{rem:B0stableAndDelta0}).
The idea is then to show that Theorem~\ref{thm:langerT1}($r$) implies Theorem~\ref{thm:Bogomolov}($r$), and Theorem~\ref{thm:Bogomolov}($\le r-4$) implies Theorem~\ref{thm:langerT1}($\le r$), for all $Y$ at once.
The case in which $Y$ is a surface corresponds to Proposition~\ref{prop:BG-surface}, while Theorem~\ref{thm:langerT1}($r=4$) is clear.

\begin{proof}[Theorem~\ref{thm:langerT1}($r$) implies Theorem~\ref{thm:Bogomolov}($r$)]
Let us assume that $E$ is $\mu_h$-semistable but $\Delta_{\psi^*\cB_0}(E)< 0$. 
Theorem~\ref{thm:langerT1} implies that the
restriction of $\res{E}{D}$ is semistable, for any general divisor $D \in |h|$. 

By induction on dimension, the restriction of $E$ to a  general complete intersection
$Y' = |h| \cap\stackrel{N-2}\ldots\cap |h|$ of dimension $2$ is semistable.
Then, Proposition~\ref{prop:BG-surface} implies the result.
\end{proof}

The second implication follows line-by-line the argument in \cite[Section~3.9]{Langer:positive}. The only difference is that we cannot do a complete induction as in \emph{loc.~cit.}, since such a strong inequality is not necessarily true for arbitrary surfaces. Therefore, we have to use the deformation argument in Proposition~\ref{prop:hf->hh} to reduce to blow-ups of $\P^2$.

\begin{proof}[Theorem~\ref{thm:Bogomolov}($\le r-4$) implies Theorem~\ref{thm:langerT1}($\le r$)]
As in the previous section, let $\Pi$ denote a general pencil in $|h|$ and consider the incidence variety $\wY$ in \eqref{eqn:incidence} with projections $p\colon \wY \to \Pi$ and $q\colon \wY \to Y$.
We denote by $e$ the class of the exceptional divisor of $q$ and, as before, $f$ the class of the fiber of $p$.
Note that the center of the blow-up $q$ is smooth and connected (for $N=2$ we use $h^N=1$). 

Note that the Harder-Narasimhan filtration of $q^* E$ with respect to $\mu_{h,f}$ corresponds to the relative Harder-Narasimhan filtration (\cite[Theorem~2.3.2]{HL:Moduli}, generalized to $\cB$-modules, with a similar proof) of $E$ with respect to $p$.
Hence, since $\res{E}{D}$ is not $\mu_{\res{h}{D}}$-semistable, $q^*E$ is not $\mu_{h,f}$-semistable.
Therefore, we consider $0 \subset E_0 \subset E_1 \subset \cdots \subset E_m = q^* E$ the (non-trivial) Harder-Narasimhan filtration with respect to $\mu_{h,f}$ and let $F_i = E_i /E_{i-1}$ be the corresponding $\mu_{h,f}$-semistable factors. 

There exist integers $a_i$ and divisor classes $l_i$ in $Y$ such that $\ch_1 (F_i) = q^* l_i+a_i e $.
Then, since $f=h-e$ and $h^{N-2}e^2=-1$, we have
\begin{equation}\label{eqn:mui}
	\mu_i=\mu_{h,f}(F_i)=\frac{h^{N-1}l_i+a_i}{r_i}.
\end{equation}
On the other hand, since $R^0q_*E_i\subset E$ and $E$ is $\mu_h$-semistable,
by Lemma~\ref{lem:BlowUpSurface}\eqref{item:ch0},\eqref{item:ch1}, we have
\begin{equation}\label{eqn:ineqHN}
	\frac{\sum_{j\leq i} h^{N-1}l_j}{\sum_{j\leq i}r_j}\leq \mu_{h}(E).
\end{equation}
Hence, by \eqref{eqn:mui} and \eqref{eqn:ineqHN}, we deduce 
\begin{equation}\label{eqn:ineq22}
	\sum_{j\leq i}r_j\left(\mu_j-\mu_h(E)\right)\leq \sum_{j\leq i}a_j .
\end{equation}

Since $\rk(F_j)\leq r-4$, by Theorem~\ref{thm:Bogomolov}($\leq r-4$) and Proposition~\ref{prop:hf->hh}, we have
$\Delta_{\widetilde{\psi}^*\cB_0}(F_j)\geq 0$ for all $j$.
Therefore,
\begin{align*}
	\frac{\Delta_{\psi^*\cB_0}(E)}{r}&=\sum_i \frac{\Delta_{\wpsi^*\cB_0}(F_i)}{r_i}- \frac{1}{r}\sum_{i<j}r_ir_j\,h^{N-2}{\left(\frac{\ch_1(F_i)}{r_i}-\frac{\ch_1(F_j)}{r_j}\right)}^2\\
	&\geq \frac{1}{r}\sum_{i<j}r_ir_j\left({\left(\frac{a_i}{r_i}-\frac{a_j}{r_j}\right)}^2-h^{N-2}{\left(\frac{l_i}{r_i}-\frac{l_j}{r_j}\right)}^2\right)\\
	&\geq \frac{1}{r}\sum_{i<j}r_ir_j\left({\left(\frac{a_i}{r_i}-\frac{a_j}{r_j}\right)}^2-{\left(\frac{h^{N-1}l_i}{r_i}-\frac{h^{N-1}l_j}{r_j}\right)}^2\right),
\end{align*}
where the last inequality follows from the Hodge Index Theorem.
By using \eqref{eqn:mui} and simplifying, we see that the last expression in the above inequality is equal to
\[
	2\sum_i a_i\mu_i-\frac{1}{r}\sum_{i<j}r_ir_j(\mu_i-\mu_j)^2.
\]
By \eqref{eqn:ineq22}, we have
\begin{align*}
	\sum_i a_i\mu_i&=\sum_i\left(\sum_{j\leq i}a_j\right)(\mu_i-\mu_{i+1})\geq \sum_i\left(\sum_{j\leq i} r_j(\mu_j-\mu_{h}(E))\right)(\mu_i-\mu_{i+1})\\
	&=\sum_i r_i\mu_i^2-r\mu_h(E)^2 =\sum_{i<j} \frac{r_ir_j}{r}(\mu_i-\mu_j)^2.
\end{align*}
Therefore, we obtain
\[
	\frac{\Delta_{\psi^*\cB_0}(E)}{r}\geq \sum_{i<j} \frac{r_ir_j}{r}(\mu_i-\mu_j)^2,
\]
as we wanted. 
\end{proof}

\section{Cubic fourfolds}\label{sec:4folds}

The goal of this section is to prove the existence of Bridgeland stability conditions on the Kuznetsov component $\Ku(X)$ of a cubic fourfold $X$.

We use the Bogomolov inequality of Section~\ref{sec:Bogomolov} to extend the notion of tilt-stability to $\Db(\P^3,\cB_0)$.
Theorem~\ref{thm:main2} will follow by the general method described in Sections~\ref{sec:inducing} and~\ref{sec:inducestability} as in the case of Fano threefolds.
Finally, we identify the central charge of the associated stability condition with the natural $A_2$-lattice associated to any cubic fourfold.
This allows us to prove a stronger version of the support property, and to obtain an open subset of the space of full numerical stability conditions on $\Ku(X)$.

\subsection*{Weak stability conditions on the twisted projective space}\label{subsec:Bridgeland}

Let $N\geq 2$. Let $X$ be a smooth cubic hypersurface in $\P^{N+2}$ and let $L_0\subset X$ be a line not contained in a $\P^{2}$ contained in $X$.
Let $\cB_0$ be the sheaf of even parts of the Clifford algebra on $\P^N$ associated to the conic fibration induced by projection from $L_0$, as in Section~\ref{sec:geometric}. Let $H$ be an hyperplane section.

We modify the Chern character as follows.

\begin{Def}
For $E\in\Db(\P^N,\cB_0)$, we set
\[
\ch_{\cB_0} (E) := \ch(\Forg(E)) \left( 1 - \frac{11}{32}H^2\right),
\]
where $H^2$ denotes the class of a codimension $2$ linear subspace in $\P^N$.
\end{Def}

In particular, note that $\ch_{\cB_0}$ differs from the usual Chern character only in degree $\geq 2$. We will only care about $\ch_{\cB_0, 2}$.
The Bogomolov inequality in Theorem~\ref{thm:Bogomolov} assumes the following more familiar form.
For any $\mu$-semistable object $E\in\Coh(\P^N,\cB_0)$, we have
\[
\Delta_{\cB_0}(E) = \ch_{\cB_0,1}(E)^2 - 2 \rk(E) \ch_{\cB_0,2}(E) \geq0,
\]
where, as usual, we used $h$ to identify the Chern characters on $\P^N$ with rational numbers.

\begin{Def}
We write $\Coh^\beta(\P^N,\cB_0)$ for the heart of a bounded t-structure obtained by tilting $\Coh(\P^N,\cB_0)$ with respect to slope-stability at the slope $\mu=\beta$.
\end{Def}

The following result, generalizing Proposition~\ref{prop:tiltstability}, can be proved analogously by using Theorem~\ref{thm:Bogomolov}.
We define first a twisted Chern character $\ch^{\beta}_{\cB_0}:= e^{-\beta} \ch_{\cB_0}$ and a lattice $\Lambda_{\cB_0}^2\cong\Z^3$, as in Example~\ref{ex:slopestability}.

\begin{Prop}\label{prop:tiltstabilityBmodules}
Given $\alpha > 0, \beta \in \R$, the pair $\sigma_{\alpha, \beta} = (\Coh^\beta(\P^N,\cB_0), Z_{\alpha, \beta})$
with
\[
Z_{\alpha, \beta}(E) := \mathfrak{i}\, \ch_{\cB_0,1}^\beta(E) + \frac 12 \alpha^2 \ch_{\cB_0,0}^\beta(E) - \ch_{\cB_0,2}^\beta(E)
\]
defines a weak stability condition on $\Db(\P^3,\cB_0)$ with respect to $\Lambda_{\cB_0}^2$. The quadratic form $Q$ can be given by the discriminant $\Delta_{\cB_0}$; these stability conditions vary continuously as $(\alpha, \beta) \in \R_{>0} \times \R$ varies.
\end{Prop}

\begin{Rem}\label{rem:B0tiltstable}
By Remark~\ref{rem:B0stableAndDelta0}, for all $j\in\Z$, $\cB_j$ is slope-stable with slope $\mu(\cB_j)=\frac{-5+2j}{4}$ and $\Delta_{\cB_0}(\cB_j)=0$. 
Hence, $\cB_j\in\Coh^\beta(\P^N,\cB_0)$ (resp.~$\cB_j[1]\in\Coh^\beta(\P^N,\cB_0)$), for $\beta<\frac{-5+2j}{4}$ (resp.~$\beta\geq\frac{-5+2j}{4}$), and by Proposition~\ref{prop:Delta0stable}, it is $\sigma_{\alpha,\beta}$-stable, for all $\alpha>0$.
\end{Rem}

\subsection*{Proof of Theorem~\ref{thm:main2}}

Let $X$ be a cubic fourfold and let us fix a line $L_0\subset X$ not contained in a plane in $X$.
By Proposition~\ref{prop:decomp}, we can realize its Kuznetsov component in the  semiorthogonal decomposition
\[
\Db(\P^3,\cB_0)=\gen{\Ku(X),\cB_1,\cB_2,\cB_3}.
\]
Since we have defined tilt-stability for $\Db(\P^3,\cB_0)$ in Proposition~\ref{prop:tiltstabilityBmodules}, the proof of Theorem~\ref{thm:main2} now goes along the exact same lines as the proof of Theorem~\ref{thm:1object}, by applying Proposition~\ref{prop:inducestability} to the exceptional collection $\langle \cB_1,\cB_2,\cB_3 \rangle$. 

We set $\beta = -1$.
The Serre functor $S$ on $\Db(\P^3,\cB_0)$ is given by $S(\blank)=\blank\otimes_{\cB_0}\cB_{-3}[3]$.
By \cite[Corollary~3.9]{Kuz:Quadric}, $\cB_i\otimes_{\cB_0}\cB_j\cong\cB_{i+j}$, and so $S(\cB_j)=\cB_{j-3}[3]$.
Hence, by Remark~\ref{rem:B0tiltstable}, $\cB_1,\cB_2,\cB_3$, $\cB_{-2}[1],\cB_{-1}[1],\cB_{0}[1] \in \Coh^{-1}(\P^3,\cB_0)$, and they are $\sigma_{\alpha,-1}$-stable for all $\alpha>0$. 

An easy computation shows that, for $\alpha$ sufficiently small,
\[
\mu_{\alpha,-1}(\BB_{-2}[1])<\mu_{\alpha,-1}(\BB_{-1}[1])<\mu_{\alpha,-1}(\BB_0[1])
<0<\mu_{\alpha,-1}(\BB_1)<\mu_{\alpha,-1}(\BB_2)<\mu_{\alpha,-1}(\BB_3).
\]
Therefore, if we tilt a second time to obtain the weak stability condition $\sigma_{\alpha,-1}^{0}$ (exactly in the same way as in Proposition~\ref{prop:tiltedtiltstability}), then 
its heart $\Coh^{0}_{\alpha,-1}(\P^3,\cB_0)$ contains 
$\cB_1,\cB_2,\cB_3,\cB_{-2}[2],\cB_{-1}[2],\cB_{0}[2]$, for $\alpha$ sufficiently small.

By Lemma~\ref{lem:TiltGoodIsAlmostGood}, $F\in (\Coh^{0}_{\alpha,-1}(\P^3,\cB_0))_0$ implies that $\Forg(F)$ is a torsion sheaf supported in dimension $0$.
But then $F\notin\Ku(X)$: indeed, by construction $\cB_2 = \cB_0(H)$, and so  $\Hom_{\cB_0}(\cB_2,F)=\Hom_{\P^3}(\cO_{\P^3}(H),\Forg(F))\neq 0$, as $\blank \otimes \cB_0$ is the left adjoint to the forgetful functor $\Forg$.
Therefore, all assumptions of Proposition~\ref{prop:inducestability} are 
satisfied, and we obtain a stability condition $\sigma_{\Ku(X)}$ on $\Ku(X)$.

By Proposition~\ref{prop:inducestability}, we know that the stability condition $\sigma_{\Ku(X)}$ in $\Ku(X)$ just constructed only satisfies the support property with respect to the lattice $\Lambda_{\cB_0,\Ku(X)}^2\cong\Z^2$ (defined analogously as in \eqref{eqn:LambdaD}).
This concludes the proof Theorem~\ref{thm:main2}. 

To describe (a subset of) the space of stability conditions in more detail, we need the full the support property, for which we will assume $k = \C$.
The numerical Grothendieck group of the Kuznetsov component is larger for any special cubic fourfold, and we will prove that $\sigma_{\Ku(X)}$ satisfies the support property for the full numerical Grothendieck group.

\subsection*{The Mukai lattice of the Kuznetsov component of a cubic fourfold}

Let $X$ be a cubic fourfold over $\C$.
The Kuznetsov component $\Ku(X)$ can be considered as a non-commutative K3 surface.
Its Serre functor is equal to the double shift functor $[2]$ (see \cite[Lemma~4.2]{Kuz:V14}) and there is an analogue of ``singular cohomology'' (and Hochschild (co)homology) for $\Ku(X)$ which is isomorphic to the one of a K3 surface (e.g., \cite{Kuz:SemiOrthogonal,Kuz:fourfold,AddingtonThomas:CubicFourfolds}).

We summarize the basic properties of the Mukai structure on $\Ku(X)$ in the statement below. The proofs and some context can be found in \cite[Section~2]{AddingtonThomas:CubicFourfolds}.

\begin{PropDef}\label{prop:MukaiLattice}
Let $X$ be a smooth cubic fourfold over $\C$.
\begin{enumerate}[{\rm (a)}]
\item The \emph{singular cohomology} of $\Ku(X)$ is defined as
\[
H^*(\Ku(X),\Z) := \left\{ \kappa \in K_{\mathrm{top}}(X)\,:\, \chi_{\mathrm{top}}\left([\cO_X(i)],\kappa\right)=0,\text{ for }i=0,1,2  \right\},
\]
where $K_{\mathrm{top}}(X)$ denotes the topological K-theory of $X$ and $\chi_\mathrm{top}$ the topological Euler characteristic\footnote{See, e.g., \cite[Section~2.1]{AddingtonThomas:CubicFourfolds} for the basic definitions; but note that our notation $H^*(\Ku(X),\Z)$ is different.}.
It is endowed with a bilinear symmetric non-degenerate unimodular even form $(\blank,\blank)=-\chi_{\mathrm{top}}(\blank,\blank)$\footnote{Our pairing is the opposite of the one in \cite{AddingtonThomas:CubicFourfolds}, to be coherent with the usual Mukai structure on K3 surfaces.} called the \emph{Mukai pairing}, and with a weight-$2$ Hodge structure
\[
H^*(\Ku(X),\C):=H^*(\Ku(X),\Z)\otimes \C=\bigoplus_{p+q=2} \widetilde{H}^{p,q}(\Ku(X)).\] As a lattice $H^*(\Ku(X),\Z)$ is abstractly isomorphic to the extended K3 lattice $U^4\oplus E_8(-1)^2$.

\item The \emph{Mukai vector} $\vv\colon K(\Ku(X))\to H^*(\Ku(X),\Z)$ is the morphism induced by the natural morphism $K(X)\to K_\mathrm{top}(X)$.

\item The \emph{algebraic Mukai lattice} of $\Ku(X)$ is defined as
\[
H^*_\mathrm{alg}(\Ku(X),\Z) := H^*(\Ku(X),\Z) \cap \widetilde{H}^{1,1}(\Ku(X)).
\]
It coincides with the image of $\vv$, and it is isomorphic with $\Knum(\Ku(X))$.
Moreover, given $E,F\in\Ku(X)$, we have
\begin{equation}\label{eqn:MukaiPairingEulerChar}
\chi(E,F) = - (\vv(E), \vv(F)).
\end{equation}
\end{enumerate}
\end{PropDef}

The signature of the Mukai pairing on $H^*_\mathrm{alg}(\Ku(X),\Z)$ is $(2,\rho)$, where $0\leq \rho \leq 20$. 

When $\Ku(X)\cong\Db(S)$, where $S$ is a K3 surface, then this coincides with the usual Mukai structure on $S$.
For a very general cubic fourfold ($\rho=0$) its Kuznetsov component is never equivalent to the derived category of a K3 surface.

The algebraic Mukai lattice always contains two special classes:
\[
\llambda_1 := \vv(\mathop{\mathrm{pr}}\cO_L(H))\quad \text{ and } \quad
\llambda_2 := \vv(\mathop{\mathrm{pr}}\cO_L(2H)),
\]
where $L\subset X$ denotes a line and 
$\mathop{\mathrm{pr}}:\Db(X)\to \Ku(X)$ is the natural projection functor.
They generate the primitive positive definite sublattice
\[
A_2 =\begin{pmatrix} 2 & -1 \\ -1 & 2 \end{pmatrix} \subset H^*_\mathrm{alg}(\Ku(X),\Z)
\]
which agrees with the image of the restriction $K_{\mathrm{top}}(\P^5) \to K_{\mathrm{top}}(X)$ followed by projection.
We think of this primitive sublattice as the choice of a polarization on $\Ku(X)$.
We  make this precise in Proposition~\ref{prop:A2}:  the central charge of the stability condition constructed in Theorem~\ref{thm:main2}  exactly corresponds to this sublattice 
(see also \cite[Section~1.2]{Kontsevich-Soibelman:stability}).
A first instance of this is given by the following result.

Let $F(X)$ denote the Fano variety of lines contained in $X$ endowed with the natural polarization $g$ coming from the Pl\"ucker embedding $F(X)\hookrightarrow \mathrm{Gr}(2,6)$.

\begin{Prop}[{\cite[Proposition~2.3]{AddingtonThomas:CubicFourfolds}}]\label{prop:OrthogonalA2Lattice}
There exist Hodge isometries
\[
\langle \llambda_1,\llambda_2 \rangle^{\perp} \cong H^4_\mathrm{prim}(X,\Z)(-1) \cong H^2_\mathrm{prim}(F(X),\Z)
\]
where $\langle \llambda_1,\llambda_2 \rangle^{\perp} \subset H^*(\Ku(X),\Z)$.
\end{Prop}

\subsection*{Bridgeland stability conditions on the Kuznetsov component}

Let $X$ be a cubic fourfold.
We now review the basic theory of Bridgeland stability conditions for K3 categories \cite{Bridgeland:K3} applied to $\Ku(X)$.

Let $\Lambda$ be a lattice together with a surjective map $v\colon K(\Ku(X))\onto \Lambda$.
Assume that $v$ factors via the surjections $K(\Ku(X)) \stackrel{\vv}{\longtwoheadrightarrow} H^*_\mathrm{alg}(\Ku(X),\Z) \stackrel{u}\longtwoheadrightarrow \Lambda$.
Let $\sigma=(\AA,Z)$ be a stability condition with respect to such $\Lambda$ and let $\eta(\sigma)\in H_{\mathrm{alg}}^*(\Ku(X),\C)$ be such that
\[
(Z\circ u) (\blank)=(\eta(\sigma),\blank).
\]
As in \cite{Bridgeland:K3}, we define $\PP\subset H_{\mathrm{alg}}^*(\Ku(X),\C)$ 
as the open subset consisting of those vectors whose real and imaginary parts span positive-definite two-planes in $H_{\mathrm{alg}}^*(\Ku(X),\R)$,
and $\PP_0$ as
\[
\PP_0 := \PP \setminus \bigcup_{\delta\in\Delta} \delta^\perp,
\]
where $\Delta:=\set{\delta\in  H_{\mathrm{alg}}^*(\Ku(X),\Z)\,:\, \chi(\delta,\delta)=2}$.

\begin{Lem}\label{lem:FullSupportStabilityKuznetsovCubicFourfold}
Let $\sigma=(\AA,Z)$ be a stability condition with respect to such $\Lambda$.
If $\eta(\sigma)\in\PP_0$, then $\sigma':=(\AA,Z\circ u)$ is a stability condition with respect to $H^*_\mathrm{alg}(\Ku(X),\Z)$.
\end{Lem}

\begin{proof}
By Proposition~\ref{prop:supportproperty2}, this follows immediately from \cite[Lemma~8.1]{Bridgeland:K3}: 
the generalization to the case of the Kuznetsov component is straight-forward, since the Mukai pairing has the correct signature $(2,\rho)$ and, by Serre duality and \eqref{eqn:MukaiPairingEulerChar}, for any $\sigma$-stable object $E$ in $\Ku(X)$, we have $\vv(E)^2\geq-2$.
\end{proof}

Motivated by the above lemma, we can then make the following definition.

\begin{Def}
A \emph{full numerical stability condition} on $\Ku(X)$ is a Bridgeland stability condition on $\Ku(X)$ whose lattice $\Lambda$ is given by the Mukai lattice $H^*_\mathrm{alg}(\Ku(X),\Z)$ and the map $v$ is given by the Mukai vector $\vv$.
\end{Def}

We denote by $\Stab(\Ku(X))$ the space of full numerical stability conditions on $\Ku(X)$.
The map $\eta\colon\Stab(\Ku(X))\to H_{\mathrm{alg}}^*(\Ku(X),\C)$ defined above is then a local homeomorphism, by Bridgeland's Deformation Theorem \cite[Theorem~2.1]{Bridgeland:Stab}.
If $\eta(\sigma)\in\PP_0$, then we have a more precise result.

\begin{Prop}\label{prop:Bridgeland}
The map $\eta\colon \eta^{-1}(\PP_0)\subset\Stab(\Ku(X)) \to \PP_0$ is a covering map.
\end{Prop}

\begin{proof}
This can be proved exactly as in \cite[Proposition~8.3]{Bridgeland:K3} 
(see also \cite[Corollary~1.3]{Arend:short}).
\end{proof}

\subsection*{The support property}
Let $X$ be a cubic fourfold over $\C$ and let us fix a line $L_0\subset X$ not contained in a plane in $X$.
We can now show that the stability condition in Theorem~\ref{thm:main2} is a full numerical stability condition on $\Ku(X)$.

Consider the fully faithful functor $\Xi\colon\Ku(X) \to \Db(\P^3,\cB_0)$ in \eqref{eqn:Xi}, and the composition $\Forg\circ\Xi$.
The composition of the induced morphism $(\Forg\circ \Xi)_*$ at level of numerical Grothendieck groups and the truncated Chern character $\ch_{\cB_0,\leq 2}$ gives a surjective morphism
$H^*_\mathrm{alg}(\Ku(X),\Z) \longtwoheadrightarrow \Lambda_{\cB_0,\Ku(X)}^2$,
where $\Lambda_{\cB_0}^2$ is the lattice generated by $\ch_{\cB_0,0}, \ch_{\cB_0,1}, \ch_{\cB_0,2}$ (see Proposition~\ref{prop:tiltstabilityBmodules}) and $\Lambda_{\cB_0,\Ku(X)}^2\subset \Lambda_{\cB_0}^2$ is nothing but the image of $K(\Ku(X))$ (see \eqref{eqn:LambdaD}).

Consider the stability condition $\sigma_{\Ku(X)}=(\AA,Z)$ defined in the proof of Theorem~\ref{thm:main2}.
By the above discussion, we can then define $\eta(\sigma_{\Ku(X)})\in H^*_\mathrm{alg}(\Ku(X),\Z)$.
To prove that $\sigma_{\Ku(X)}$ is a full numerical stability condition on $\Ku(X)$ we only need to check the condition of Lemma~\ref{lem:FullSupportStabilityKuznetsovCubicFourfold}, namely the following
proposition:

\begin{Prop} \label{prop:A2}
We have $\eta(\sigma_{\Ku(X)})\in (A_2)_\C \cap\PP\subset\PP_0$. 
\end{Prop}

\begin{proof}
We consider the subspace $V$ in $H^*_\mathrm{alg}(\Ku(X),\R)$ generated by the real and imaginary part of $\eta(\sigma_{\Ku(X)})$.
We claim that $V=A_2$.

To prove the claim, we freely use the notation from Section~\ref{sec:geometric}.
First of all, by definition of $Z_\alpha=Z_{\alpha,-1}^0$, it is straightforward to check that $V$ has real dimension $2$.
Indeed, $\ch_{\cB_0,\leq 2}(\Xi(\mathop{\mathrm{pr}}(\cO_L(H))))=(4,-1,-\frac{15}{8})$ and $\ch_{\cB_0,\leq 2}(\Xi(\mathop{\mathrm{pr}}(\cO_L(2H))))=(-8,8,-\frac{18}{8})$.
Since 
\[
Z_\alpha=\ch_{\cB_0,1}^{-1}+ \mathfrak{i}\left(-\tfrac12\alpha^2\ch_{\cB_0,0}^{-1}+\ch_{\cB_0,2}^{-1}\right),
\]
we have $Z_\alpha(\mathop{\mathrm{pr}}(\cO_L(H)))=3+\mathfrak{i}(-2\alpha^2-\frac{7}{8})$ and $Z_\alpha(\mathop{\mathrm{pr}}(\cO_L(2H)))=\mathfrak{i}(4\alpha^2+\frac{7}{4})$, and they are linearly independent.

Hence, to prove the claim $V = A_2$ it remains to show that $\eta(\sigma_{\Ku(X)})\in (A_2)_\C$; equivalently, we have to show that for $F$ with
$\vv(F) \in A_2^\perp = H^4_{\mathrm{prim}}(X, \Z)(-1)$ (see Proposition~\ref{prop:OrthogonalA2Lattice}), we have 
$Z_{\alpha}(F) = 0$. 

Let $j \colon \P^2 \into \P^3$ be the inclusion of a hyperplane, and let 
$j_X \colon X_H \into X$ be the inclusion of the corresponding hyperplane section of $X$ containing $L_0$. Let $\Xi_H \colon \Db(X_H) \to \Db(\P^2, \BB_0|_{\P^2})$ be the restriction of the functor $\Xi$.

The assumption $\vv(F) \in H^4_{\mathrm{prim}}(X, \Z)$
implies that $\ch\left(j_{X*}^{} j_X^* F\right) = 0$, since the
class $[F] - [F\otimes \cO_X(-H)]$ in the topological $K$-group is zero.
On the other hand, since our formula for the central charge only depends on $\ch_i(\Forg(\Xi(F)))$ for $0 \le i \le 2$, it is determined by the Chern character of
\begin{align*}
j_*^{} j^* \Forg\bigl(\Xi(F)\bigr) = j_*^{} \Forg\ \bigl(\Xi_H(j_X^* F)\bigr)
= \Forg \circ \Xi \bigl(j_{X*}^{} j_X^* F\bigr),
\end{align*}
where we used base change in the first equality, and projection formula in the second. By the observation above, the class of this object in the $K$-group of $\P^3$ vanishes, and thus its central charge
$Z_{\alpha}(\blank)$ is zero. Therefore $\eta(\sigma_{\Ku(X)}) \in (A_2)_\C$ as claimed.

By \cite[Proposition~1, page~596]{Voisin:cubics} the primitive cohomology $H^4_{\mathrm{prim}}(X, Z)(-1)$ cannot contain algebraic classes $\delta$ with square $\delta^2 = -2$; in other words, 
$(A_2)_\C \cap\PP\subset\PP_0$ completing the proof.
\end{proof}

\begin{Rem}
Combining Propositions~\ref{prop:Bridgeland} and \ref{prop:A2}, we obtain an open subset $\Stab^{\dag}(\Ku(X)) \subset \Stab(\Ku(X))$ with a covering to $\PP_0$. However, at this point we cannot prove that it forms a connected component. A proof would follow from the existence of semistable objects 
for every primitive class $\vv \in H^*_{\mathrm{alg}}(\Ku(X), \Z)$ and for
every stability condition $\sigma$ in this open subset.
\end{Rem}

Finally, let us point out that Proposition~\ref{prop:glue} and its proof give:
\begin{Cor}
There are stability conditions $\sigma'$ on $\Db(X)$. They satisfy the support property with respect to the image $\Lambda \subset H^*(X, \Q)$ of the Chern character.
\end{Cor}

\appendix

\section{The Torelli Theorem for cubic fourfolds}\label{app:Torelli}
\begin{center}
\small{by {\sc A.~Bayer, M.~Lahoz, E.~Macr\`i, P.~Stellari, X.~Zhao}}
\end{center}

\vspace{0.2cm}

In the recent inspiring paper \cite{HR16:Torelli_cubic_fourfolds}, Huybrechts and Rennemo gave a new proof of the Torelli Theorem for cubic fourfolds by first proving a categorical version of it.
The aim of this appendix is to present a different proof, based on Theorem~\ref{thm:main2}, of their categorical Torelli Theorem; our proof works for \emph{very general} cubic fourfolds.

\begin{Thm}\label{thm:CategoricalTorelli}
Let $X$ and $Y$ be smooth cubic fourfolds over $\C$.
Assume that $H_{\mathrm{alg}}^*(\Ku(X),\Z)$ has no $(-2)$-classes.
Then $X\cong Y$ if and only if there is an equivalence
$\Phi \colon \Ku(X) \to \Ku(Y)$ whose induced map $H^*_{\mathrm{alg}}(\Ku(X), \Z) \to H^*_{\mathrm{alg}}(\Ku(Y), \Z)$ commutes with the action of $(1)$.
\end{Thm}

In the above statement, we denote by $(1)$ the natural autoequivalence of the Kuznetsov component induced by $\blank \otimes\OO_X(1)$ followed by projection; it is called the \emph{degree shift functor}. Note that the algebraic Mukai lattice is isomorphic to the numerical Grothendieck group, hence the action induced by $\Phi$ can be defined without assuming that it is of Fourier-Mukai type.

Theorem~\ref{thm:CategoricalTorelli} is a very general version of \cite[Corollary~2.10]{HR16:Torelli_cubic_fourfolds} in the cubic fourfold case (with the minor improvement that we do not need to assume that the equivalence is given by a Fourier-Mukai functor, and that we only need compatibility with $(1)$ on the level of algebraic Mukai lattice).  
It is still enough to deduce the classical Torelli Theorem, as we briefly sketch in Section~\ref{subsec:appendix:ClassicalTorelli}.
Particular cases of it also appeared as \cite[Proposition~6.3]{BMMS:Cubics} (for generic cubics containing a plane) and \cite[Theorem~1.5, (iii)]{Huy:cubics} (for cubics such that $A_2$ is the entire algebraic Mukai lattice); however, both proofs rely on the classical Torelli Theorem for cubic fourfolds.

The main idea of the proof of Theorem~\ref{thm:CategoricalTorelli} is to use the existence of Bridgeland stability conditions on $\Ku(X)$.
As observed in \cite[Section~5]{KuznetsovMarkushevic}, the Fano variety of lines $F(X)$ of a cubic fourfold $X$ is isomorphic to a moduli space of torsion-free stable sheaves on $X$ which belong to $\Ku(X)$.
Given a line $L_X\subset X$, we denote by $F_{X,L_X}\in\Ku(X)$ the corresponding sheaf.
If $H_{\mathrm{alg}}^*(\Ku(X),\Z)$ has no $(-2)$-classes, we show in Proposition~\ref{prop:FanoStable} that the objects $F_{X,L_X}$ are also Bridgeland stable in $\Ku(X)$ for any stability condition.
Moreover, an argument by Mukai implies that any object with the same numerical class and the same Ext-groups must be one of them (up to shift), see Proposition~\ref{prop:MukaiConnectedness}.
Given an equivalence of triangulated categories $\Phi\colon\Ku(X)\mapto{\iso}\Ku(Y)$, we consider the images $\Phi(F_{X,L_X})$.
If the induced between the algebraic Mukai lattice commutes with the degree shift functor $(1)$, then we can assume that all objects $\Phi(F_{X,L_X})$ and $F_{Y,L_Y}$ have the same numerical class up to composing $\Phi$ with the shift functor $[1]$.
We  use this to obtain an isomorphism between $F(X)$ and $F(Y)$.
Finally, by \cite{BM:projectivity}, moduli spaces of Bridgeland stable objects come equipped with a natural line bundle.
If we choose a stability condition with central charge in the $A_2$-lattice, as the one we construct in Theorem~\ref{thm:main2}, the induced line bundle is exactly the Pl\"ucker polarization on the Fano variety of lines (up to constant).
Hence the isomorphism between $F(X)$ and $F(Y)$ preserves the Pl\"ucker polarization.
This is enough to recover an isomorphism $X\cong Y$, by an elementary argument by Chow \cite{Chow:homsp}, see \cite[Proposition~4]{Charles:Torelli}.
	
This argument should hold without the assumption on $(-2)$-classes, and so prove Theorem~\ref{thm:CategoricalTorelli} without assumptions, as originally stated in \cite{HR16:Torelli_cubic_fourfolds}. 
This would also directly imply the strong version of the classical Torelli theorem, as originally stated in \cite{Voisin:cubics}. 
The two issues are to prove that the objects $F_{X,L_X}$ are Bridgeland stable with respect to the stability conditions we constructed in Theorem~\ref{thm:main2}, and that we can change the equivalence $\Phi$ by autoequivalences of $\Ku(Y)$ until it  preserves such stability conditions.
A proof of this generalized version has been recently given in \cite{LPZ:Cubics}.


\subsection{Classification of stable objects}\label{subsec:appendix:classification}

Let $X$ be a cubic fourfold. We freely use the notation in Section~\ref{sec:4folds}.
We start by recalling the following elementary but very useful result due to Mukai, which  allows us to control (in)stability of objects with small $\Ext^1$. It first appeared in \cite{Mukai:BundlesK3}; see \cite[Lemma~2.5]{BB:K3Pic1} for the version stated here:

\begin{Lem}[Mukai]\label{lem:Mukai}
Let $A \to E \to B$ be an exact triangle in $\Ku(X)$ with $\Hom(A,B)=0$.
Then
\[
\dim_\C \Ext^1(A,A)+\dim_\C \Ext^1(B,B)\leq\dim_\C \Ext^1(E,E).
\]
\end{Lem}

As first corollary, we show that under our assumption there are no objects with $\Ext^1=0$.

\begin{Lem}\label{lem:NoRigid}
Assume that $H_{\mathrm{alg}}^*(\Ku(X),\Z)$ has no $(-2)$-classes.
Then there exists no non-zero object $E\in\Ku(X)$ with $\Ext^1(E,E)=0$.
\end{Lem}

\begin{proof}
Let $E\in\Ku(X)$ be a non-zero object such that $\Ext^1(E,E)=0$.
Let $\sigma\in\Stab(\Ku(X))$.
By Lemma~\ref{lem:Mukai}, we can assume that $E$ is $\sigma$-semistable and that it has a unique $\sigma$-stable factor $E_0$.
Therefore, $\vv(E)^2<0$.
But then $\vv(E_0)^2<0$ as well. 
Since $\Hom(E_0,E_0)\cong\C$, we have $ 0 > \vv(E_0)^2 = \ext^1(E_0, E_0) - 2 \neq -2$; this is a contradiction as $\Ext^1(E_0,E_0)$ is even-dimensional.
\end{proof}

Using Lemma~\ref{lem:NoRigid}, we can show that objects with $\Ext^1\cong\C^2$ are always stable.

\begin{Lem}\label{lem:SemirigidAlwaysStable}
Assume that $H_{\mathrm{alg}}^*(\Ku(X),\Z)$ has no $(-2)$-classes.
Let $E\in\Ku(X)$ be an object with $\Ext^1(E,E)\cong\C^2$.
Then, for all $\sigma\in\Stab(\Ku(X))$, $E$ is $\sigma$-stable.
In particular, $\vv(E)^2=0$.
\end{Lem}

\begin{proof}
Let $E\in\Ku(X)$ be an object with $\Ext^1(E,E)\cong\C^2$.
Let $\sigma\in\Stab(\Ku(X))$.
By Lemma~\ref{lem:Mukai} and Lemma~\ref{lem:NoRigid}, we deduce that $E$ is $\sigma$-semistable with a unique $\sigma$-stable object $E_0$.
Therefore, $\vv(E)^2\leq0$.
But then $-2 \le \vv(E_0)^2\leq0$, and so $\vv(E_0)^2=0$, by assumption.
We deduce that $\vv(E)^2=0$ and so that $\Hom(E,E)\cong\C$.
This implies that $E=E_0$, as we wanted.
\end{proof}

Finally, we can study stability of objects with $\Ext^1\cong\C^4$, the case of interest for us.

\begin{Lem}\label{lem:StabilityFano}
Assume that $H_{\mathrm{alg}}^*(\Ku(X),\Z)$ has no $(-2)$-classes.
Let $E\in\Ku(X)$ be an object with $\Ext^{<0}(E,E)=0$, $\Hom(E,E)\cong\C$, and $\Ext^1(E,E)\cong\C^4$.
Then, for all $\sigma\in\Stab(\Ku(X))$, $E$ is $\sigma$-stable.
\end{Lem}

\begin{proof}
By assumption, $\vv(E)^2=2$.
Therefore, $\vv(E)$ is a primitive vector in $H_{\mathrm{alg}}^*(\Ku(X),\Z)$.
By Lemma~\ref{lem:Mukai}, Lemma~\ref{lem:NoRigid}, and Lemma~\ref{lem:SemirigidAlwaysStable}, if $E$ is not $\sigma$-stable, then there exists a triangle $A\to E \to B$, where $A$ and $B$ are both $\sigma$-stable with $\ext^1=2$.
Hence, $\vv(A)^2=\vv(B)^2=0$ and $(\vv(A)+\vv(B))^2=2$.
But then $(\vv(A)-\vv(B))^2=-2$, a contradiction.
\end{proof}

Before stating the main result, we recall a construction by Kuznetsov and Markushevich.
Given a line $L\subset X$, we define a torsion-free sheaf $F_L := \Ker \left( \OO_X^{\oplus 4} \onto I_L(1) \right)$ as the kernel of the evaluation map.
Then by \cite[Section~5]{KuznetsovMarkushevic}, $F_L$ is a torsion-free Gieseker-stable sheaf on $X$ which has the same $\Ext$-groups as $I_L$, and which belongs to $\Ku(X)$.
By definition of $\llambda_1$, one easily verifies $\vv(F_L)=\llambda_1$.
By letting $L$ vary, the sheaves $F_L$ span a connected component of the moduli space of Gieseker-stable sheaves which is isomorphic to $F(X)$ \cite[Proposition~5.5]{KuznetsovMarkushevic}.
We denote by $\mathcal{F}_\mathcal{L}$ the universal family.

\begin{Prop}\label{prop:FanoStable} Assume that $H_{\mathrm{alg}}^*(\Ku(X),\Z)$ has no $(-2)$-classes.
Let $L\subset X$ be a line.
Then, for all $\sigma\in\Stab(\Ku(X))$, the sheaf $F_L$ is $\sigma$-stable.
\end{Prop}

\begin{proof}
This is now immediate from Lemma~\ref{lem:StabilityFano}.
\end{proof}

The last result we need is about moduli spaces, by generalizing an argument by Mukai \cite{Mukai:BundlesK3} (see \cite[Theorem~4.1]{KLS:SingSymplecticModuliSpaces}).
Let $\sigma=(Z,\AA)$ be a Bridgeland stability condition on $\Ku(X)$.
Let us also fix a numerical class $\vv\in H_{\mathrm{alg}}^*(\Ku(X),\Z)$.
We denote by $M^{\mathrm{spl}}(\Ku(X))$ the space parameterizing simple objects in $\Ku(X)$, which is an algebraic space locally of finite-type over $\C$ by \cite{Inaba}.
We also denote by $M_{\sigma}^{\mathrm{st}}(\vv)\subset M^{\mathrm{spl}}(\Ku(X))$ the subset parameterizing $\sigma$-stable objects in $\AA$ with Mukai vector $\pm \vv$.

\begin{Prop}\label{prop:MukaiConnectedness}
Assume there exists a smooth integral projective variety $M\subset M_{\sigma}^{\mathrm{st}}(\vv)$ of dimension $\vv^2+2$.
Then $M = M_{\sigma}^{\mathrm{st}}(\vv)$.
\end{Prop}

\begin{proof}
The proof works exactly as in \cite[Theorem~4.1]{KLS:SingSymplecticModuliSpaces}.
For simplicity, we assume the existence of a universal family $\FF$ on $M$. Note that this is the case in our situation where $M=F(X)$.
The general case, in which only a quasi-universal family exists, can be treated similarly as in \cite[Lemma~4.2]{KLS:SingSymplecticModuliSpaces}.

Suppose that $M_{\sigma}^{\mathrm{st}}(\vv) \neq M$, and consider objects $F\in M$ and $G \in M_{\sigma}^{\mathrm{st}}(\vv) \setminus M$.
We consider the product $M\times X$, and we denote by $p\colon M\times X\to M$ and $q\colon M\times X \to X$ the two projections.
We will implicitly treat all objects in $\Ku(X)$ as objects in $\Db(X)$ without mentioning.
The idea is to look at the following objects in $\Db(M)$: $\mathbb{F} := p_* \lHom(q^* F, \FF)$ and $\mathbb{G} := p_* \lHom(q^* G, \FF)$.

To apply the argument in \cite[Theorem~4.1]{KLS:SingSymplecticModuliSpaces}, we need to show that $\mathbb{F}$ is quasi-isomorphic to a complex of locally free sheaves on $M$ of the form $A^0 \to A^1 \to A^2$ and $\mathbb{G}[-1]$ is a locally-free sheaf.
By \cite[Proposition~5.4]{Bridgeland-Maciocia:K3Fibrations}, it suffices to show that, for all closed points $y\in M$, the complex $\mathbb{F}\otimes k(y)$ is supported in degrees $0,1,2$.
Since the projection $p$ is flat, by the Projection Formula and Cohomology and Base Change (see \cite[Lemma~1.3]{BondalOrlov:Main}), we have $\mathrm{Tor}_{-j}(\mathbb{F}, k(y)) \cong \Ext^{j}(F,(i_{y}\times\id)^{*}\FF)$,
where $i_{y}\times\id \colon \{ y \}\times X \to M\times X$ denotes the inclusion.
But $F$ and $(i_{y}\times\id)^{*}\FF$ both belong to a heart of a bounded t-structure in $\Ku(X)$. Hence, $\Ext^j(F,(i_{y}\times\id)^{*}\FF)=0$, for all $j\neq 0,1,2$, as we wanted.

A similar computation gives that $\mathbb{G} \otimes k(y)$ is supported only in degree $1$, and so $\mathbb{G}[-1]$ is a locally-free sheaf on $M$.
The rest of the argument can be carried out line-by-line following \cite{KLS:SingSymplecticModuliSpaces}.
\end{proof}

Recall that, when $M_{\sigma}^{\mathrm{st}}(\vv)$ is a proper algebraic space over $\C$, by \cite[Section~4]{BM:projectivity}, we can define a nef divisor class $\ell_\sigma$ on $M_{\sigma}^{\mathrm{st}}(\vv)$.
Proposition~\ref{prop:FanoStable} and Proposition~\ref{prop:MukaiConnectedness} give a complete description of the moduli space $M_{\sigma}^{\mathrm{st}}(\llambda_1)$.
When $\eta(\sigma)\in (A_2)_\C\cap \PP\subset H_{\mathrm{alg}}^*(\Ku(X),\Z)$, we can also describe $\ell_\sigma$.

\begin{Thm}\label{thm:FanoStable}
Let $X$ be a cubic fourfold such that $H_{\mathrm{alg}}^*(\Ku(X),\Z)$ has no $(-2)$-classes.
Then $M_{\sigma}^{\mathrm{st}}(\llambda_1)\cong F(X)$ is a fine moduli space, for any $\sigma\in\Stab(\Ku(X))$.
Moreover, if $\eta(\sigma)\in (A_2)_\C\cap \PP$, then $\ell_\sigma$ is a positive multiple of the divisor class $g=\llambda_1+2\llambda_2$ of the Pl\"ucker embedding\footnote{Here we use \cite[Proposition~7]{addington:two_conjectures} to identify $\NS(F(X))$ with $\llambda_1^\perp$; see also Proposition~\ref{prop:OrthogonalA2Lattice}.}.
\end{Thm}

\begin{proof}
By Proposition~\ref{prop:FanoStable}, the Fano variety of lines is a smooth integral projective variety of dimension $4=\llambda_1^2+2$ which is contained in $M_{\sigma}^{\mathrm{st}}(\llambda_1)$.
Hence, by Proposition~\ref{prop:MukaiConnectedness}, we have $F(X)=M_{\sigma}^{\mathrm{st}}(\llambda_1)$.
The universal family $\FF_\LL$ is also a universal family for objects in $\Ku(X)$.
Hence $F(X)$ is a fine moduli space of Bridgeland stable objects in $\Ku(X)$.
Finally, it is a straightforward computation, as in \cite[Lemma~9.2]{BM:projectivity}, to see that $\ell_\sigma$ is proportional to $\llambda_1+2\llambda_2$.
The fact that this is the Pl\"ucker polarization on $F(X)$ can be found, for example, in \cite[Equation (6)]{addington:two_conjectures}.
\end{proof}

\subsection{Proof of Theorem~\ref{thm:CategoricalTorelli}}

Clearly only one implication is non-trivial. Let us pick any stability condition $\sigma$ on $\Ku(X)$ such that $\eta(\sigma)\in (A_2)_\C\cap \PP$ and consider the fine moduli space $M_{\sigma}^{\mathrm{st}}(\llambda_1)$.
By Theorem~\ref{thm:FanoStable}, $M_{\sigma}^{\mathrm{st}}(\llambda_1)$ is isomorphic to $F(X)$ and carries a universal family $\mathcal{F}_{X,\mathcal{L}_X}$.

We recall from \cite[Proposition~3.12]{Huy:cubics} that the action of $(1)$ on the cohomology $H^*(\Ku(X), \Z)$ leaves $A_2^\perp$ invariant, whereas it cyclically permutes the roots
$\llambda_1, \llambda_2$ and $-\llambda_1-\llambda_2$ in $A_2$.

Now consider an equivalence
$\Phi\colon\Ku(X)\to\Ku(Y)$ as in Theorem~\ref{thm:CategoricalTorelli}, and let $\sigma':=\Phi(\sigma)$.
By the previous paragraph, it sends the distinguished sublattice
$A_2 \subset H^*_{\mathrm{alg}}(\Ku(X), \Z)$ to the corresponding sublattice for $Y$. Since the group generated by $(1)$ and $[1]$ acts transitively on the roots of $A_2$, we can replace $\Phi$ by a functor with $\Phi_*(\llambda_1) = \llambda_1$ and $\Phi_*(\llambda_2) = \llambda_2$.
Then it automatically induces a bijection between
$M_{\sigma}^{\mathrm{st}}(\llambda_1)$ and $M_{\sigma'}^{\mathrm{st}}(\llambda_1)$.
We need to show that this bijection is actually an isomorphism.

Consider the composition $\Psi\colon\Db(X)\to\Ku(X)\stackrel{\Phi}{\to}\Ku(Y)\hookrightarrow\Db(Y)$,
of $\Phi$ with the natural projection and inclusion.
If this is of Fourier-Mukai type, then it makes sense to consider the functor $\Psi\times\id\colon\Db(X\times M_{\sigma}^{\mathrm{st}}(\llambda_1))\to\Db(Y\times M_{\sigma}^{\mathrm{st}}(\llambda_1))$.
Then the object $(\Psi\times \id)(\FF_{X,\LL_X})$ provides a universal family on $Y\times M_{\sigma}^{\mathrm{st}}(\llambda_1)$.
Arguing as in \cite[Section~5.3]{BMMS:Cubics}, this gives a morphism $f\colon M_{\sigma'}^{\mathrm{st}}(\llambda_1)\to M_{\sigma}^{\mathrm{st}}(\llambda_1)$. Since it is induced by $\Phi$, this is an isomorphism.

If $\Psi$ is not of Fourier-Mukai type, we can  proceed as in \cite[Section~5.2]{BMMS:Cubics}.
While the functor $\Psi\times \id$ may not be well-defined, it still makes sense to define $(\Psi\times \id)(\FF_{X,\LL_X})$, and then argue as before.

By construction, we have  $f^*(\ell_{\sigma'})=\ell_\sigma$. By Theorem~\ref{thm:FanoStable}, $\ell_\sigma$ is the Pl\"ucker polarization; by the compatibility with $\Phi$ and the distinguished sublattice $A_2$, the same holds for $\ell_{\sigma'}$.
Hence we get an isomorphism $F(X)\to F(Y)$ which preserves the Pl\"ucker polarization.
By \cite[Proposition~4]{Charles:Torelli}, we get an isomorphism $X\cong Y$.

\subsection{The classical Torelli theorem}\label{subsec:appendix:ClassicalTorelli}

We are now ready to deduce the classical Torelli theorem for cubic fourfolds, by using the same argument as in \cite{HR16:Torelli_cubic_fourfolds}.

\begin{Thm}[Voisin]\label{thm:Torelli}
Two smooth complex cubic fourfolds $X$ and $Y$ over $\C$ are isomorphic if and only if there exists a Hodge isometry $H^4_{\mathrm{prim}}(X,\Z)\cong H^4_{\mathrm{prim}}(Y,\Z)$ between the primitive cohomologies.
\end{Thm}

This result was originally proved in \cite{Voisin:cubics}. Later Loojienga provided another proof in \cite{Looijenga:cubics} while describing the image of the period map. 
Charles \cite{Charles:Torelli} gave an elementary proof relying on the Torelli theorem for hyperk\"ahler manifolds \cite{Verbitsky:torelli}.

\begin{proof}
We briefly sketch the argument in \cite[Section~4.2]{HR16:Torelli_cubic_fourfolds}. 
Let $\phi\colon H^4_{\mathrm{prim}}(X,\Z)\xrightarrow{\simeq} H^4_{\mathrm{prim}}(Y,\Z)$
be a Hodge isometry.
By \cite[Proposition~3.2]{HR16:Torelli_cubic_fourfolds}, it induces a Hodge isometry $\phi'\colon H^*(\Ku(X),\Z)\xrightarrow{\simeq} H^*(\Ku(Y),\Z)$ that preserves the natural orientation.

A general deformation argument based on \cite{Huy:cubics} shows that $\phi'$ extends over a local deformation $\mathrm{Def}(X)\cong\mathrm{Def}(Y)$. 
The set $D\subset\mathrm{Def}(X)$ of points corresponding to cubic fourfolds $X'$ such that $\Ku(X')\cong \Db(S,\alpha)$, where $S$ is a smooth projective K3 surface and $\alpha$ is an element in the Brauer group $\mathrm{Br}(S)$, and $H_{\mathrm{alg}}^*(\Ku(X'),\Z)$ has no $(-2)$-classes is dense (see \cite[Lemma~3.22]{HMS:generic_K3s}). 
Moreover, as argued in \cite[Section~4.2]{HR16:Torelli_cubic_fourfolds}, for any $t\in D$ there is an orientation preserving Hodge isometry $\phi_t\colon  H^*(\Ku(X_t),\Z)\xrightarrow{\simeq} H^*(\Ku(Y_t),\Z)$ which commutes with the action on cohomology of the degree shift functor $(1)$ and which lifts to an equivalence $\Phi_t\colon\Ku(X_t)\to\Ku(Y_t)$. 
Now we can apply Theorem~\ref{thm:CategoricalTorelli} and get an isomorphism $X_t\cong Y_t$, for any $t\in D$. Since the moduli space of cubic fourfolds is separated, this yields $X\cong Y$.
\end{proof}



\end{document}